\documentclass[12pt]{amsart}   	
\usepackage[margin=1in]{geometry}                		
\usepackage{ifxetex,ifluatex}
\newif\ifxetexorluatex
\ifxetex
  \xetexorluatextrue
\else
  \ifluatex
    \xetexorluatextrue
  \else
    \xetexorluatexfalse
  \fi
\fi

\ifxetexorluatex
  \usepackage{fontspec}
\else
  \usepackage[T1]{fontenc}
  \usepackage[utf8]{inputenc}
  \usepackage{lmodern}
\fi

\usepackage{stmaryrd}
\usepackage{mathdots}
\usepackage{bbm}
\usepackage{amsmath,amssymb,amsfonts}
\usepackage{amsthm}
\usepackage{enumerate}
\usepackage{yfonts}
\usepackage{multicol}
\usepackage{ragged2e}
\usepackage{faktor}
\usepackage{pgfplots}
\usepackage{esint}

\usepackage{pinlabel}
\usepackage{graphicx}	
\usepackage{subcaption}
\usepackage{tikz}
\usepackage{tikz-cd}

\usepackage[shortlabels]{enumitem}
\usepackage{longtable}

\newtheorem{theorem}{Theorem}
\newtheorem*{theorem*}{Theorem}

\newtheorem{proposition}{Proposition}
\newtheorem*{proposition*}{Proposition}

\newtheorem{corollary}[proposition]{Corollary}
\newtheorem*{corollary*}{Corollary}

\newtheorem{lemma}[proposition]{Lemma}
\newtheorem*{lemma*}{Lemma}

\newtheorem{remark}[proposition]{Remark}
\newtheorem*{rmk*}{Remark}

\newtheorem*{claim*}{Claim}

\theoremstyle{definition}
\newtheorem{definition}[proposition]{Definition}
\newtheorem*{definition*}{Definition}

\numberwithin{equation}{section}
\numberwithin{theorem}{section}
\numberwithin{proposition}{section}
\numberwithin{figure}{section}

\usepackage{setspace}
\usepackage{hyperref}
\hypersetup{colorlinks,allcolors=blue}
\usepackage[capitalize]{cleveref}

\usepackage{verbatim}
\usepackage{tikz}
\usepackage{tikz-cd}
\usetikzlibrary{decorations.markings}
\usetikzlibrary{angles,calc,intersections,quotes,arrows,fit}
\usetikzlibrary{shapes.geometric}

\usepackage{faktor}
\usepackage[T1]{fontenc}

\newcommand{\diff}{\mathrm{d}}
\newcommand{\A}{\mathcal{A}}
\newcommand{\B}{\mathcal{B}}
\newcommand{\C}{\mathcal{C}}
\newcommand{\E}{\mathbb E}

\newcommand{\J}{\mathcal{J}}
\newcommand{\K}{\mathcal{K}}
\newcommand{\M}{\mathbb{M}}
\newcommand{\N}{\mathbb{N}}
\newcommand{\Q}{\mathcal{Q}}
\newcommand{\R}{\mathbb{R}}
\newcommand{\T}{\mathcal{T}}
\newcommand{\V}{\mathbb V}
\newcommand{\Z}{\mathbb{Z}}

\newcommand{\Nc}{\mathcal{N}}
\newcommand{\Pp}{\mathbb{P}}

\newcommand{\Ell}{\mathscr L}
\newcommand{\1}{\boldsymbol{1}}

\usepackage{xparse}
\NewDocumentCommand{\Hom}{mmg}{\ensuremath{\text{Hom}_{\IfNoValueTF{#3}{}{#3}}(#1,#2)}}
\NewDocumentCommand{\Tor}{mmmg}{\ensuremath{\text{Tor}_{#3}^{\IfNoValueTF{#4}{}{#4}}(#1,#2)}}
\NewDocumentCommand{\Ext}{mmmg}{\ensuremath{\text{Ext}^{#3}_{\IfNoValueTF{#4}{}{#4}}(#1,#2)}}

\usepackage[utf8]{inputenc}
\usepackage[T1]{fontenc}
\usepackage{mathrsfs}

\usepackage[backend=biber,style=numeric, doi = true]{biblatex}
\hypersetup{breaklinks=true}
\usepackage{xurl}
\title{Rigorous Derivation of the Wave Kinetic Equation for full $\beta$-FPUT System}
\author[Katja D. Vassilev]{Katja D. Vassilev}
\address{Department of Mathematics, The University of Chicago, Chicago, IL 60637}
\email{kdv@uchicago.edu}
\author[Boyang Wu]{Boyang Wu}
\address{Department of Mathematics, University of Michigan, Ann Arbor, MI 48109}
\email{boyangwu@umich.edu}
\usepackage{fancyhdr}
\pgfplotsset{compat=1.18}	

\begin{document}
\begin{abstract}
The Fermi--Pasta--Ulam--Tsingou (FPUT) system, describing the evolution of $N$ coupled harmonic oscillators, has been the subject of much attention since the 1950's when experiments which contradicted predictions of thermalization of the system. A full explanation of this behavior is still not fully known. Here, we rigorously derive the corresponding wave kinetic equation, which provides a precise evolution of the statistics for the FPUT system and demonstrates thermalization in an appropriate regime. In particular, we justify the kinetic equation for the 4-wave $\beta$-FPUT system in the kinetic limit $N \to \infty$ and $\beta \to 0$ for weakly nonlinear scaling laws $\beta \sim N^{-\gamma}$, reaching times up to $T_{\mathrm{kin}}^{2/3}$, where $T_{\mathrm{kin}}$ represents the kinetic (thermalization) timescale. 

While we use a typical diagrammatic expansion to derive the kinetic equations, few works have dealt with nonlinearities with non-resonant terms, which are not part of the kinetic equation, which is the major novelty of this work. The only other such work \cite{DIP25} made use of a normal form method to push the non-resonant terms to higher order nonlinearities. Here, we directly incorporate the non-resonant terms into the diagrammatic expansion and demonstrate corresponding gains. This method can be adapted to other 4-wave non-resonant nonlinearities.

\end{abstract}

\maketitle
{\small\tableofcontents}

\section{Introduction}
\label{section-kinetictheory}
\subsection{The \texorpdfstring{$\beta$}{beta}-FPUT model} 
The Fermi-Pasta-Ulam-Tsingou (FPUT) system models a system of $N$ identical masses connected by nonlinear springs with elastic force with Hamiltonian 
\begin{align}
H=\sum_{j=0}^{N-1}\left(\frac{p_j^2}{2m}+\frac{\kappa}{2}(q_{j+1}-q_j)^2+\frac{\alpha}{3}(q_{j+1}-q_j)^3+\frac{\beta}{4}(q_{j+1}-q_j)^4\right), \label{eq-Hamiltonian}
\end{align}
describing the displacement $q_j(t)$ and momentum $p_j(t)$ of each particle for spring constants $\kappa, \alpha, \beta$ and periodic boundary conditions. This system gained significant attention in the 1950's due to simulations at Los Alamos which attempted to exhibit the thermal equipartition, yielding shocking results \cite{EJS}. Their experiments showed that the system presented quasi-periodic behavior with energy essentially returning to the initial condition. 

The focus of this paper is to establish the validity of Fermi's initial prediction of thermalization by rigorously deriving statistical behavior for the $\beta$-FPUT system (set $\alpha = 0$) in the limit that the number of masses $N \to \infty$ and parameter $\beta \to 0$. In particular we derive the wave kinetic equation (WKE), referred to as the phonon Boltzmann equation for this system. This work serves as a follow-up to the preprint \cite{FPUTRD} and thesis \cite{FPUTthesis} by Wu which establish the same but for shorter timescales for the reduced $\beta$-FPUT system, which discounts all non-resonant interactions in the nonlinearity.

\subsection{Reformulation of FPUT system} 
Rewriting the system \eqref{eq-Hamiltonian} as an ODE for the $\beta$-FPUT system, the evolution of the system is given by
\begin{align}
\begin{cases}m\dot{q_j} = p_j \\
m\ddot{q}_j=\kappa\left[(q_{j+1}-q_j)-(q_j-q_{j-1})\right]+\beta\left[(q_{j+1}-q_j)^3-(q_j-q_{j-1})^3\right] \end{cases} \label{eq-ODE} 
\end{align}
where $j=0,1,...,N-1$ and we adopt the convention that $q_0=q_N$. For simplicity, we set $m = \kappa = 1$. The phonon Boltzmann equation describes the system in wave action variables, so we perform several transformations to the system to formulate it in this manner.
\subsubsection{Discrete Fourier Transform}
Denote $\Z_N:=N^{-1}\Z$. For $(q_j)_{j = 0, \ldots, N - 1}$ and $(p_j)_{j = 0, \ldots, N - 1}$, we define their discrete Fourier transform at frequency $k \in \Z_N \cap [0,1)$ as
\begin{align}
Q_k&=\frac{1}{\sqrt{N}}\sum_{j=0}^{N-1}q_je^{-i2\pi kj},\mbox{ }q_j=\frac{1}{\sqrt{N}}\sum_{k\in \Z_N \cap [0,1)}Q_ke^{i2\pi jk}.\label{eq-discrete-ft-q} \\
P_k&=\frac{1}{\sqrt{N}}\sum_{j=0}^{N-1}p_je^{-i2\pi kj},\mbox{ }p_j=\frac{1}{\sqrt{N}}\sum_{k\in \Z_N \cap [0,1)}P_ke^{i2\pi jk}.\label{eq-discrete-ft-p}
\end{align}
Note that
\begin{align}
Q_{k}=Q_{1-k}^*,\mbox{ }P_{k}=P_{1-k}^*, \mbox{ for }k\in \Z_N\cap[0,1).
\end{align}
We may then rewrite \eqref{eq-ODE} as
\begin{align}
\begin{cases}
\dot{Q}_k = P_k, \\
\left(\Ddot{Q}_{k}+\omega_{k}^2{Q}_{k}\right)=\beta\sum_{k_1+k_2+k_3=k}\hat{T}_{k,1,2,3}{Q}_{k_1}{Q}_{k_2}{Q}_{k_3},
\end{cases}\label{eq-1.7}
\end{align}
where the sum $\sum_{k_1+k_2+k_3=k}$ is taken over $\{k_1,k_2,k_3 \in \Z_N \cap [0,1):k_1+k_2+k_3=k \mbox{ mod }1\}$, and
\begin{align}
&\omega_{k}=\omega(k)=2\left|\sin\left(\pi k\right)\right|, \label{eq-omega}\\
&\hat{T}_{k,1,2,3}=16\sin\left(\pi(k_1 + k_2 + k_3)\right)\sin\left(\pi k_1\right)\sin\left(\pi k_2 \right)\sin\left(\pi k_3\right). 
\end{align}

\subsubsection{Normal Modes}
We introduce the normal modes $a_k(t)$:
\begin{align}
a_k= \left(\frac{\sqrt{2}}{2}\right)\left[(\omega_k)^{\frac{1}{2}}Q_k+i(\omega_k)^{-\frac{1}{2}}P_k\right],  \mbox{ for }k\in \Z_N \cap [0,1),
\end{align}
so that we may write
\begin{align*}
i\dot{a}_{k}=\omega_{k}{a}_{k}+\frac{\beta}{N}\sum_{k_1+k_2+k_3=k}\frac{-\hat{T}_{-1,2,3,4}}{4\sqrt{\omega_k\omega_{k_1}\omega_{k_2}\omega_{k_3}}}(a_{k_1}+a_{N-{k_1}}^*)(a_{k_2}+a_{N-{k_2}}^*)(a_{k_3}+a_{N-{k_3}}^*). 
\end{align*}
Expanding the above, we may write the following, which is the subject of the remainder of this paper:
\begin{align}
i\dot{a}_{k}=&\omega_{k}{a}_{k}+\frac{\beta}{N}\sum_{\substack{k_1,k_2,k_3 \\ k_i \in \Z_N \cap [0,1)}}\bigg[{T}_{+,+,+}a_{k_1}a_{k_1}a_{k_3}\delta(k-k_1-k_2-k_3)\notag\\
&+3{T}_{+,+,-}a_{k_1}a_{k_2}a_{k_3}^*\delta(k-k_1-k_2+k_3)+3{T}_{+,-,-}a_{k_1}a_{k_2}^*a_{k_3}^*\delta(k-k_1+k_2+k_3)\label{FPU}\tag{FPUT}\\
&+{T}_{-,-,-}a_{k_1}^*a_{k_2}^*a_{k_3}^*\delta(k+k_1+k_2+k_3)\bigg] \notag
\end{align}
where 
\begin{align}
{T}_{\zeta_1, \zeta_2, \zeta_3} = {T}_{\zeta_1, \zeta_2, \zeta_3}(k_1, k_2, k_3) = -\zeta_1\zeta_2\zeta_3\frac{1}{4}\iota(\zeta_1k_1 + \zeta_2 k_2 + \zeta_3 k_3)\sqrt{\omega_k}\prod_{i = 1}^3 \left[\iota(k_i)\sqrt{\omega_{k_i}}\right],
\end{align}
where we define $\iota(x):=\text{sgn} \sin (\pi x)$ a $2$-periodic function on $\mathbb{R}$. Note that we separate the nonlinear terms of \eqref{FPU} into resonant and non-resonant terms depending on if it is possible for signs $\zeta_1, \zeta_2, \zeta_3 \in \{\pm\}$ there are $k_1, k_2, k_3 \in \Z_N \cap (0,1)$ which satisfy: 
\begin{equation} \label{eq-resonance}
\begin{cases}
k - \zeta_1 k_1 - \zeta_2 k_2 - \zeta_3 k_3 = 0, \\
\omega_k - \zeta_1 \omega_{k_1} - \zeta_2 \omega_{k_2} - \zeta_3\omega_{k_3} = 0. \end{cases}
\end{equation}
Due to the subadditivity of the dispersion relation $\omega,$ see for instance \cite{bustamante2019exact}, for $k_1, k_2 \in \Z_N \cap (0,1),$ we must have \[\omega(k_1 + k_2) < \omega(k_1) + \omega(k_2).\] Therefore, the only resonant term is when $\{\zeta_1, \zeta_2, \zeta_3\} = \{+,+,-\}$. 

\subsection{Statement of the main result} 
Here, we seek an statistical description of the evolution of the system \eqref{FPU} when averaged over \textit{well-prepared} initial data. In this case, the initial data takes the form
\begin{equation}\tag{DAT}\label{DAT}
a_k(0) = \sqrt{n_{\mathrm{in}}(k)}\eta_k(\varrho)
\end{equation}
where $n_{\mathrm{in}} \in C^\infty(\mathbb T \to (0,\infty))$ is periodic, and $\{\eta_k(\varrho)\}$ are i.i.d. mean-zero complex random variables with unit variance, which are assumed to be either standard complex Gaussian or uniformly distributed on the unit circle.

In particular, we look at the evolution of all second order correlations. Under reasonable physical assumptions on the lattice, the only nontrivial a priori correlations are $\E|a_k(t)|^2$ and $\E a_ka_{1-k}$, see for instance \cite{PhononBoltzmann, FPUTthesis}. When we take the kinetic limit, corresponding to taking $N \to \infty$ and $\beta \to 0$, it is expected that $\E a_ka_{1-k}$ should vanish, while $\E|a_k(t)|^2$ should evolve according to the following wave kinetic equation: 
\begin{equation}\tag{WKE} \label{WKE}
\begin{cases}
\partial_t n(t,k) = \K(n(t))(k), \\
n(0,k) = n_{\mathrm{in}}(k),
\end{cases}
\end{equation}
where the collision operator $\K$ is given by 
\begin{align}\label{KIN}\tag{KIN}
\K(\phi)(\xi) = \int_{\begin{subarray}{b}{(\xi_1, \xi_2, \xi_3 ) \in \mathbb{T}^{3}} \\ {\xi_1 + \xi_2 - \xi_3 = \xi \text{(mod 1)}} \end{subarray} } &\boldsymbol{\delta}(\omega_1 + \omega_2 - \omega_3 - \omega) \hspace{.2cm} \left|3T_{+,+,-}(\xi_1, \xi_2, \xi_3)\right|^2\\
&\phi\phi_{1} \phi_{2} \phi_{3} \left(\frac{1}{\phi} - \frac{1}{\phi_{1}} - \frac{1}{\phi_{2}} + \frac{1}{\phi_{3}} \right) \diff \xi_1 \diff \xi_2 \diff \xi_3, \notag
\end{align}
where we denote, for $i = 1,2,3$,
\begin{align*}
\phi = \phi(\xi,t), \hspace{.5cm} \phi_i = \phi(\xi_i,t), \hspace{.5cm} 
\omega = 2\sin\left(\pi \xi\right), \hspace{.5cm} \omega_i = 2\sin\left(\pi \xi_i\right).
\end{align*}
Note that the collision operator \eqref{KIN} does not contain any terms from the non-resonant interactions of \eqref{FPU} due to \eqref{eq-resonance}. Note that given smooth initial data $n_{\text {in}}(k)$, the equation \eqref{WKE} has a unique local solution $n=n(t,k)$ on some short time interval depending on $n_{\text {in}}$, as established in \cite{FPUTWP}. The timescale on which this behavior is expected to hold is the \textit{kinetic timescale,} 
\begin{equation}
T_{\mathrm{kin}} = \frac{1}{4\pi\beta^2}.
\end{equation}

Here, we rigorously establish that these correlations evolve in this manner. Note that in order to rigorously justify these assertions, we need to clarify a \textit{scaling law}, dictating how we are taking $N \to \infty$ and $\beta \to 0$ in the kinetic limit. Such scaling laws take the form $\beta = N^{-\gamma},$ for $\gamma \in (0,1)$. See \cite{WKE2023} for a heuristic explanation of why these are the scaling laws we consider.

\begin{theorem} \label{main}
Fix $\gamma \in(0,1)$. Fix a smooth function $n_{\mathrm {in }}\in C^{\infty}(\mathbb{T}\rightarrow [0,\infty))$, $A \geq 40$, and $\epsilon \ll 1$. Consider the equation \eqref{FPU} with random initial data \eqref{DAT}, and assume $\beta=N^{-\gamma}$ so that $T_{\mathrm{kin}}\sim N^{2\gamma}$. Fix $T=N^{-\epsilon}\min(N,N^{\frac{4}{3}\gamma})$.
Then for ${N^{0+} \leq t \leq T}$,
\[\mathbb E\left|a(t,k)\right|^2 = n_{\mathrm{in}}(k)  + \frac{t}{T_{\mathrm{kin}}} \mathcal K(n_{\mathrm{in}})(k) + o_{\ell_k^\infty}\left(\frac{t}{T_{\mathrm{kin}}}\right),\] 
and \[\E a_k(t) a_{1-k}(t) = o_{\ell_k^\infty}\left(\frac{t}{T_{\mathrm{kin}}}\right).\]
Note that in the above, $o_{\ell_k^\infty}\left(\frac{t}{T_{\mathrm{kin}}}\right)$ is a quantity bounded in $\ell_k^\infty$ by $L^{-\theta}\frac{t}{T_{\mathrm{kin}}}$ for some $\theta > 0$.
\end{theorem}

\subsection{Background literature}
\textit{Wave kinetic theory}, or \textit{wave turbulence theory}, is the formal study of the statistical behavior of non-linear wave systems, which can be viewed as the wave-analog of Boltzmann's kinetic theory for particles. This parallel was first drawn in 1929 when Peierls proposed the phonon Boltzmann equation, the equation we derive here, to describe systems of anharmonic crystals \cite{peierls1929}. Soon after, Nordheim (and later Uehling and Uhlenbeck) suggested the quantum Boltzmann equation to describe systems of quantum interacting gases \cite{Nordheim, Uehling1933}. The theory was then further developed in the 1960's in works on plasma physics \cite{vedenov1967theory} and water waves \cite{WaterWaves62,WaterWaves63}. The main theoretical consequences during that period were the discovery of so-called \textit{Kolmogorov-Zakharov spectra} solutions to kinetic equations discovered by Zakharov in 1965 \cite{zakharov1965weak}, with significant work since focused on demonstrating and understanding these solutions (for a summary of results, see \cite{zakharov2012kolmogorov}). 

\subsubsection{FPUT recurrence}
Tracing back to the 1920's, Fermi attempted to prove ergodicity of general Hamiltonian systems \cite{Fermi1, Fermi2}. This led to the seminal experiments of Fermi, Pasta, Ulam, and Tsingou on MANIAC computers which were intended to verify Fermi's claim of ergodicity, that energy would eventually spread equally to all modes \cite{EJS}. However, these simulations exhibited something remarkable:  when they initially excited one mode, eventually energy would almost fully return to this mode with losses of less than 1\%. In fact, later experiments of Tuck and Menzel showed evidence of a so-called ``super-period" for the energy \cite{Tuck1972}. These results became known as the ``FPUT paradox" and since there has been significant work devoted to explaining the phenomenon exhibited in these experiments. 

One of the first explanations was to draw a relationship between the FPUT system and integrable systems. The landmark numerical work of Zabusky and Kruscal in the 1960s suggested that solutions to the FPUT are approximated in the continuum limit $N \to \infty$ by the KdV (Korteweg-de Vries) system \cite{zabusky1965interaction}, with this relationship rigorously justified \cite{schneider2000counter, bambusi2006metastability}. Similarly, the FPUT system can also be approximated in regimes by the Toda lattice, another integrable system, introduced in the 1960's \cite{Toda1967} and studied numerically in for instance \cite{ponno2011two}. The method of normal forms was subsequently used to study the dynamics of the Toda lattice, and consequently FPUT \cite{HKToda, HKFPU}. Using the method of Lax pairs, one can also make more precise the relationship between FPUT, KdV, and the Toda lattice \cite{BambusiTodaKdv1, BambusiTodaKdv2}.

Around the same time, Izrailev and Chirikov proposed that the behavior of FPUT could be explained via KAM (Kolmogorov-Arnold-Moser) theory \cite{israiljev1965statistical}. In particular, that FPUT would exhibit quasi-periodic behavior for sufficiently small initial energy relative to the number of oscillators $N$ (in the initial experiments $N$ was taken at most 64). The most accurate threshold for such behavior is not agreed upon numerically  \cite{zaslavskiui1972stochastic, bocchieri1970anharmonic, nishida1971note, shepelyansky1997low, carati1999specific,carati2004definition,berchialla2004localization, ponno2005fermi, carati2007averaging,lvov2018double}. Under certain circumstances, quasi-periodicity has been shown for small energy using normal form techniques \cite{rink2006proof}.

More recently, work has focused on the kinetic equation, whose validity implies that equipartition of energy must hold for weak nonlinearity. For instance numerical work in the early 2000's showed that the kinetic equation is accurate in predicting the exponents for thermalization. Subsequent numerical work have begun to establish which types of wave interactions are responsible for energy transfer, see for instance \cite{FPU, pistone2018universal, bustamante2019exact, dematteis2020coexistence, de2022anomalous, ganapa2023quasiperiodicity, pezzi2025multi}. For more theoretical work on the subject, see \cite{lukkarinen2007kinetic, Lukkarinen2016} which look at anomalous diffusion for FPUT, and \cite{FPUTWP} which establishes well-posedness of the kinetic equation for FPUT. See \cite{Campbell2005, gallavotti2007fermi, Bambusi2015, FPUTreview} for a summary of numerical and theoretical work concerning FPUT.  

\subsubsection{Wave Turbulence theory}
In the past twenty years, there has been significant work to justify wave turbulence by rigorously deriving wave kinetic equations via the convergence of a Feynman diagram expansion. Demonstrating this convergence was first done in the linear setting by Spohn \cite{Spohn08} and later extended in  \cite{ErdosYau, ErdosSalmhoferYau}. The first results in the nonlinear setting focused on the cubic NLS in dimension $d \geq 2$, initially in the deterministic setting \cite{2DNLS, BGHSCR, BGHSdyn}. In the kinetic setting, a series of works were then able to justify the kinetic equation first for short times \cite{BGHSonset, CG1, CG2, 2019} and then for long times \cite{WKE, WKE2023, WKElong} in $d \geq 3$. In the latter of these works, Deng-Hani establish the validity of the kinetic equation for arbitrarily long times for the entire range of relevant scaling laws. Remarkably, the combinatorial techniques introduced in these works were adapted by these authors along with Ma to establish a long-time derivation of the Boltzmann equation for particle systems \cite{Boltzmannlong, Hilbert6}. Since the work \cite{WKE}, several works have derived kinetic or effective equations for a many 3 and 4-wave systems, see for instance \cite{QuadraticInhomo, KdVInhomogeneous, WickNLS, MaZK, StochasticNLS, KleinGordon}.

The aforementioned results focus on systems in dimension $d \geq 2$ as \cite{ODW} showed that the combinatorial tools used to bound Feynman diagrams do not reach the kinetic timescale with \cite{DIP25} improving the timescale to $T_{\mathrm{kin}}^{2/3}$ and showing an obstruction past this point. In addition to the combinatorial issues in 1D, the FPUT (and water waves) system (1) requires a \textit{time-dependent phase renormalization} to remove divergent interactions and (2) contains non-resonant terms in the nonlinearity. This \textit{time-dependent phase renormalization} was independently introduced in \cite{DIP25} and \cite{FPUTRD}, where \cite{DIP25} required energy estimates to justify the renormalization due to the presence of derivatives in the nonlinearity. To address the non-resonant terms, while \cite{DIP25} employs a normal form to bound the solution, here we are able to use a more direct approach to derive the kinetic equation. However, a normal form transformation can be performed up to the kinetic timescale \cite{VCFPUT}. 

\subsection{Overview of the proof}
The proof follows the same strategy of \cite{FPUTRD, FPUTthesis}, which dealt with the same system \eqref{FPU} with all non-resonant terms removed. Here, we show that the same equation can be derived even with the presence of the non-resonant terms and that all additional non-trivial correlations do vanish in the kinetic limit. As is now typical in derivations of kientic equations, we rely on a diagrammatic expansion of the solution $a_k(t)$:
\begin{equation*}\label{eq-expansion}
a_k = a_k^{(0)} + a_k^{(1)} + a_k^{(2)} + \ldots \ldots+ a_k^{(M)} + \mathcal R_k, 
\end{equation*}
where $M$ is chosen to be a sufficiently large threshold, $a_k^{(j)}$'s represent the $j$-th iterates, and $\mathcal R_k$ denotes the remainder term.

\begin{enumerate}

    \item Each of the iterates $a_k^{(n)}$ we write as the sum over \textit{ternary trees} $\mathcal T$, defined in Section \ref{sec-tt}, effectively a Feynman diagram expansion representing the interaction history of waves. We keep track of the nonlinear wave interactions by \textit{decorations} of these trees with wave numbers. Since we are dealing with second order correlations, we also introduce \textit{enhanced couples} $\Q$, two ternary trees whose leaves have been paired. 
    
    For each couple $\Q$, we show the convergence of the corresponding expression $\K_\Q$, which essentially consists of a sum over all decorations along with an oscillatory integral for each decoration. The broad strategy is to bound $\K_\Q$ by using the oscillatory integral to reduce the estimate to a counting problem on decorations. See Section \ref{sec-int} for an estimate of the oscillatory portion. Due to the nature of the nonlinearity, our ternary trees and enhanced couples contain both resonant and non-resonant interactions. This is one of the main differences of this work with \cite{DIP25}, where non-resonant interactions are pushed to higher order via normal form, so that the tree expansion is no longer ternary and any ternary node corresponds to a resonant interaction. 
    
    \item For the resonant term of the nonlinearity associated to $a_{k_1}a_{k_2}a_{k_3}^*$, it has long been known that for exact resonances where there is exact pairing of wave numbers, i.e. $\{k, k_3\} = \{k_1, k_2\}$ leads to divergence in the iterates due to the growth of $\sum_{\ell \in \Z_N \cap [0,1)}T_{+,+,-} (\ell, k, \ell) |a_\ell(t)|^2$, see for instance \cite{TextbookWT}. Without the presence of the coefficient $T_{+,+,-}(\ell, k, \ell),$ such terms can be removed via Wick-ordering using the conservation of mass, see \cite{2019}. Since the coefficients depend on the nonlinearity, we instead must perform a phase renormalization to remove these interactions. This was formally done in \cite{TextbookWT, SW, chibbaro20184} and rigorously done recently in \cite{DIP25, FPUTRD}. This phase renormalization is deterministic and satisfies: \[\partial_t \tilde \omega_k(t) = \left( \frac{\beta T}{N}\right) \sum\limits_{\ell \in \Z_N \cap [0,1)} 2T_{+,+,-}(\ell, k, \ell)\sum\limits_{\Q} \K_\Q(t,t,\ell).\] 
    This phase renormalization effectively allows us to discount couples with such divergent terms, which is why we use enhanced couples. 
    
    \item For non-resonant terms of the nonlinearity, in particular those associated to $a_{k_1}a_{k_2}^*a_{k_3}^*$, we may have pairing $\{k_2, k_3\} = \{-k, k_1\}$ which in theory could lead to a similar divergence to what is described above, when bounding the corresponding couple $\K_\Q$ via counting estimates. However, we show in Section \ref{subsec-loops} that such terms, which we refer to as \textit{self-loops}, lead to gains in the oscillatory portion of $\K_\Q$. This is one of the novelties in this work, which applies in lieu of a normal form in this setting.
    
    \item Another apparent divergence in the estimate for $\K_\Q$ is the presence of so-called \textit{irregular chains}, first identified in \cite{WKE}. If we were simply to use counting estimates, we would get very large bounds for some $\K_\Q$ which could diverge as early as $T_{\mathrm{kin}}^{1/2}$. However, as in \cite{WKE}, we may exploit cancellations between similar such \textit{congruent couples} and remove the irregular chains via \textit{splicing}, as described in Section \ref{subsec-splicing}. 
    
    \item To reduce the analysis of each $\K_\Q$ to a counting problem, we turn each couple into a \textit{molecule}, first introduced in \cite{WKE}. Each molecule is a directed graph of \textit{atoms} and \textit{bonds} expressing all relevant relationships in the corresponding couple for the purposes of counting decorations, Parent-Child (PC) relationships and Leaf Pair (LP) relationships. To count the number of decorations of a couple (and its corresponding molecule), we perform a \textit{counting algorithm} on the couple in Section \ref{section-algorithm}. This first involves a \textit{cutting algorithm}, first done in \cite{WKE2023} and adapted to the 1D case in \cite{DIP25}, which transforms the molecule into many smaller, easy to count parts. Our counting algorithm is a modified version of that in \cite{DIP25}, with corresponding counting estimates proved in Section \ref{section-counting}.
    
    \item For the remainder $\mathcal R_k$, we follow the same approach as in \cite{WKE,WKE2023,ODW} and use a contraction mapping argument, where we can reduce the analysis to showing that there is a linear map $\Ell$ for which $1 - \Ell$ is invertible in a suitable space. We further express $(1 - \mathscr L)^{-1}$ via Neumann series as:\[(1 - \mathscr{L})^{-1} = (1 - \mathscr L^{M})^{-1}\,\bigl(1 + \mathscr{L} + \cdots + \mathscr{L}^{M-1}\bigr),\] so that in fact we need to bound powers $\Ell^n$. These we may define in terms of \textit{flower trees} and \textit{flower couples}, which have a similar structure to our ternary trees and enhanced couples. This allows us to use the same analysis we used to bound $\K_\Q$ for $\Ell^n$. See \ref{section-remainder}. 
    \item We prove the main theorem in Section \ref{section-mainthm}. For enhanced couples of order larger than 2 or terms involving the remainder, we may use our estimates on $\K_\Q$ and $\mathcal R_k$. For couples of order 0, 1, and 2, we calculate them exactly and show that for $\E |a_k(t)|^2,$ they yield a first order expansion of \eqref{WKE}, and for $\E a_k(t)a_{1-k}(t),$ they vanish in the limit. Note that we rely on the normal form estimates proven in \cite{VCFPUT} to bound couples of order two which contain non-resonant contributions. 
\end{enumerate}

\noindent \textit{Acknowledgements.} The authors would like to thank their advisor Zaher Hani for the numerous insightful conversations and discussions through the preparation of this paper. This research was supported by the Simons Collaboration Grant on Wave Turbulence and NSF grant DMS-1936640.

\section{Preparations and Main Tools}
\label{section-preparations}
\subsection{Preliminary Reductions}
For a solution $a(t) = (a_k(t))$ to \eqref{FPU}, set 
\begin{equation}\label{eq-bk}
b_k(t) = e^{it(T\omega_k + \tilde{\omega}_k)}a_k(Tt)
\end{equation}
where $\tilde{w}_k$ is a deterministic phase renormalization we define in the following sections meant to cancel divergent terms. The rest of the paper will focus on the analysis of $b_k(t)$, where $t \in [0,1]$ and $k \in \Z_N \cap [0,1)$. We may write: 
\begin{equation}
\begin{cases}
\partial_t b_k = \mathcal C^+(b,b,b) + i\tilde{\omega}_k'(t) b_k(t). \\
b_k(0) = \sqrt{n_{\mathrm{in}}(k)} \eta_k(\varrho),
\end{cases}
\end{equation}
where 
\begin{equation}
\mathcal C^+(b,b,b)_k(t) = \sum_{\substack{\zeta_1, \zeta_2, \zeta_3 \in \{\pm\}}} \mathcal C_{\zeta_1, \zeta_2, \zeta_3}^+(b,b,b)_k(t),
\end{equation}
and 
\begin{align}
\mathcal C_{\zeta_1, \zeta_2, \zeta_3}^\zeta(u,v,w)_k(t) := -i \frac{\beta T}{N} &\sum_{k_1, k_2, k_3 \in \Z_N \cap [0,1)]} \delta(\zeta k - \zeta_1 k_1 -\zeta_2 k_2 - \zeta_3 k_3) \zeta T_{\zeta_1, \zeta_2, \zeta_3} \\
& \hspace{.25cm}\times \left[u_{k_1}^{\zeta_1} v_{k_2}^{\zeta_2} w_{k_3}^{\zeta_3}e^{i\zeta\Gamma^{(\zeta, \zeta_1, \zeta_2, \zeta_3)}_k(t)} \right],
\end{align}
with
\begin{align}
\Gamma^{(\zeta, \zeta_1, \zeta_2, \zeta_3)}_k(t) := Tt\Omega^\zeta_{\zeta_1, \zeta_2, \zeta_3}(k, k_1, k_2, k_3) + \tilde{\Omega}^\zeta_{\zeta_1, \zeta_2, \zeta_3}(k, k_1, k_2, k_3)(t),
\end{align}
defining $\Omega^\zeta_{\zeta_1, \zeta_2, \zeta_3}$ and $\widetilde \Omega^\zeta_{\zeta_1, \zeta_2, \zeta_3}(t)$ as
\begin{align}
\Omega^\zeta_{\zeta_1, \zeta_2, \zeta_3}(k, k_1, k_2, k_3) &:= \zeta\omega_k - \zeta_1 \omega_{k_1} - \zeta_2 \omega_{k_2} - \zeta_3\omega_{k_3}, \\
\widetilde\Omega^\zeta_{\zeta_1, \zeta_2, \zeta_3}(k, k_1, k_2, k_3)(t) &:= \zeta \widetilde\omega_k(t) - \zeta_1 \widetilde\omega_{k_1}(t) - \zeta_2 \widetilde\omega_{k_2}(t) - \zeta_3\widetilde\omega_{k_3}(t).
\end{align}

\subsection{Parameters, notations, and norms} \label{subsec-notation}
Throughout, let $C$ denote a large constant depending only on $(n_{\mathrm{in}}, \epsilon)$, which may vary from line to line. Let $\theta $ denote any sufficiently small constant, depending on $\epsilon$, and $\delta$ a fixed small constant depending only on $\epsilon$. Define $M:= \frac{10^6}{\epsilon}$, a constant which we will use in our Ansatz in Section \ref{subsec-ansatz}.

For $t \in [0,1]$ and any function $F = F(t)$ defined on $[0,1]$, denote the Duhamel operator by
\begin{equation}\label{eq-duhamel}
\mathcal I F(t) = \int_0^t F(s) \diff s.
\end{equation}
Define the time Fourier transform (the use of $\widehat{\cdot}$ depends on the context) by
\begin{equation*}
\widehat{u}(\tau) = \int_\R u(t) e^{-2\pi i \tau t} \diff t, \hspace{1cm} u(t) = \int_\R \widehat{u}(\tau) e^{2\pi i \tau t} \diff \tau. 
\end{equation*}
For a function $a(t) = (a_k(t))_{k \in \Z_N \cap [0,1)}$, define the $Z$ norm to be
\begin{equation}
||a||_Z^2 = \sup_{0 \leq t \leq 1} N^{-1} \sum_{k \in \Z_N} |a_k(t)|^2.
\end{equation}

\subsection{Trees and Couples} \label{sec-tt}
\begin{definition} (Trees) \label{def:Trees}
\begin{enumerate}[label=(\roman*)]
    \item A \textit{ternary tree} $\T$ is a rooted tree where each branching node has precisely three children.  Let $\mathcal{N}$ denote the set of branching nodes and define $n:= |\mathcal N|$ to be the \textit{order} of the tree. Denote the set of leaves by $\mathcal{L}$, so $|\mathcal L| = 2n + 1$. Denote the \textit{root} node $\mathfrak{r}$. Let $\mathcal{T}=\bullet$ denote a \emph{trivial} tree with order 0. When relevant, we denote the children subtrees of $\mathfrak r$ by $\mathcal T_1, \mathcal T_2, \mathcal T_3$ from left to right.
    \item We assign to each ternary tree $\mathcal{T}$ a sign $+\text{ or }-$, sometimes added as a superscript. For any node $\mathfrak{n}\in \mathcal{N}$, let its children be $\mathfrak{n}_1,\mathfrak{n}_2,\mathfrak{n}_3$ from left to right. We assign each node the sign $\zeta_{\mathfrak{n}}\in\{\pm1\}$. Define $\zeta(\mathcal T) = \prod_{\mathfrak n \in \mathcal N} (-i\zeta_{\mathfrak n})$. We call $\mathfrak n \in \mathcal N$ a $(m_{\zeta_n}, m_{-\zeta_n})$ node if $m_{\zeta_n}$ children of root $\mathfrak r$ have sign $\zeta_n$ and $m_{-\zeta_n}$ have sign $-\zeta_n$ and denote this the \textit{type} of $\mathfrak n$. Note $m_{\zeta_n} + m_{-\zeta_n}$ = 3. 
    \item An \textit{enhanced tree} is a ternary tree $\mathcal{T}$ equipped with a designated set $\mathcal N_D \subset \mathcal N$ of $(2,1)$ \textit{degenerate} nodes. For $\mathfrak n \in \mathcal N_D$ with children $\{\mathfrak n_1, \mathfrak n_2, \mathfrak n_3\},$ we specify a choice of a distinguished child $\mathfrak n_c$, where $\zeta_{\mathfrak n_c} = \zeta_{\mathfrak n}$. See Figure \ref{fig:ternary-tree-diamond}. 
    \item A \textit{decoration} $\mathscr D$ of an enhanced tree $\T$ is a set of vectors $\{(k_{\mathfrak n})_{\mathfrak n \in \T}\}$ such that $k_\mathfrak{n}\in \Z_N \cap [0,1)$. Further, if $\mathfrak n$ is a branching node, 
    \begin{equation*}
    \zeta_{\mathfrak n} k_{\mathfrak n} = \zeta_{\mathfrak n_1} k_{\mathfrak n_1} + \zeta_{\mathfrak n_2} k_{\mathfrak n_2} +\zeta_{\mathfrak n_3} k_{\mathfrak n_3},
    \end{equation*}
    where $\mathfrak n_1, \mathfrak n_2, \mathfrak n_3$ are the three children of $\mathfrak n$ labeled from left to right. In addition, whenever $\mathfrak{n}\in \mathcal{N}_{D}$ is a \emph{degenerate} branching node with children $\{\mathfrak{n}_c, \mathfrak{n}_2, \mathfrak n_3\}$ (not necessarily in order), we impose the \emph{enhanced} constraints:
    \[ \quad k_{\mathfrak{n}} = k_{\mathfrak{n}_c} \quad \text{and} \quad k_{\mathfrak{n}_3} = k_{\mathfrak{n}_2}. \]  We say $\mathscr D$ is a $k$-decoration if $k_{\mathfrak r} = k$. Given a decoration $\mathscr D$, for each branching node $\mathfrak n \in \N$, we define the resonance factors $\Omega_{\mathfrak n}, \widetilde{\Omega}_{\mathfrak n}$, and $\Gamma_{\mathfrak n}$ as 
    \begin{align*}
        \Omega_{\mathfrak n} &:= \zeta_{\mathfrak n}\omega_{k_{\mathfrak n}} - \zeta_{\mathfrak n_1}\omega_{k_{\mathfrak n_1}} - \zeta_{\mathfrak n_2} \omega_{k_{\mathfrak n_2}} - \zeta_{\mathfrak n_3}\omega_{k_{\mathfrak n_3}} \\
        \widetilde\Omega_{\mathfrak n}(t) &:= \zeta_{\mathfrak n}\widetilde\omega_{k_{\mathfrak n}}(t) - \zeta_{\mathfrak n_1}\widetilde\omega_{k_{\mathfrak n_1}}(t) - \zeta_{\mathfrak n_2}\widetilde\omega_{k_{\mathfrak n_2}}(t) - \zeta_{\mathfrak n_3}\widetilde\omega_{k_{\mathfrak n_3}}(t) \\
        \Gamma_{\mathfrak n}(t) &:= Tt \Omega_\mathfrak n + \widetilde\Omega_\mathfrak n(t)
    \end{align*}
    We also define 
    \begin{align}
    T_\mathfrak n(k_{\mathfrak n_1}, k_{\mathfrak n_2}, k_{\mathfrak n_3}) &:= \zeta_{\mathfrak n} T_{\zeta_{\mathfrak n_1}, \zeta_{\mathfrak n_2}, \zeta_{\mathfrak n_3}} (k_{\mathfrak n_1}, k_{\mathfrak n_2}, k_{\mathfrak n_3})
    \end{align}
    so that
    \begin{equation} \label{eq-eps}
    \epsilon_{k_{\mathfrak{n}_1}k_{\mathfrak{n}_2}k_{\mathfrak{n}_3}} := \begin{cases}
    -T_\mathfrak n(k_{\mathfrak n_1}, k_{\mathfrak n_2}, k_{\mathfrak n_3}) & \text{if } k_{\mathfrak{n}_1}=k_{\mathfrak{n}_2}=k_{\mathfrak{n}_3}; \\
    +T_\mathfrak n (k_{\mathfrak n_1}, k_{\mathfrak n_2}, k_{\mathfrak n_3}) & \text{otherwise},
    \end{cases}
    \end{equation} 
    and
    \begin{equation}
    \epsilon_{\mathscr D} := \prod_{\mathfrak n \in \mathcal N} \epsilon_{k_{\mathfrak n_1}k_{\mathfrak n_2}k_{\mathfrak n_3}}. 
    \end{equation}

    \noindent Lastly, for any $k$-decoration of $\T$, we may choose $d_\mathfrak n \in \{0,1\}$ and define $q_{\mathfrak{n}}$ for each $\mathfrak{n}\in \mathcal{N}$ inductively by:
    \begin{align} \label{eq-def-qn}
    q_{\mathfrak{n}} &= 0 \text{ if } \mathfrak{n}\in \mathcal{L}\text{; } q_{\mathfrak{n}} = d_{\mathfrak{l}}q_{\mathfrak{l}}- d_{\mathfrak{n}_2}q_{\mathfrak{n}_2}+d_{\mathfrak{n}_3}q_{\mathfrak{n}_3}+\Omega_{\mathfrak{n}} \text{ if }\mathfrak{n}\in \mathcal{N}.
    \end{align}
\end{enumerate}

\begin{figure}[ht]
\centering
\begin{tikzpicture}[scale=.8,
  level distance=1.5cm,
  level 1/.style={sibling distance=4.cm},
  level 2/.style={sibling distance=1.25cm},
  level 3/.style={sibling distance=0.8cm}]
  \tikzset{
    solid node/.style={circle,draw,fill=black,inner sep=1.5pt},
    rectangle node/.style={
      shape=diamond, draw=black, fill=white, minimum size=6pt, inner sep=0pt
    }
  }

  \node[solid node,label=above:{\tiny $(\mathfrak r,+)$}] (root) {}
    child {
      node[solid node,label=left:{\tiny $(\mathfrak n_1,+)$}] {}
      child { node[solid node,label=below:{\tiny $(\mathfrak m_1,+)$}] {} }
      child { node[solid node,label=below:{\tiny $(\mathfrak m_2,+)$}] {} }
      child { node[solid node,label=below:{\tiny $(\mathfrak m_3,+)$}] {} }
    }
    child {
      node[rectangle node,label=right:{\tiny $(\mathfrak n_2,-)$}] {}
      child { node[solid node,label=below:{\tiny $( \mathfrak{l}_2,+)$}] {} }
      child { node[solid node,label=below:{\tiny $(\mathfrak l_c,-)$}] {} }
      child { node[solid node,label=below:{\tiny $( \mathfrak{l}_3,-)$}] {} }
    }
    child {
      node[solid node,label=right:{\tiny $(\mathfrak n_3,-)$}] {}
      child { node[solid node,label=below:{\tiny $(\mathfrak p_1,+)$}] {} }
      child { node[solid node,label=below:{\tiny $(\mathfrak p_2,+)$}] {} }
      child { node[solid node,label=below:{\tiny $(\mathfrak p_3,-)$}] {} }
    };
\end{tikzpicture}
\caption{A (1,2) enhanced ternary tree with root $(\mathfrak r, +)$ and three branching nodes $\mathfrak n_i$, each with children labeled with their corresponding signs. Node $\mathfrak n_2$ is chosen to be degenerate so a decoration must have $k_{\mathfrak n_2} = k_{\mathfrak l_c}$ and $k_{\mathfrak l_2} = k_{\mathfrak l_c^*}$.}
\label{fig:ternary-tree-diamond}
\end{figure}
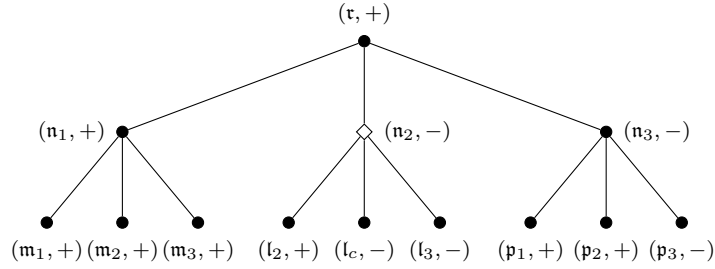 
\end{definition}

\begin{definition}\label{def-treeexp}
    For an enhanced ternary tree $\mathcal T$ of order $n$, define 
    \begin{equation}\label{eq-JT}
        (\J_\T)_k(t) := \left(\frac{\beta T}{N}\right)^n \zeta(\T) \sum_{\mathscr{D}}\epsilon_{\mathscr D} \mathcal{A}_{\mathcal{T}}(t,\Omega[\mathcal N], \widetilde{\Omega}[\mathcal N])\prod_{\mathfrak{l}\in \mathcal{L}}\sqrt{n_{\mathrm{in}}(k_{\mathfrak{l}})}\mathcal{B}_{\mathcal{T}}(\eta_{\mathcal{T}}(\varrho)),
    \end{equation}
    where 
    \begin{align}
    \mathcal{A}_{\mathcal{T}}(t,\Omega[\Nc], \widetilde{\Omega}[\Nc]) &= \int_{\mathcal D} \prod_{\mathfrak n \in \mathcal N} e^{i\zeta_{\mathfrak n}\Gamma_{\mathfrak n}(t_{\mathfrak n})}\diff t_n \label{eq-AT}
    \end{align}
    for 
    \begin{equation} \label{eq-D}
    \mathcal D = \{t[\mathcal N] : 0 < t_{\mathfrak n'} < t_{\mathfrak n} < t \text{ whenever } \mathfrak n' \text{ is a child node of } \mathfrak n\}.
    \end{equation}
    and $\B_\T$ defined inductively as 
    \begin{align}
    \B_\bullet(\eta_\bullet(\varrho)) &:= \eta_{k_\mathfrak r}^{\zeta_{\mathfrak r}}(\varrho) \\
    \B_\T(\eta_\T(\varrho)) &:= \left[\prod_{j = 1}^3 \B_{\T_j}(\eta_{\T_j}(\varrho))\right] - \mathbbm{1}_{\mathfrak{r}\in \mathcal{N}_D}\mathcal{B}_{\mathcal{T}_{c}}(\eta_{\mathcal{T}_{c}}(\varrho)) \mathbb{E}\left(\mathcal{B}_{\mathcal{T}_{2}}(\eta_{\mathcal{T}_{2}}(\varrho))\mathcal{B}_{\mathcal{T}_3}(\eta_{\mathcal{T}_3}(\varrho))\right), 
    \end{align}
    where if $\mathfrak r \in \Nc_d$, then $\mathfrak{n}_c$, $\mathfrak{n}_2$, $\mathfrak{n}_3$ are the three children of $\mathfrak r$ (not necessarily in order). Correspondingly, $\mathcal{T}_{c}$, $\mathcal{T}_{2}$, and $\mathcal{T}_3$ are trees with root node $\mathfrak{n}_c$, $\mathfrak{n}_2$, $\mathfrak{n}_3$ respectively.
\end{definition}

Note we may prove a generalization of Complex Isserlis' Theorem: 
\begin{lemma} \label{lem-multiplicity}
For enhanced trees $\mathcal{T}_1$ and $\mathcal{T}_2$,
\begin{align}
    \mathbb{E}\left(\mathcal{B}_{\mathcal{T}_1}(\eta_{\mathcal{T}_1}(\varrho)){\mathcal{B}^*_{\mathcal{T}_2}(\eta_{\mathcal{T}_2}(\varrho))}\right) = \sum_{\substack{\mathscr P \textit{: enhanced pairings of}\\ \mathcal{T}_1,\mathcal{T}_2}} 1. \label{enhanced-pairing}
\end{align}
\end{lemma}
\begin{proof}
    We prove by induction. Suppose the formula is true for $|\mathcal N_D|=k-1$, then for $|\mathcal N_D|=k$, without loss of generality, we assume that there is one more degenerate node $\mathfrak{n}_*$ coming from $\mathcal{T}_1$. Assume that $\widetilde{\mathcal{B}}_{\mathcal{T}_1}$ represents the multiplicity (product of random phases) of the enhanced tree ${\mathcal{T}_1}$ with $|\mathcal N_{D_1}|=k_1-1$, that is $\mathfrak{n}_* \notin \mathcal N_{D_1}$, $\mathcal{T}_{1j}$'s are subtrees of degenerate node $\mathfrak{n}_*$ for $j=1, 2, 3$, and $\widetilde{\widetilde{\mathcal{B}}}_{\mathcal{T}_1}$ represents the multiplicity of ${\mathcal{T}_1}$ excluding $\mathcal{T}_{11}$ and $\mathcal{T}_{12}$. Here we assume that the root $\mathfrak{n}_{11}$ of $\mathcal{T}_{11}$ is $\mathfrak{n}_{1c^*}$. Then from Definition \ref{def-treeexp} we know that:
\begin{align}
    \mathbb{E}\left(\mathcal{B}_{\mathcal{T}_1}{\mathcal{B}^*_{\mathcal{T}_2}}\right) = \mathbb{E}\left(\widetilde{\mathcal{B}}_{\mathcal{T}_1}{\mathcal{B}^*_{\mathcal{T}_2}}\right) - \mathbb{E}\left(\mathcal{B}_{\mathcal{T}_{11}}{\mathcal{B}^*_{\mathcal{T}_{12}}}\right)\mathbb{E}\left(\widetilde{\widetilde{\mathcal{B}}}_{\mathcal{T}_1}{\mathcal{B}^*_{\mathcal{T}_2}}\right)
\end{align}
which is the number of enhanced pairing of $\mathcal{T}_1,\mathcal{T}_2$ (excluding any complete pairing from $\mathcal{T}_{11}$ and $\mathcal{T}_{12}$) and concludes the proof.
\end{proof}

\begin{definition}{(Couples)} \label{def-couples}
\begin{enumerate}
    \item A \emph{couple} $\Q$ consists of two trees $\T^+$ and $\T^-$, each labeled with opposite signs for which the total number of leaves of sign - equals the total number of leaves of sign +, along with a partition $\mathscr{P}$ of the set of leaves $\mathcal{L}^+ \cup \mathcal{L}^-$ into $(n+1)$ disjoint two-element subsets, where $n = n(\T^+) + n(\T^-)$ is called the \emph{order} of the couple. The partition $\mathscr{P}$ must satisfy the condition $\zeta_{\mathfrak{l}} = -\zeta_{\mathfrak{l}'}$ for every pair $\{\mathfrak{l}, \mathfrak{l}'\} \in \mathscr{P}$. For a  couple $\Q = \{\T^+, \T^-, \mathscr P\}$, we denote the set of branching nodes by $\mathcal{N} = \mathcal{N}^+ \cup \mathcal{N}^-$, and the leaves by $\mathcal{L} = \mathcal{L}^+ \cup \mathcal{L}^-$. Define $\zeta(\Q) \;=\; \prod_{\mathfrak{n} \in \mathcal{N}} \bigl(-i\,\zeta_{\mathfrak{n}}\bigr)$. 
    \item A \emph{decoration} $\mathscr{E}$ of a couple $\Q$ is obtained by decorating each tree $\T^+$ and $\T^-$ according to $\mathscr{D}^+$ and $\mathscr{D}^-$, subject to the further requirement that $k_{\mathfrak{l}} = k_{\mathfrak{l}'}$ whenever $\{\mathfrak{l}, \mathfrak{l}'\} \in \mathscr{P}$. Define $\epsilon_{\mathscr{E}} \;=\; \epsilon_{\mathscr{D}^+}\,\epsilon_{\mathscr{D}^-}$. A decoration $\mathscr{E}$ is called a \emph{$k$-decoration} if $k_{\mathfrak{r}^+} = k_{\mathfrak{r}^-} = k$.
    \item A couple $\Q$ is called an \textit{enhanced couple} if its partition $\mathscr P$ further satisfies that for any $\mathfrak n \in \mathcal N_D$, the leave descendants $\mathcal{L}(\mathfrak{n}_2, \mathfrak{n}_3)$ of $\{ \mathfrak{n}_2, \mathfrak{n}_3\}$ are not completely paired. See Figure \ref{fig:ternary-tree-diamond-2}. 
\end{enumerate}

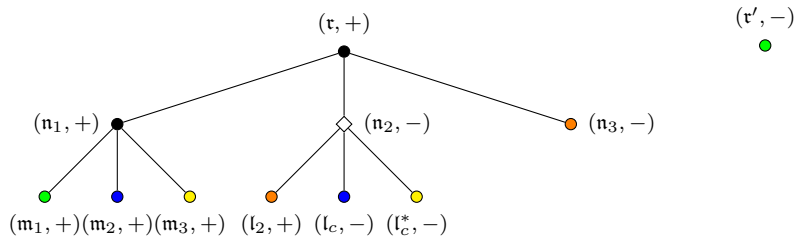
\begin{figure}[ht]
\centering
\begin{tikzpicture}[scale=.8,
  baseline=(root.base),
  level distance=1.2cm,
  level 1/.style={sibling distance=3.75cm},
  level 2/.style={sibling distance=1.2cm},
  level 3/.style={sibling distance=0.5cm}]
  \tikzset{
    solid node/.style={circle,draw,fill=black,inner sep=1.5pt},
    deckdiamond/.style={
      shape=diamond, draw=black, fill=white, minimum size=6pt, inner sep=0pt
    }
  }

  \node[solid node,label=above:{\tiny $(\mathfrak r,+)$}] (root) {}
    child {
      node[solid node,label=left:{\tiny $(\mathfrak n_1,+)$}] {}
      child { node[solid node, fill=green, label=below:{\tiny $(\mathfrak m_1,+)$}] {} }
      child { node[solid node, fill=blue,label=below:{\tiny $(\mathfrak m_2,+)$}] {} }
      child { node[solid node, fill=yellow,label=below:{\tiny $(\mathfrak m_3,+)$}] {} }
    }
    child {
      node[deckdiamond,label=right:{\tiny $(\mathfrak n_2,-)$}] {}
      child { node[solid node, fill=orange,label=below:{\tiny $(\mathfrak l_2,+)$}] {} }
      child { node[solid node, fill=blue,label=below:{\tiny $( \mathfrak{l}_c,-)$}] {} }
      child { node[solid node, fill=yellow,label=below:{\tiny $( \mathfrak{l}_c^*,-)$}] {} }
    }
    child {
      node[solid node, fill=orange,label=right:{\tiny $(\mathfrak n_3,-)$}] {}
    };
\end{tikzpicture}
\hspace{0.5cm}
\begin{tikzpicture}[scale=1,
  level distance=1.2cm,
  level 1/.style={sibling distance=3.5cm},
  level 2/.style={sibling distance=1.2cm},
  level 3/.style={sibling distance=0.5cm}]
  \tikzset{
    solid node/.style={circle,draw,fill=black,inner sep=1.5pt},
    deckdiamond/.style={
      shape=diamond, draw=black, fill=white, minimum size=6pt, inner sep=0pt
    }
  }

  \node[solid node, fill=green,label=above:{\tiny $(\mathfrak r',-)$}] (root) {};
\end{tikzpicture}
\caption{An enhanced couple with leaves paired in the same color. Note that the leaves $\mathfrak{l}_c^*$ and $\mathfrak l_2$ cannot be paired together by Definition \ref{def-couples}.}
\label{fig:ternary-tree-diamond-2}
\end{figure}
\end{definition}

\begin{definition}
For an enhanced couple $\Q = \{\T^+, \T^-, \mathscr P\}$ of order $n$, define  
\begin{align}\label{eq-KQ}
(\K_\Q)(t,s,k)  &:= \left( \frac{\beta T}{N}\right)^n \,\zeta(\Q)\,
\sum_{\mathscr E} \epsilon_{\mathscr E}
\A_\Q(t,s, \Omega[\Nc], \widetilde{\Omega}[\Nc])\prod_{\mathfrak{l} \in \mathcal{L}}^{+} n_{\mathrm{in}}\bigl(k_{\mathfrak{l}}\bigr),
\end{align}
where the sum is taken over all $k$-decorations $\mathscr{E}$ of $\Q$, the product $\prod_{\mathfrak{l} \in \mathcal{L}}^{+}$ is taken over leaves with $+$ signs, and
    \begin{align}
    \mathcal{A}_{\mathcal{Q}}(t,s,\Omega[\Nc], \widetilde{\Omega}[\Nc]) &= \int_{\mathcal E} \prod_{\mathfrak n \in \mathcal N} e^{i\zeta_{\mathfrak n}\Gamma_{\mathfrak n}(t_{\mathfrak n})}\diff t_n, \label{eq-AQ}
    \end{align}
    with 
    \begin{align}\label{eq-E}
    \mathcal{E} &= \Bigl\{\,t[\mathcal{N}] : 0 < t_{\mathfrak{n}'} < t_{\mathfrak{n}}
    \text{ whenever $\mathfrak{n}'$ is a child of $\mathfrak{n}$}; \\
    &\qquad\qquad\quad t_{\mathfrak{n}} < t \text{ for } \mathfrak{n} \in \mathcal{N}^{+},
    \text{ and } t_{\mathfrak{n}} < s \text{ for } \mathfrak{n} \in \mathcal{N}^{-}
\Bigr\}, \nonumber
\end{align}
\end{definition}

So, by Lemma \ref{lem-multiplicity},
\begin{equation}
\E \left[(\J_{\T^+})_k(t)\overline{(\J_{\T^-})_k(s)}\right] = \sum_{\mathcal Q} (\K_{\Q})(t,s,k), 
\end{equation}
where the summation is taken over all enhanced couples $\Q = \{\T^+, T^-\}$ with partition $\mathscr P$. Note that if the total number of leaves of sign + in the couple is not equal to the total number of leaves of sign -, no such couple exists, so $\E \left[(\J_{\T^+})_k(t)\overline{(\J_{\T^-})_k(s)}\right] = 0$. 

\subsection{Phase Renormalization and Ansatz} \label{subsec-ansatz}We take the following Ansatz for $b_k(t)$, namely 
\begin{equation}\label{eq-ansatz}
b_k(t) = \sum_{|\T| \leq M} (\J_\T)_k(t) + \mathcal R_k(t), 
\end{equation}
for the sum over enhanced trees $\T$ of sign $+$, $\mathcal R$ the \textit{remainder} and $M$ defined in Section \ref{subsec-notation}. Note that defining the remainder $\mathcal R$ depends on the choice of phase renormalization $\widetilde \omega_k$ in \eqref{eq-bk}: 

\begin{proposition}[Phase Renormalization] \label{prop-phase-renormalization}
For each $M$, here exists unique $\widetilde \omega_k$ such that 
\begin{equation}\label{eq-phase}
\begin{cases}\partial_t \tilde \omega_k(t) = \left( \frac{\beta T}{N}\right) \sum\limits_{\ell \in \Z_N \cap [0,1)} 2T_{+,+,-}(\ell, k, \ell)\sum\limits_{\Q} \K_\Q(t,t,\ell) \\
\tilde \omega_k (0) = 0.
\end{cases}
\end{equation}
for enhanced couples $\Q = \{\T^+, \T^-, \mathscr P\}$ such that $|\T^+|, |\T^-| \leq M$. 
\end{proposition}
\begin{proof}
Note that $T_{+,+,-}(\ell, k, \ell) = \frac{1}{4\kappa^2}\omega_k \omega_l$, so we may write $\tilde \omega_k(t)  = \omega_k A(t)$. Then, 
\begin{align*}
\dot{A}(t) &= \left( \frac{\beta T}{N}\right) \sum_{\ell \in \Z_N \cap [0,1)} \frac{3}{2\kappa^2} \omega_\ell \sum_{\Q} \K_\Q (t,t, \ell),
\end{align*}
where $\partial_t A(t) = G(t,A(t))$. Notice that $G$ is smooth in its variables and for fixed $N$ and almost every $\varrho,$ $G$ is bounded. Therefore, we have existence and uniqueness of $A$ for almost every $\varrho$. 
\end{proof}

\begin{remark}
Note that the only dependence on $k$ in $\widetilde \omega_k$ is via $\omega_k$, so we may drop the dependence on $\widetilde \Omega$ in \eqref{eq-AT} and \eqref{eq-AQ}. See \cite{FPUTRD} for a more detailed calculation of Proposition \ref{prop-phase-renormalization} to see that $\widetilde \omega_k$ must have the structure above. 
\end{remark}

Define
\begin{align*}
\mathcal C^D(u,v)_k(t) &= -i\left( \frac{\beta T}{N}\right) \sum\limits_{\ell \in \Z_N \cap [0,1)} T_{+,+,-}(k, \ell, \ell)  \E \left[u_\ell \overline{v_\ell}\right],
\end{align*}

Now, we may write the remainder in the following way: 

\begin{proposition}[Structure of Remainder] \label{strucRem} With $\widetilde \omega_k$ as in \eqref{eq-phase}, the remainder $\mathcal R$ takes the form
\begin{equation} \label{eq-remainder}
\mathcal R = \mathscr{R} + \mathscr L(\mathcal R) + \mathscr Q (\mathcal R, \mathcal R) + \mathscr C(\mathcal R, \mathcal R, \mathcal R),
\end{equation}
where 
\begin{align}
\mathscr{R} &= \sum_{\substack{n_i := |\T_i| < M \\ n_1 + n_2 + n_3 \geq M}}\mathcal I \mathcal C^+(\J_{\T_1}, \J_{\T_2}, \J_{\T_3}) - \mathscr{R}^D \label{eq-rem}\\
\mathscr L &= \sum_{\substack{|\T_i| < M}}^{\mathrm{cyc}} \mathcal I \mathcal C^+(\J_{\T_1}, \J_{\T_2}, \mathcal R) - \mathscr L^D \label{eq-linearop}\\
\mathscr Q &= \sum_{\substack{|\T| < M}}^{\mathrm{cyc}} \mathcal I \mathcal C^+(\J_{\T}, \mathcal R, \mathcal R) \\
\mathscr C &=  \mathcal I \mathcal C^+(\mathcal R, \mathcal R, \mathcal R),
\end{align}
for enhanced trees of sign + and $\sum^{\text{cyc}}$ indicating cyclic permutations of the entries of $\mathcal C^+$. The degenerate terms are 
\begin{align}
\mathscr R^D &= \sum_{\substack{n_i := |\T_i| < M \\ n_1 + n_2 + n_3 \geq M}} \mathcal I \mathcal C^D(\J_{\T_1}, \J_{\T_3})_k(\J_{\T_2})_k + \mathcal I \mathcal C^D(\J_{\T_2}, \J_{\T_3})_k(\J_{\T_1})_k \\
\mathscr L^D &= \sum_{\substack{|\T_i| < M}} 2\mathcal I \mathcal C^D(\J_{\T_1}, \J_{\T_2})_k\mathcal R_k
\end{align}
\end{proposition}

\begin{proof}
This can be shown using the fact that \[b_k(t) = b_k(0) + \mathcal I \left[\mathcal C^+(b,b,b)_k(t) +i \widetilde \omega_k'(t)b_k(t)\right.]\]
This is a standard computation with this Ansatz. See for instance, \cite{2019}. The only clarification needed is for the addition of the degenerate terms. For this, we note that we may also write an equivalent inductive definition of $(\J_\T)_k$ for enhanced tree $\T$, where 
\begin{align*}
(\J_\bullet)_k(t) &= b_k(0) = \sqrt{n_{\mathrm{in}}(k)} \eta_k^{\zeta_{\mathfrak r}} (\varrho), \\
(\J_\T)_k(t) &= \mathcal I\left[\mathcal C^{\zeta_{\mathfrak r}}_{\zeta_{\mathfrak n_1}, \zeta_{\mathfrak n_2}, \zeta_{\mathfrak n_3}}(\J_{\T_1}^{\zeta_{\mathfrak n_1}}, \J_{\T_2}^{\zeta_{\mathfrak n_2}}, \J_{\T_3}^{\zeta_{\mathfrak n_3}})_k(t) - \mathbbm{1}_{\mathfrak{r}\in \mathcal{N}_D} (\mathcal \J_{\mathcal{T}_{c}})_k \mathcal C^D(\J_{\T_{2}}, \overline{\J_{\T_3}})_k(t).\right] 
\end{align*}
\end{proof}

\subsection{Molecules}\label{subsec-mol}
\begin{definition}[Molecules]\label{def-molecules}
A \textit{molecule} $\mathbb{M}$ is a directed graph whose vertices are called \textit{atoms} and edges are called \textit{bonds}. Multiple bonds between the same pair of atoms and self-loops are allowed. We write $v \in \mathbb{M}$ for an atom $v$ in $\mathbb{M}$, and $\ell \in \mathbb{M}$ for a bond $\ell$ in $\mathbb{M}$. If $v$ is an endpoint of $\ell$, we write $\ell \sim v$. For each such pair $(v,\ell)$, define
\[
\zeta_{v,\ell} = \begin{cases}
1 & \text{if $\ell$ is outgoing from $v$,}\\
-1 & \text{if $\ell$ is incoming to $v$.}
\end{cases}
\]
We further require that $\mathbb{M}$ does not have any connected components consisting only of atoms each having degree 4 (where the degree is considered in the undirected sense).

For a molecule $\mathbb{M}$, define the quantity
\[
\chi := E - V + F,
\]
where $E$ is the number of bonds, $V$ is the number of atoms, and $F$ is the number of connected components. Likewise, let $L$ denote the number of self-loops in $\M$. An \textit{atomic group} in $\mathbb{M}$ is a subset of atoms together with all bonds connecting those atoms. A single bond $\ell$ is called a \textit{bridge} if removing it increases the number of connected components by one.

An \textit{enhanced molecule} is a molecule equipped with a designated set of atoms, denoted $\V_D$ where each $v \in \V_D$ has in-degree and out-degree at most 2. For all $\ell \sim v$, we may additionally designate a pair of bonds of opposite signs $\ell^+, \ell^-$ such that if we remove the two bonds, the number of connected components of the molecule remains unchanged. 
\end{definition}

\begin{definition}[Molecules of couples]\label{couple-molecule}
Let $\Q$ be a nontrivial couple. We define the corresponding \textit{molecule} $\mathbb{M} = \mathbb{M}(\Q)$ as follows. The atoms of $\mathbb{M}$ correspond to the branching nodes $\mathfrak{n} \in \mathcal{N}$ of $\Q$. For any two branching nodes $\mathfrak{n}_1, \mathfrak{n}_2$, we draw a bond between the corresponding atoms $v_1, v_2$ if:
\begin{enumerate}
\item One of $\mathfrak{n}_1, \mathfrak{n}_2$ is a parent of the other (Parent-Child, or PC), or
\item A child of $\mathfrak{n}_1$ is paired with a child of $\mathfrak{n}_2$ as leaves (Leaf Pair, or LP).
\end{enumerate}
The direction of each bond is determined as follows:
\begin{enumerate}[label=(\roman*)]
\item For an LP bond, the bond is directed away from the atom whose corresponding child in the couple has sign $-$ and towards the atom whose corresponding child has sign $+$.
\item For a PC bond, if the child atom corresponds to a branching node with sign $-$, then the direction goes from the P atom to the C atom. If the child atom corresponds to a branching node with sign $+$, the direction is reversed.
\end{enumerate}

We may form a \textit{labeled molecule} by labeling each bond. For a Parent-Child (PC) bond, we label the bond as PC and place a P at the atom corresponding to the parent branching node and a C at the atom corresponding to the child branching node. For a Leaf Pair (LP) bond, we label the bond as LP.

For any atom $v \in \mathbb{M}(\Q)$, let $\mathfrak{n} = \mathfrak{n}(v)$ be its corresponding branching node in $\Q$. For any bond $\ell \sim v$, define $\mathfrak{m} = \mathfrak{m}(v,\ell)$ by:
\begin{enumerate}[label=(\roman*)]
\item If $\ell$ is PC and $v$ is labeled $C$, then $\mathfrak{m} = \mathfrak{n}$.
\item If $\ell$ is PC and $v$ is labeled $P$, then $\mathfrak{m}$ is the branching node corresponding to the other endpoint of $\ell$.
\item If $\ell$ is LP, then $\mathfrak{m}$ is the leaf (from the leaf pair defining $\ell$) that is a child of $\mathfrak{n}$.
\end{enumerate}

For an enhanced couple $\Q$, we may form an enhanced molecule $\M = \M(\Q)$ by performing the above in addition to designating $v = v(\mathfrak n)$ a degenerate atom for $\mathfrak n$ a degenerate branching node of $\Q$. We then have atoms $v_p, v_c, v_c^*, v_3$ corresponding to $\mathfrak n_p, \mathfrak n_c, \mathfrak n_c^*, \mathfrak n_3$ (the potential parent of $\mathfrak n$ along with its three children). We let $(\ell_1^+, \ell_1^-)$ be the bonds connecting $v$ to $v_p$ and $v_c$ according to signs and $(\ell_2^+, \ell_2^-)$ be the bonds connecting $v$ to $v_c^*$ and $v_3$ according to signs. See Figure \ref{fig:molecule}.
\end{definition}
\begin{figure}
\centering

\begin{tikzpicture}[
        scale = .6, baseline=(top.base),
        level distance=1.4cm,
        level 1/.style={sibling distance=2.1cm},
        level 2/.style={sibling distance=0.8cm},
        hollow node/.style  ={circle,draw,inner sep=1.6},
        diamond node/.style ={diamond,draw,inner sep=1.6},
        solid node/.style   ={circle,draw,fill=black,inner sep=1.6},
        red node/.style     ={circle,draw,fill=red,inner sep=1.6},
        blue node/.style    ={circle,draw,fill=blue,inner sep=1.6},
        purple node/.style  ={circle,draw,fill=purple,inner sep=1.6},
        orange node/.style  ={circle,draw,fill=orange,inner sep=1.6},
        yellow node/.style  ={circle,draw,fill=yellow,inner sep=1.6},
        green  node/.style  ={circle,draw,fill=green,inner sep=1.6}]
  \node[hollow node, label = left: {\tiny $k$}] (top) {1}
    child{node[hollow node, label = left: {\tiny $m$}]{2}
      child{node[green node,  label = below: {\tiny $k$}]{}}
      child{node[blue node,   label = below: {\tiny $l$}]{}}
      child{node[yellow node, label = below: {\tiny $p$}]{}}
    }
    child{node[diamond node, label = left:{\tiny $l$}]{3}
      child{node[orange node, label = below: {\tiny $p$}]{}}
      child{node[blue node,   label = below: {\tiny $l$}]{}}
      child{node[yellow node, label = below: {\tiny $p$}]{}}
    }
    child{node[orange node,  label = below: {\tiny $p$}]{}};
\end{tikzpicture}
\hspace{0.5cm}
\begin{tikzpicture}[
        scale = .6, baseline=(top.base),
        level distance=1.4cm,
        level 1/.style={sibling distance=1.3cm},
        level 2/.style={sibling distance=1.0cm},
        hollow node/.style  ={circle,draw,inner sep=1.6},
        diamond node/.style ={diamond,draw,inner sep=1.6},
        green  node/.style  ={circle,draw,fill=green,inner sep=1.6}]
  \node[green node, label = left: {\tiny $k$}] (top) {};
\end{tikzpicture}
\hspace{1.2cm}
\begin{tikzpicture}[
        scale=0.5, baseline=(top.base), >=stealth,
        open node/.style   ={circle,draw,inner sep=1.8},
        diamond node/.style={diamond,draw,inner sep=1.5},
        empty node/.style  ={inner sep=.5,outer sep=0}]
  \def\L{4}
  \def\shift{3pt}
  \coordinate (A) at (0,0);           
  \coordinate (B) at (\L,0);          
  \coordinate (C) at ($(A)+(60:\L)$); 

  \node[open node]    (N2) at (A) {2};
  \node[diamond node] (N3) at (B) {3};
  \node[open node]    (top) at (C) {1}; 

  \draw[->,thick,transform canvas={yshift=\shift}]
      (N3) -- node[empty node,pos=.6,fill=white] {$l$} (N2);
  \draw[->,thick,transform canvas={yshift=-\shift}]
      (N3) -- node[empty node,pos=.4,fill=white] {$p$} (N2);

  \draw[->,thick,transform canvas={xshift=\shift, yshift=2.5pt}]
      (top) -- node[empty node,pos=.6,fill=white] {$p$} (N3);
  \draw[->,thick,transform canvas={xshift=-\shift, yshift=-2.5pt}]
      (top) -- node[empty node,pos=.4,fill=white] {$l$} (N3);

  \draw[->,thick]
      (N2) -- node[empty node,midway,fill=white] {$m$} (top);
\end{tikzpicture}
\caption{An example of an enhanced couple and its corresponding enhanced molecule, corresponding to Figure \ref{fig:ternary-tree-diamond-2}.}
\label{fig:molecule}
\end{figure}
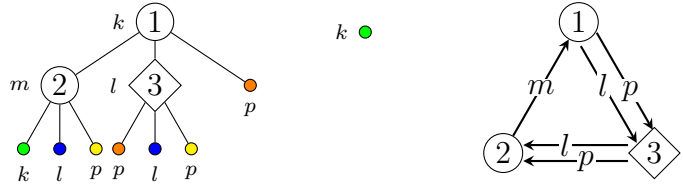

\begin{proposition}[Correspondence between Molecules and Couples]
\label{prop:molecules-couples}
We summarize the relationship between nontrivial couples and molecules as follows:
\begin{enumerate}
\item For any nontrivial couple $\mathcal{Q}$ of order $n$, the construction in Definition~\ref{couple-molecule} yields a connected \emph{molecule} $\mathbb{M}(\mathcal{Q})$.  This molecule $\mathbb{M}$ has exactly $n$ atoms, $2n-1$ bonds, and either two atoms of degree $3$, or a single atom of degree $2$, with all remaining atoms having degree $4$.
\item Conversely, given any molecule $\mathbb{M}$ with $n$ atoms (as in Definition~\ref{def-molecules}), there are at most $C^n$ different couples $\mathcal{Q}$ (if any) for which $\mathbb{M}(\mathcal{Q})$ is exactly $\mathbb{M}$.
\end{enumerate}
\end{proposition}
\begin{proof}
See Proposition 9.4 and 9.16 of \cite{WKE}. 
\end{proof}

\begin{definition}[Decorations of Molecules]
\label{def:decorations-molecules}
Let $\mathbb{M}$ be a molecule. For each atom (vertex) $v \in \mathbb{M}$, we fix $k_v \in \Z_N$, with the condition $k_v = 0$ if $v$ has degree 4, and $\alpha_v \in \mathbb{R}$.

A \emph{$(k_v,\alpha_v)$-decoration} of $\mathbb{M}$ assigns an integer $k_\ell \in \Z_N$ to each bond $\ell \in \mathbb{M}$, in such a way that for each atom $v$,
\[
\sum_{\ell \sim v} \zeta_{v,\ell} \,k_\ell \;=\; k_v,
\quad\text{and}\quad
\bigl|\Omega_v - \alpha_v\bigr|\;<\;T^{-1},
\]
where
\[
\Omega_v = \sum_{\ell: \ell\sim v} \zeta_{v,\ell} \, \omega(k_\ell).
\]
Here, the summation $\sum_{\ell \sim v}$ runs over all bonds $\ell$ such that $\ell\sim v$, and $\zeta_{v,\ell}$ is defined to be $+1$ whenever $\ell$ is outgoing from $v$, and $-1$ otherwise. 

Suppose $\mathbb{M}$ arises from a nontrivial couple $\mathcal{Q}$, and let $k\in\Z_N$. Define a \textit{$k$-decoration} of $\mathbb M$ to be a $(k_v, \alpha_v)$-decoration where 
\[
k_v \;=\;
\begin{cases}
0 & \text{if $v$ has degree 2 or 4,}\\
+k & \text{if $v$ has degree 3 and $\mathfrak n(v)$ has sign +.}\\
-k & \text{if $v$ has degree 3 and $\mathfrak n(v)$ has sign -.}
\end{cases}
\]
Given any $k$-decoration of $\mathcal Q$ in the sense of Definition \ref{def-couples}, define a $k$-decoration of $\mathbb M(\mathcal Q)$ such that $k_{\ell} = k_{\mathfrak m(v,\ell)}$ for an endpoint $v$ of $\ell$. $k_\ell$ is well-defined and independent of the choice of $v$, and gives a one-to-one correspondence between $k$-decorations of $\mathcal Q$ and $k$-decorations of $\mathbb M(\mathcal Q)$. Moreover, for such decorations we have 
\begin{equation}
\Omega_v = \begin{cases}
-\zeta_{\mathfrak n(v)}\Omega_{\mathfrak n(v)} & \text{if } v \text{ has degree 2 or 4,} \\
-\zeta_{\mathfrak n(v)}\Omega_{\mathfrak n(v)} - \zeta_{\mathfrak n(v)}\omega(k) & \text{if } v \text{ has degree 3.} \\
\end{cases}
\end{equation}
\end{definition}

\begin{definition}[Degenerate Atoms]
\label{def:degenerate-atoms}
Consider an enhanced molecule $\mathbb{M}$ equipped with a $(k_v,\alpha_v)$-decoration as above. We say that an atom $v \in \mathbb{M}$ is \emph{degenerate} if there exist two \emph{degenerate} bonds $\ell^+, \ell^- \sim v$, pointing in opposite directions at $v$, such that $k_{\ell^+}=k_{\ell^-}$. Furthermore, $v$ is \emph{fully degenerate} if \emph{all} bonds $\ell\sim v$ share the same decoration $k_\ell$. In particular, for each degenerate branching node in the couple $\Q$, there corresponds a degenerate atom in the molecule $\mathbb{M}(\Q)$.
\end{definition}

\begin{definition}[Gaps and Tame Atoms] \label{def-gap}Suppose $\M$ is a molecule with a $(k_v, \alpha_v)$-\hspace{0pt}decoration. For atom $v$ and bonds $\ell_1, \ell_2 \sim v$ of opposite directions, we define the \textit{gap} of $v$ with respect to $(\ell_1, \ell_2)$ to be 
\begin{equation}
h:= |k_{\ell_1} - k_{\ell_2}|.
\end{equation}
We say $v$ is \textit{small gap} (SG) with respect to $(\ell_1, \ell_2)$ if $h \lesssim \beta N^{\delta}$ and \textit{large gap} (LG) otherwise. We say that $v$ is tame if for all $\ell_1, \ell_2 \sim v$ of opposite directions, $v$ is small gap with respect to $(\ell_1, \ell_2)$. Likewise, if $v$ is large gap with respect to all $\ell_1, \ell_2 \sim v$ of opposite directions, we say that $v$ is a \textit{large gap} (LG) atom. Note that any atom where all bonds have the same direction is automatically LG.  
    
\end{definition}

\section{Main Estimates}
\label{section-mainest}
The purpose of this section is to provide the statements of the main estimates including the Feynman diagram estimates, whose proofs take up Sections \ref{section-counting} - \ref{section-remainder}, and integral estimates, which allow us to bound oscillatory integrals. 

\subsection{Bounds on Feynman diagrams}
Recall we fixed sufficiently large constant $M$ in Section \ref{subsec-notation}. 
\begin{proposition}(Bounds on Couples)\label{bound-couple}
For each $1 \leq n \leq M^3$ and $0 \leq t \leq  1$, and enhanced trees $\mathcal{T}_{1}, \mathcal{T}_{2}$, the following should hold:
\begin{align}
\left|\sum_{\Q}\K_\Q(t,t,k)\right|& \leq C\beta^2 T (N^{-\delta})^{n - \frac{5}{2}},\label{E1}\\
\left|\sum_{\Q}\dot{\K}_\Q(t,t,k)\right|&\leq C \beta^2 T^2 (N^{-\delta})^{n - \frac{5}{2}} ,\label{E2}
\end{align}
for some constant $C$, where the summation is taken over couples $\Q = \{\T^+, \T^-, \mathscr P\}$ such that $|\T^+| +  |\T^-| = n$.
\end{proposition}
\begin{proposition}{(Bound on Operator $\Ell$)} \label{prop-remainder}
With probability $\geq 1 - N^{-A}$, the linear operator $\Ell$ defined in (\ref{eq-linearop}) satisfies  
\begin{equation} \label{remainder}
||\Ell^n||_{Z \to Z} \lesssim N^{-\delta n/2}N^{50}
\end{equation}
for each $ 0 \leq n \leq M$ and some $A \geq 40$. 
\end{proposition}
The proof of Proposition \ref{bound-couple} will occupy Sections \ref{section-counting} - \ref{section-algorithm}, while Proposition \ref{prop-remainder} will be proved in Section \ref{section-remainder}.

\subsection{The integral estimates} \label{sec-int}
\begin{proposition}\label{prop-integral_est}
    For any ternary tree $\mathcal{T}$, fix functions $f_\mathfrak{n}:[0,1]\to\mathbb{C} \in C^1$ for each branching node $\mathfrak{n} \in \mathcal{N}$. Consider the corresponding function $\mathcal A_\T^f$ defined for $t \in [0,1]$ similarly to \eqref{eq-AT}:
    \begin{align}
    \mathcal{A}_{\mathcal{T}}^f(t,\Omega[\Nc]) &= \int_{\mathcal D} \prod_{\mathfrak n \in \mathcal N} e^{i\zeta_{\mathfrak n}\Omega_{\mathfrak n}(Tt_\mathfrak n + A(t_\mathfrak n))}f_\mathfrak n(t_\mathfrak n)\diff t_n, \label{eq-ATf}
    \end{align}
    where $\mathcal D$ is defined in \eqref{eq-D}. Then, if 
    \begin{align}
    |\dot{A}(t)| &\lesssim \beta T, \label{eq-Adot-bound} \\
    |\ddot{A}(t)| &\lesssim T, \label{eq-Addot-bound}
    \end{align}
    and the functions $f_\mathfrak n$ satisfy
    \begin{equation} 
    \|f_{\mathfrak n} \|_{C^1} \leq D_\mathfrak n, \label{eq-f-bound}
    \end{equation} 
    then 
    \begin{align}
    \left|\mathcal{A}^f_{\mathcal{T}}(t,\Omega[\mathcal{N}], \widetilde{\Omega}[\Nc])\right|&\lesssim \sum_{(d_{\mathfrak{n}}:\mathfrak{n}\in \mathcal{N})}\prod_{\mathfrak{n}\in \mathcal{N}}D_{\mathfrak{n}}{\langle Tq_{\mathfrak{n}}\rangle}^{-1}. \label{eq-ATf-bound}
    \end{align}
    for $q_\mathfrak n$ defined in \eqref{eq-def-qn}. 
\end{proposition} 

\begin{proof}
Throughout, note that since $\left|\dot{A}({t_{\mathfrak{n}}})\right| \lesssim \beta T$,
\begin{align}
    \left|\frac{T}{T + \dot A(t)}\right| \lesssim 1. \label{eq-frac-phase-bound}
\end{align}
Therefore, via integration by parts, we may obtain 
\begin{align}
&\int_0^te^{i\zeta_r\Omega_\mathfrak r(Tt' + A(t'))}f_\mathfrak r(t') \diff t' = \frac{e^{i\zeta_\mathfrak r \Omega_\mathfrak r(Tt + A(t))}f_\mathfrak r(t)}{i \zeta_\mathfrak r \Omega_\mathfrak r(T + \dot A(t))} - \frac{f_\mathfrak r(0)}{i \zeta_\mathfrak r \Omega_r(T + \dot A(0))} \notag\\
&\hspace{2cm}- \frac{1}{i\zeta_\mathfrak r \Omega_\mathfrak r T}\int_0^t \frac{e^{i \zeta_\mathfrak r \Omega_\mathfrak r(Tt' + A(t'))}T \dot f_\mathfrak r(t')}{T + \dot A(t')} \diff t' + \frac{1}{i\zeta_\mathfrak r \Omega_\mathfrak r T}\int_0^t \frac{e^{i\zeta_\mathfrak r\Omega_\mathfrak r(Tt' + A(t'))}T \ddot A(t') f_\mathfrak r(t')}{(T + \dot A(t'))^2} \diff t'.\label{eq-ATf-1}
\end{align}
Using \eqref{eq-frac-phase-bound} as well as (\ref{eq-Adot-bound} - \ref{eq-f-bound}), we may obtain 
\begin{align*}
\left|\int_0^te^{i\zeta_r\Omega_\mathfrak r(Tt' + A(t'))}f_\mathfrak r(t') \diff t'\right| \lesssim \frac{D_{\mathfrak r}}{\langle T\Omega_\mathfrak r \rangle}.
\end{align*}

Now suppose that the bound \eqref{eq-ATf-bound} holds for all trees of order at most $n - 1$. Note that any ternary tree $\T$ can be obtained from a tree $\T'$ of order $n - 1$ where we obtain $\T$ from $\T'$ by attaching a single branching node to a leaf $\mathfrak l$ whose siblings have scale at most 1. Let $\mathfrak p$ denote the parent of $\mathfrak l$, $\T_\mathfrak p$ denote the corresponding tree rooted at $\mathfrak p$ with branching nodes $\mathcal N_\mathfrak p$ and $\mathcal N_{\mathfrak p}^- = \mathcal N \setminus \mathcal N_p$. Define 
\begin{equation}
\mathcal A_{\mathcal N_{\mathfrak p}^-}(t[\mathcal N_{\mathfrak p}^-]) := \prod_{\mathfrak n \in \mathcal N_{\mathfrak p}^-} e^{i \zeta_\mathfrak n \Omega_\mathfrak n(Tt_\mathfrak n + A(t_\mathfrak n))} f_\mathfrak n(t_\mathfrak n) \diff t_\mathfrak n.
\end{equation}
For simplicity, let's assume $\zeta_\mathfrak p = \zeta_\mathfrak l = +$ and that $\mathfrak l$ is the third child of $\mathfrak p$. Then, we may write
\begin{align}
\mathcal A_{\T'}^f(t, \Omega[\Nc]) &= \int_{\mathcal D_\mathfrak p} \mathcal A_{\mathcal N_{\mathfrak p}^-}(t[\mathcal N_{\mathfrak p}^-]) e^{i\Omega_\mathfrak p(Tt_\mathfrak p + A(t_\mathfrak p))}f_\mathfrak p(t_\mathfrak p) \prod_{j = 1}^2 \mathcal A_{\T_j}^f (t_\mathfrak p, \Omega[\Nc_j]) \diff t_\mathfrak p \\
\mathcal A_{\T}^f(t, \Omega[\Nc]) &= \int_{\mathcal D_p} \mathcal A_{\mathcal N_{\mathfrak p}^-}(t[\mathcal N_{\mathfrak p}^-]) e^{i\Omega_\mathfrak p(Tt_\mathfrak p + A(t_\mathfrak p))}f_\mathfrak p(t_\mathfrak p) \\
& \hspace{3cm} \times\int_0^{t_\mathfrak p} e^{i\Omega_\mathfrak l(Tt_\mathfrak l + A(t_{\mathfrak l}))} f_{\mathfrak l}(t_\mathfrak l)\diff t_\mathfrak l\prod_{j = 1}^2 \mathcal A_{\T_j}^f (t_\mathfrak p, \Omega[\Nc_j]) \diff t_\mathfrak p, \notag
\end{align}
where $\mathcal D_\mathfrak p = \{t[\Nc_{\mathfrak p}^-] \colon 0 < t_{\mathfrak n'} < t_{\mathfrak n} < t \text{ whenever $\mathfrak n' \in \Nc_{\mathfrak p}^-$ is a child of $n \in \Nc_{\mathfrak p}^-$}\}$. If $|T\Omega_\mathfrak l| \lesssim 1,$ then we may use that \eqref{eq-ATf-bound} holds for $\T'$ and 
\begin{align*}
\left\lVert f_{\mathfrak p}(t_\mathfrak p) \int_0^{t_\mathfrak p} e^{i\Omega_\mathfrak l(Tt_\mathfrak l + A(t_{\mathfrak l}))} f_{\mathfrak l}(t_\mathfrak l)\diff t_\mathfrak l \right \rVert_{C^1} \lesssim D_\mathfrak p .
\end{align*}
On the other hand if $|T\Omega_\mathfrak l| \gtrsim 1$ then, we may calculate $\mathcal A_\T^f(t, \Omega[\Nc])$ using \eqref{eq-ATf-1}, so that 
\begin{equation*}
\mathcal A_{\T}^f(t, \Omega[\Nc]) = I + II + III, 
\end{equation*}
for 
\begin{align}
    I &:= \frac{1}{i \Omega_\mathfrak l T}\int_{\mathcal{D}_\mathfrak{p}} \mathcal{A}_{\Nc_{\mathfrak p}^-}(s[\mathcal{N}^-_\mathfrak{p}])e^{i(\Omega_\mathfrak p + \Omega_\mathfrak l)(Tt_\mathfrak p + A(t_\mathfrak p))} \frac{ Tf_{\mathfrak{p}}(t_\mathfrak p)f_{\mathfrak{l}}(t_{\mathfrak p})}{T+\dot A(t_\mathfrak p)} \prod_{j=1}^2\mathcal{A}^f_{\mathcal{T}_j}({t_{\mathfrak{p}}},\Omega[\Nc_j]) \diff t_\mathfrak p, \\
    II &:= -\frac{f_{\mathfrak{l}}(0)}{iT\Omega_{\mathfrak{l}}(1+\dot A(0))}\mathcal{A}^f_{\mathcal{T}'}(t,\Omega[\mathcal{N}'], \widetilde{\Omega}[\Nc']),\\
    III &:= \frac{1}{i \Omega_\mathfrak l T} \int_{\mathcal{D}_\mathfrak{p}} \mathcal{A}_{\Nc_{\mathfrak p}^-}(t[\mathcal{N}^-_\mathfrak{p}])e^{i\Omega_\mathfrak p(T t_\mathfrak p + A(t_\mathfrak p))}f_{\mathfrak{p}}(t_\mathfrak p)\prod_{j=1}^2\mathcal{A}^f_{\T_j}(t_\mathfrak p,\Omega[\Nc_j]) \\
    &\hspace{1.5cm}\times\int_0^{t_\mathfrak p}e^{i \Omega_\mathfrak l(Tt_\mathfrak l + A(t_\mathfrak l))}\left[\frac{T\ddot{A}(t_{\mathfrak{l}})f_{\mathfrak{l}}(t_\mathfrak l)}{(T + \dot A(t_\mathfrak l))^2}+\frac{T\dot{f}_{\mathfrak{l}}(t_{\mathfrak{l}})}{T + \dot A(t_\mathfrak l)}\right]\diff t_\mathfrak l \diff t_\mathfrak p. \notag
\end{align}

For $I$, we note that 
\begin{align*}
\frac{d}{dt_\mathfrak p} \left(\frac{Tf_\mathfrak p(t_\mathfrak p) f_\mathfrak l(t_\mathfrak p))}{T + \dot A(t_\mathfrak p)}\right) &= \frac{T(\dot f_\mathfrak p(t_\mathfrak p) f_\mathfrak l(t_\mathfrak p) + f_\mathfrak p(t_\mathfrak p) \dot f_\mathfrak l(t_\mathfrak p))}{T + \dot A(t_\mathfrak p)} - \frac{Tf_\mathfrak p(t_\mathfrak p) f_\mathfrak l(t_\mathfrak p) \ddot A(t_\mathfrak p)}{(T + \ddot A(t_\mathfrak p))^2},
\end{align*}
so that 
\begin{align*}
\left\lVert\frac{Tf_\mathfrak p(t_\mathfrak p) f_\mathfrak l(t_\mathfrak p))}{T + \dot A(t_\mathfrak p)}\right\rVert_{C^1} \lesssim D_\mathfrak p D_\mathfrak l
\end{align*}
and we may use the fact that \eqref{eq-ATf-bound} holds for $\T'$ bound $I$ by 
\begin{equation} \label{eq-Atf-induction-bound}
\frac{D_\mathfrak l}{\langle T \Omega_\mathfrak l\rangle} \sum_{(d_{\mathfrak{n}}:\mathfrak{n}\in \Nc')}\prod_{\mathfrak{n}\in \Nc'}D_{\mathfrak{n}}{\langle Tq_{\mathfrak{n}}\rangle}^{-1},
\end{equation}
where $q_\mathfrak n$ may depend on $\Omega_\mathfrak l$. We may easily recover \eqref{eq-Atf-induction-bound} for $II$. Finally, we recover \eqref{eq-Atf-induction-bound} for $III$ by noticing
\begin{align*}
\left \lVert f_\mathfrak p(t_\mathfrak p) \int_0^{t_\mathfrak p}e^{i \Omega_\mathfrak l(Tt_\mathfrak l + A(t_\mathfrak l))}\left[\frac{T\ddot{A}(t_{\mathfrak{l}})f_{\mathfrak{l}}(t_\mathfrak l)}{(T + \dot A(t_\mathfrak l))^2}+\frac{T\dot{f}_{\mathfrak{l}}(t_{\mathfrak{l}})}{T + \dot A(t_\mathfrak l)}\right]\diff t_\mathfrak l \right \rVert_{C^1} \lesssim D_\mathfrak p D_\mathfrak l.
\end{align*}
\end{proof}

\subsection{Plan of paper}
In Section \ref{section-counting}, we prove the counting estimates which we make use of in Section \ref{section-algorithm} as well as a related estimate which allows us to bound iterates of \eqref{WKE}. In Section \ref{section-irrchains}, we introduce irregular chains and self-loops in couples and prove estimates which allow us to remove or discount such structures in Section \ref{section-algorithm}, where we prove Proposition \ref{bound-couple} by reducing it to a counting estimate on molecules and laying out a counting algorithm. In Section \ref{section-remainder}, we adapt the arguments in Section \ref{section-algorithm} to prove Proposition \ref{prop-remainder} and then bound the remainder $\mathcal R$. Finally, in Section \ref{section-mainthm}, we put everything together to prove Theorem \ref{main}.

\section{Counting Estimates}
\label{section-counting}
\allowdisplaybreaks
In this section, we prove the vector counting estimates and the convergence of iterates. Denote $\Z_N:=N^{-1}\Z$ and $\mathcal{S}$ as the Schwartz space.

\begin{proposition}[Three Vector Counting] \label{prop-3vc}
Let $1<T\lesssim N^{1-\epsilon}$ for $\epsilon \ll 1$. Then, for $k \in \Z_N \cap (0,1)$ and any choice of $\zeta_1, \zeta_2, \zeta_3 \in \{\pm\}$ and $\lambda \in \R$, the set 
\begin{align}
S_3 = \{(k_1, k_2, k_3) \in (\Z_N \cap [0,1))^3 \colon &\zeta_1 k_1 + \zeta_2k_2 + \zeta_3 k_3 = k \text{ (mod 1)}, \notag \\
&\left|\zeta_1 \omega(k_1) + \zeta_2 \omega(k_2) + \zeta_3 \omega(k_3)- \omega(k) - \lambda\right| \leq T^{-1}\},
\end{align}
satisfies 
\begin{equation}
\sum_{(k_1, k_2, k_3) \in S_3} F(k_1, k_2) \lesssim N^2 T^{-1} \log T,
\end{equation}
where $F(k_1, k_2) = \omega(k_1) \omega(k_2) \omega(k - \zeta_1 k_1 - \zeta_2 k_2)$.
\end{proposition}

\begin{proof}
It suffices to establish the estimate for a modification of the set $S_3$ where we instead impose $\zeta_1 k_1 + \zeta_2 k_2 + \zeta_3 k_3 = k$, for any $k$ equivalent to the original $k$ mod 1. Since $k_1, k_2, k_3 \in \Z_N \cap [0,1)$, there are only finitely many such $k$ we need to consider, in particular $k \in \Z_N \setminus \Z \cap (-3,3)$. Fix $\chi \in \mathcal S(\R)$ which is 1 on $B\left(\sqrt{\frac{m}{4\kappa}}T\lambda, 1\right)$ and such that $\widehat{\chi}$ is supported on a finite ball. Note that we may replace $F$ with a non-negative $\mathcal S(\R^2)$ function supported on $B(0,2)$ which is equal to $F$ on $S_3$. Let $e(x) = \exp(2\pi i x)$ and define 
\begin{equation*}
\Omega(x,y) = \zeta_1 \sin(\pi x) + \zeta_2 \sin(\pi y) + \zeta_3 \sin(\zeta_3\pi (k - \zeta_1 x - \zeta_2 y )) - |\sin(\pi k)|,
\end{equation*}
so that 
\begin{align}
\nabla \Omega(x,y) =  \pi&\left(\zeta_1 \cos(\pi x) - \zeta_1 \cos(\pi \zeta_3( k - \zeta_1 x - \zeta_2y)), \right.\notag\\
&\hspace{3pt}\left.\zeta_2 \cos(\pi y) - \zeta_2\cos(\pi\zeta_3(k - \zeta_1x - \zeta_2y)) \right)
\end{align}
Therefore, 
\begin{align*}
\sum_{(k_1, k_2, k_3) \in S_3}F(k_1, k_2) \leq \sum_{k_1, k_2 \in \Z_N}F(k_1, k_2) \chi(T\Omega(k_1, k_2)). 
\end{align*}
Expressing $\chi$ in terms of its Fourier transform and using Poisson summation, we get 
\begin{align}
\sum_{k_1, k_2 \in \Z_N} F(k_1, k_2) \chi(T\Omega(k_1, k_2)) &= \sum_{k_1, k_2 \in \Z_N} F(k_1, k_2)\int_{-\infty}^{\infty}\hat{\chi}(\tau)e\left(\tau T\Omega(x,y)\right)\diff\tau \\
&=N^2T^{-1}\int_{-\infty}^{\infty}\hat{\chi}\left(\frac{\tau}{T}\right)\int_{(x,y)\in \R^2} F(x,y)e\left(\tau \Omega(x,y)\right)\diff x\diff y\diff\tau \notag  \\
&+N^2T^{-1}\int_{-\infty}^{\infty}\hat{\chi}\left(\frac{\tau}{T}\right)\sum_{(c,d)\in\mathbb{Z}^2\setminus\{(0,0)\}}\int_{(x,y)\in \R^2} F(x,y)\notag \\
&\hspace{3cm} \times e\left(\tau \Omega(x,y)-cNx-dNy\right)\diff x\diff y\diff\tau \label{eq-poisson-exp}\\
&:= I + II. \notag
\end{align}
For $II$, denote the phase by $\Phi(x,y) = \tau \Omega(x,y) - cNx - dNy$, so that for $c \neq 0$
\begin{equation} \label{eq-phase-bound}
\left| \nabla_x \Phi(x,y) \right| \gtrsim N|c|
\end{equation}
and similarly for $\left| \nabla_y \Phi(x,y) \right|$ if $d \neq 0$, since $\tau \lesssim T \lesssim N^{1 - \epsilon}$. Therefore, we my integrate by parts sufficiently many times so that $|II| \ll N^2 T^{-1}$. 

For $I$, we may rewrite it by inverting the Fourier transform in $\chi$, so that 
\begin{equation} \label{eq-counting-main term}
I = N^2 \int_{(x,y) \in \R^2} F(x,y) \chi(T\Omega(x,y))\diff x \diff y.
\end{equation}
Let's assume for now that $\zeta_3 = +$. Defining $z:= k - \zeta_1 x - \zeta_2y$, we may consider the integral over the following two sets (since the $F$ is supported on $B(0,2)$: 
\begin{align}
E_{<}&:= \{(x,y) \in B(0,2) \colon |z - y - 2n| < T^{-1} \text{ for all } n \in \Z\}, \\
E_{\geq} &:= B(0,2) \setminus E_{<}. 
\end{align}
On $E_{<}$, we may use that the size of this set is $\lesssim T^{-1},$ so that the corresponding integral on $E_{<}$ is $ \lesssim N^2T^{-1}$. On $E_{\geq}$, note that we may substitute $y$ with $s = \Omega(x,y)$ so that 
\begin{align}
N^2 \int_{(x,y) \in E_{\geq}} F(x,y) \chi(T\Omega(x,y)) \diff x \diff y &= N^2 \int_s \int_{\substack{x \in \R \\ (x,y(x,s)) \in E_{\geq}}} \frac{F(x, y(x,s))\chi(Ts)} 
{\pi\zeta_2(\cos(\pi y) -  \cos(\pi z)))} \diff x \diff s. \label{eq-3vc-main-term}
\end{align}
Note that 
\begin{align*}
\left|\frac{\sin(\pi x) \sin(\pi y) \sin(\pi z)}{\cos(\pi y) - \cos(\pi z)}\right| &= \left|\frac{\sin(\pi x) \sin(\pi y) \sin(\pi z)}{2 \sin\left(\frac{\pi(z + y)}{2} \right) \sin\left( \frac{\pi(z-y)}{2}\right)}\right|\lesssim \left|\frac{\sin(\pi x)}{\sin\left( \frac{\pi(z-y)}{2}\right)}\right|.
\end{align*}
Since we are on $E_{\geq}$,  $|\sin\left(\frac{\pi(z - y)}{2}\right)| \gtrsim T^{-1},$ so 
\begin{align*}
\eqref{eq-3vc-main-term} \lesssim N^2 \log T \int_s \chi(Ts) \diff s \lesssim N^2 T^{-1} \log T.
\end{align*}
\end{proof}

\begin{corollary} {(Convergence of Iterates)}\label{cor-iterates}
For fixed $k\in \mathbb{Z}_N \cap (0,1)$, let $F \in \mathcal{S}(\R^2)$, $\chi \in \mathcal{S}(\R)$ with $\hat{\chi}$ supported in a ball of finite radius. Then for all $\delta>0$ sufficiently small we have 
\begin{align}
&\sum_{\substack{x, y\in \Z_N \cap (0,1)}} F(x, y)\chi(T\Omega(x, y)) \notag\\
&\hspace{50pt}= N^2\int_{\substack{x, y\in (0,1)}} F(x, y)\chi(T\Omega(x, y))\diff x \diff y+ O(N^{2 - \delta}T^{-1}), \label{eq-iterate-conv}
\end{align}
where $\zeta_1, \zeta_2, \zeta_3 \in \{\pm\}$, and
\begin{align*}
\Omega(x, y) &= \zeta_1 \sin(\pi x) +\zeta_2\sin(\pi y)+\zeta_3\sin(\pi(k - x - y))- \sin(\pi k).
\end{align*}
\end{corollary}
    
\begin{proof}
Note that the expansion \eqref{eq-poisson-exp} is still valid, where term $I$ is precisely the right hand side of \eqref{eq-iterate-conv} as seen in \eqref{eq-counting-main term}. For term $II$, note that we may still use \eqref{eq-phase-bound} to integrate sufficiently many terms so that $|II| \lesssim N^{2 - \delta}T^{-1}$.
\end{proof}
    
\begin{proposition}[Two Vector Counting] \label{vector_counting-2}
Let $1<T\lesssim N^{1-\epsilon}$ for $\epsilon \ll 1$. Then for $k\in \mathbb{Z}_N \cap (0,1)$ and $\lambda\in \mathbb{R}$, the sets:
\begin{align}
S_2^{\pm} &= \Big\{(k_1,k_2)\in \mathbb{Z}_N^2\cap [0,1)^2: k_1\pm k_2= k \text{ (mod 1)},\notag\\
&\left|\omega(k_1)\pm\omega(k_2) - \omega(k) - \lambda \right| \leq T^{-1}\Big\}.
\end{align}
satisfy the bounds:
\begin{align}
\sum_{\substack{(k_1,k_2)}\in S_2^{+}}G(k_1) &\lesssim NT^{-\frac{1}{2}}, \label{eq-2vc+}\\
\sum_{\substack{(k_1, k_2)}\in S_2^{-}}G(k_1) &\lesssim \min(N, NT^{-1}|k|^{-1}), \label{eq-2vc-}
\end{align}
where 
\begin{align}
    G(x) &= \omega(x) \omega(k - x).
\end{align}
\end{proposition}
\begin{proof}
As in the proof of Proposition \ref{prop-3vc}, we may separately consider a modification of the set(s) $S_2^{\pm}$, where we consider a single $k \in \Z_N \cap (-2,2)$ equivalent to the original $k$ mod 1. Similarly, we fix $\chi \in \mathcal S(\R)$ which is 1 on $B\left(\sqrt{\frac{m}{4\kappa}}T\lambda, 1\right)$ such that $\widehat\chi$ is supported on a finite ball and replace $G$ with a non-negative $\mathcal S(\R)$ function supported on $B(0,2)$ equal to $G$ on $S_2^\pm$. We also define $e(x) = \exp(2\pi i x)$ and 

\begin{equation*}
\Omega(x) = \sin(\pi x) + \sin (\pi (k - x)) - |\sin(\pi k)|
\end{equation*}
so that 
\begin{equation}
\Omega'(x) = -\pi \left( \cos(\pi x) - \cos(\pi(k - x))\right) = -2\pi \sin\left(\frac{\pi k}{2}\right) \sin \left(\frac{\pi(2x-k)}{2}\right). 
\end{equation}
Note that we only need to define $\Omega(x)$ with a plus sign, since if $x - y = k$ then $y = -(k - x)$ and we may use that the sine function is odd.  

Writing $\chi$ in terms of its Fourier transform and applying Poisson summation as in \ref{prop-3vc},  
\begin{align}
    \sum_{\substack{(k_1, k_2)}\in S_2^{\pm}}G(k_1)  \leq& \sum_{k_1} G(k_1)\chi\left(T\Omega(k_1)\right) \notag\\
    =&NT^{-1}\int_{-\infty}^{\infty}\hat{\chi}\left(\frac{\tau}{T}\right)\int_{x\in \R} G(x)e\left(\tau \Omega(x)\right)\diff x \diff d\tau \notag\\
    &+NT^{-1}\int_{-\infty}^{\infty}\hat{\chi}\left(\frac{\tau}{T}\right)\sum_{c\in\mathbb{Z}\setminus\{0\}}\int_{x\in \R} G(x)e\left(\tau \Omega(x)-cNx\right)\diff x \diff\tau \notag\\
    :=&I + II. \notag
\end{align}

Just as in the proof of Proposition \ref{prop-3vc}, denoting the phase by $\Phi(x) = \tau \Omega(x,y) -cNx$, we observe that for $c \neq 0, $ 
\begin{equation*}
|\Phi'(x)| \gtrsim N|c|
\end{equation*}
since $\tau \lesssim T \lesssim N^{2 - \epsilon}$. Integrating by parts sufficiently many times we obtain $|II| \ll N^2T^{-1}$. 

For $I$, we may rewrite it by inverting the Fourier transform in $\chi$, so that 
\begin{equation}
I = N \int_{x \in \R} G(x) \chi(T\Omega(x)) \diff x.
\end{equation}

We separately consider the integral over the following two sets for any $D \in \R$ (possibly depending on $N$ or $T$):
\begin{align}
E_{<}(D) &= \{x \in B(0,2) \colon |2x - k - n| \lesssim D \text{ for all } n \in \Z\} \\
E_{\geq}(D) &= B(0,2) \setminus E_{<}
\end{align} 

On $E_{<}(D),$ we may bound the corresponding integral by $ND$. On $E_{\geq}$, we perform the substitution $s = \Omega(x),$ so that 

\begin{align*}
N \int_{\substack{x \in \R \\ x \in E_\geq(D)}} G(x) \chi(T\Omega(x)) \diff x &= N \int_{\substack{s \in \R \\ x(s) \in E_\geq(D)}} \frac{G(x(s))}{\Omega'(x(s))} \chi(Ts)\diff s.
\end{align*}

One can verify
\begin{align*}
\frac{\sin(\pi x) \sin(\pi (k - x)}{\cos(\pi x) - \cos(\pi (k - x))} = \sin\left(\frac{\pi x}{2}\right) \cos \left(\frac{\pi(k - x)}{2}\right) \left[ \frac{1}{\sin\left(\frac{\pi k}{2}\right)} - \frac{1}{\sin\left(\frac{\pi (2x - k)}{2}\right)}\right],
\end{align*}
so that 
\begin{equation} \label{eq-bound_+}
\left|\frac{G(x(s))}{\Omega'(x(s))}\right| \lesssim \frac{1}{|\sin(\pi k/2)|} + \frac{1}{|\sin(\pi(2x - k)/2)|}.
\end{equation}

To obtain \eqref{eq-2vc+} for $S_2^+$, we take $D = T^{-1/2},$ so that  
\begin{equation*}
|I| \lesssim NT^{-1}k^{-1} + NT^{-1/2}.
\end{equation*}
To discount the first term, note that it comes from the first term of \eqref{eq-bound_+} which blows up when $|k - 2m|$ is very small for some $m \in \Z$. However, in this case, we may ignore the analysis above and use the fact that $k_1, k_2 \in \Z_N \cap [0,1)$ so that if $|k| \lesssim T^{-1/2}$ (or $|k - 2| \lesssim {T^{-1/2}}$), both $k_1, k_2$ are restricted to an interval of length $O(T^{-1/2})$ around 0 (or 1). 

To obtain \eqref{eq-2vc-}, we take $D = |k|^{-1}$ (where we now use the original $k \in (0,1)$), so that $|I| \lesssim NT^{-1}|k|^{-1}$. Note that the alternate bound of $N$ corresponds to the trivial counting. 
\end{proof}

\section{Irregular Chains \& Self Loops}
\label{section-irrchains}
\subsection{Double Bonds and Chains}\label{subsec-splicing}

\begin{definition}
Let $\Q$ be an enhanced couple such that $\mathbb{M}(\Q)$ has two atoms, $v_1$ and $v_2$, connected by precisely two edges oriented in opposite directions. Denote by $\mathfrak{n}_j = \mathfrak{n}(v_j)$ for $j=1,2$. Then, up to symmetry, exactly one of the following two scenarios arises:

\begin{enumerate}[label=(\roman*)]
\item \textbf{Cancellation (CL) Double Bond:} There exists a child $\mathfrak{n}_{12}$ of $\mathfrak{n}_1$ which is paired with a child $\mathfrak{n}_{21}$ of $\mathfrak{n}_2$. 
Moreover, $\mathfrak{n}_2$ is itself a child of $\mathfrak{n}_1$ such that $\mathfrak{n}_{12}$ and $\mathfrak{n}_2$ have opposite signs. 
All other leaf-children of $\mathfrak{n}_1$ and $\mathfrak{n}_2$ remain unpaired. 
In the corresponding molecular structure, this configuration represents a double bond composed of one LP bond and one PC bond.

\item \textbf{Connectivity (CN) Double Bond:} There are children $\mathfrak{n}_{11}, \mathfrak{n}_{12}$ of $\mathfrak{n}_1$ with opposite signs, as well as children $\mathfrak{n}_{21}, \mathfrak{n}_{22}$ of $\mathfrak{n}_2$ having opposite signs. These four children pair off according to their signs, leaving any remaining leaf-children of $\mathfrak{n}_1$ and $\mathfrak{n}_2$ unpaired. In addition, neither $\mathfrak{n}_1$ nor $\mathfrak{n}_2$ is a child of the other in this scenario. In the molecular structure, this corresponds to a double bond in which both bonds are LP.
\end{enumerate}
We say that atoms $v_1$ and $v_2$ are connected by a CL or CN bond, respectively. 
\end{definition}
\begin{proof}
See \cite{WKE2023} and \cite{ODW}. 
\end{proof}

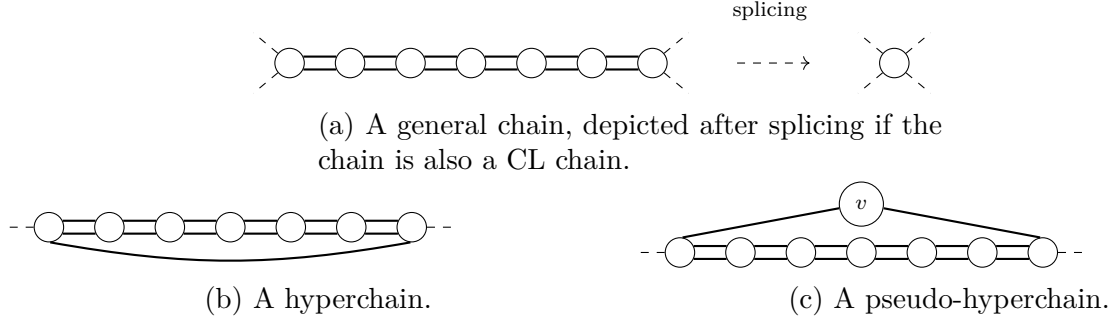
\begin{figure}
\begin{subfigure}[t]{0.5\linewidth}
\hspace{-1cm}
\begin{tikzpicture}[scale = .8]
    \tikzstyle{every node} = [circle, draw = black]
    \node (1) at (-10,0) {};
    \node (2) at (-9,0) {};
    \node (3) at (-8,0) {};
    \node (4) at (-7,0) {};
    \node (5) at (-6,0) {};
    \node (6) at (-5,0) {};
    \node (7) at (-4,0) {};

    \draw[thick] (1.25) -- (2.155);
    \draw[thick] (2.205) -- (1.335);
    \draw[thick] (2.25) -- (3.155);
    \draw[thick] (3.205) -- (2.335);
    \draw[thick] (3.25) -- (4.155);
    \draw[thick] (4.205) -- (3.335);
    \draw[thick] (4.25) -- (5.155);
    \draw[thick] (5.205) -- (4.335);
    \draw[thick] (5.25) -- (6.155);
    \draw[thick] (6.205) -- (5.335);
    \draw[thick] (6.25) -- (7.155);
    \draw[thick] (7.205) -- (6.335);

    \draw[dashed] (1) -- (-10.6,.5);
    \draw[dashed] (1) -- (-10.6,-.5);

    \draw[dashed] (7) -- (-3.4,.5);
    \draw[dashed] (7) -- (-3.4,-.5);

    \draw[dashed, ->] (-2.6,0) -- node[above, pos = .45,draw = none]{\tiny{splicing}} (-1.4,0);

    \node (8) at (0,0) {};
    \draw[dashed] (8) -- (-.6,.5);
    \draw[dashed] (8) -- (-.6,-.5);
    \draw[dashed] (8) -- (.6,.5);
    \draw[dashed] (8) -- (.6,-.5);
\end{tikzpicture}
    \caption{A general chain, depicted after splicing if the chain is also a CL chain.}
\end{subfigure}
\begin{subfigure}[t]{0.5\linewidth}
\begin{tikzpicture}[scale = .8]
    \tikzstyle{every node} = [circle, draw = black]
    \node (1) at (-5,0) {};
    \node (2) at (-4,0) {};
    \node (3) at (-3,0) {};
    \node (4) at (-2,0) {};
    \node (5) at (-1,0) {};
    \node (6) at (0,0) {};
    \node (7) at (1,0) {};

    \draw[thick] (1.25) -- (2.155);
    \draw[thick] (2.205) -- (1.335);
    \draw[thick] (2.25) -- (3.155);
    \draw[thick] (3.205) -- (2.335);
    \draw[thick] (3.25) -- (4.155);
    \draw[thick] (4.205) -- (3.335);
    \draw[thick] (4.25) -- (5.155);
    \draw[thick] (5.205) -- (4.335);
    \draw[thick] (5.25) -- (6.155);
    \draw[thick] (6.205) -- (5.335);
    \draw[thick] (6.25) -- (7.155);
    \draw[thick] (7.205) -- (6.335);

    \draw[thick] (1.275) to [out=-10,in=-170] (7.265);

    \draw[dashed] (1) -- (-5.7,0);
    \draw[dashed] (7) -- (1.7,0);
    
\end{tikzpicture}
\caption{A hyperchain.}
\label{pic-chain-hyper}
\end{subfigure}
~
\begin{subfigure}[t]{0.5\linewidth}
\begin{tikzpicture}[scale = .8]
    \tikzstyle{every node} = [circle, draw = black]
    \node (1) at (-5,0) {};
    \node (2) at (-4,0) {};
    \node (3) at (-3,0) {};
    \node (4) at (-2,0) {};
    \node (5) at (-1,0) {};
    \node (6) at (0,0) {};
    \node (7) at (1,0) {};
    \node (8) at (-2,.8) {\tiny$v$};

    \draw[thick] (1.25) -- (2.155);
    \draw[thick] (2.205) -- (1.335);
    \draw[thick] (2.25) -- (3.155);
    \draw[thick] (3.205) -- (2.335);
    \draw[thick] (3.25) -- (4.155);
    \draw[thick] (4.205) -- (3.335);
    \draw[thick] (4.25) -- (5.155);
    \draw[thick] (5.205) -- (4.335);
    \draw[thick] (5.25) -- (6.155);
    \draw[thick] (6.205) -- (5.335);
    \draw[thick] (6.25) -- (7.155);
    \draw[thick] (7.205) -- (6.335);
    \draw[thick] (1.80) -- (8);
    \draw[thick] (8) -- (7.100);

    \draw[dashed] (1) -- (-5.7,0);
    \draw[dashed] (7) -- (1.7,0);
\end{tikzpicture}
\caption{A pseudo-hyperchain.}
\label{pic-chain-pseudohyper}
\end{subfigure}
\caption{An example of several types of chains with $q = 6$.}
\label{fig-chain}
\end{figure}

\begin{figure}[t]
\centering
\begin{tikzpicture}[scale = .9]
    \tikzstyle{n node} = [circle, draw = black]
    \tikzstyle{o node} = [inner sep=1,outer sep=0]

    \node[n node] (1a) at (-2,0) {};
    \node[n node] (2a) at (-1,0) {};
    \node[o node] (3a) at (0,0) {$\cdots$};
    \node[n node] (4a) at (1,0) {};
    \node[n node] (5a) at (2,0) {};

    \node[n node] (1b) at (-2,-1) {};
    \node[n node] (2b) at (-1,-1) {};
    \node[o node] (3b) at (0,-1) {$\cdots$};
    \node[n node] (4b) at (1,-1) {};
    \node[n node] (5b) at (2,-1) {};

    \node[n node] (1c) at (-2,-2) {};
    \node[n node] (2c) at (-1,-2) {};
    \node[o node] (3c) at (0,-2) {$\cdots$};
    \node[n node] (4c) at (1,-2) {};
    \node[n node] (5c) at (2,-2) {};

    \node[o node]  at (-2,-1.5) {$\cdots$};
    \node[o node]  at (2,-1.5) {$\cdots$};

    \draw[thick] (1a.25) -- (2a.155);
    \draw[thick] (2a.205) -- (1a.335);
    \draw[thick] (2a.25) -- (3a.165);
    \draw[thick] (3a.195) -- (2a.335);
    \draw[thick] (3a.15) -- (4a.155);
    \draw[thick] (4a.205) -- (3a.345);
    \draw[thick] (4a.25) -- (5a.155);
    \draw[thick] (5a.205) -- (4a.335);

    \draw[thick] (1b.25) -- (2b.155);
    \draw[thick] (2b.205) -- (1b.335);
    \draw[thick] (2b.25) -- (3b.165);
    \draw[thick] (3b.195) -- (2b.335);
    \draw[thick] (3b.15) -- (4b.155);
    \draw[thick] (4b.205) -- (3b.345);
    \draw[thick] (4b.25) -- (5b.155);
    \draw[thick] (5b.205) -- (4b.335);

    \draw[thick] (1c.25) -- (2c.155);
    \draw[thick] (2c.205) -- (1c.335);
    \draw[thick] (2c.25) -- (3c.165);
    \draw[thick] (3c.195) -- (2c.335);
    \draw[thick] (3c.15) -- (4c.155);
    \draw[thick] (4c.205) -- (3c.345);
    \draw[thick] (4c.25) -- (5c.155);
    \draw[thick] (5c.205) -- (4c.335);

    \draw[thick] (1a) -- (1b);
    \draw[thick] (5a) -- (5b);

    \draw[dashed] (1a) -- (-2,.5);
    \draw[dashed] (5a) -- (2,.5);
    \draw[dashed] (1c) -- (-2,-2.5);
    \draw[dashed] (5c) -- (2,-2.5);

    \draw[dashed] (1b) -- (-2,-1.4);
    \draw[dashed] (5b) -- (2,-1.4);
    \draw[dashed] (1c) -- (-2,-1.6);
    \draw[dashed] (5c) -- (2,-1.6);
    
\end{tikzpicture}
\caption{A wide ladder, with 3 rungs shown. Note that each rung may have a different length.}
\label{fig-ladder}
\end{figure}
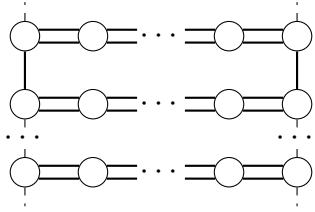

\begin{definition}[Chains]\label{def-chain}
Let $\Q$ be an enhanced couple, and let $\M(\Q)$ be its corresponding molecule. A \emph{chain} in $\M(\Q)$  is a sequence of atoms  $(v_0, \ldots, v_q)$ such that each consecutive pair $v_i,v_{i+1}$ (for $0 \leq i \leq q-1$) is connected by either a CL or CN double bond. We refer to $q$ as the \textit{length} of the chain. We define the following special types of chains:

\begin{enumerate}
    \item A \textit{hyperchain} is a chain where $v_0$ and $v_q$ are also joined by an additional bond, which is not part of a CL or CN double bond. 
    \item A \textit{pseudo-hyperchain} is a chain where $v_0$ and $v_q$ are each connected to an atom $v$ not in the chain via single bonds. 
    \item A \textit{wide ladder} is a collection of chains ${(v_0^{(1)}, \ldots, v_{q^{(1)}}^{(1)}), \ldots, (v_0^{(m)}, \ldots, v_{q^{(m)}}^{(m)})}$ such that for each $1 \leq i \leq m - 1$, there is a bond connecting the vertices $v_0^{(i)}$ and $v_0^{(i+1)}$, as well as a bond connecting the vertices $v_{q^{(i)}}^{(i)}$ and $v_{q^{(i+1)}}^{(i+1)}$. Each chain in the wide ladder is called a \textit{rung}.
    \item An \emph{irregular chain} is a chain whose double bonds are all CL.
    \item A \emph{maximal chain} is a chain for which $v_0$ and $v_q$ do not connect via CL or CN double bonds to any atoms outside the chain.
\end{enumerate}
We similarly define \textit{maximal hyperchains, pseudo-hyperchains, wide ladders, and irregular chains}. See Figures \ref{fig-chain} and \ref{fig-ladder}.
\end{definition}

\begin{definition}[Irregular Chains] \label{def-irrchain}
Note that we may similarly refer to chains at the level of the couple $\Q$. In this case, an irregular chain refers to a sequence of nodes $(\mathfrak n_0, \ldots, \mathfrak n_q)$ such that for $j = 0, \ldots, q-1$, $\mathfrak n_j$ has $\mathfrak n_{j + 1}$ as a child and another child $\mathfrak m_{j + 1}$ (of opposite sign to $\mathfrak n_{j + 1}$), which is paired to a child $\mathfrak p_{j + 1}$ of $\mathfrak n_{j + 1}$. Denote the remaining child of $\mathfrak n_0$ to be $\mathfrak n_0^1$ and the remaining children of $\mathfrak n_q$ to be $\mathfrak n_0^{2}$ and $\mathfrak n_0^{3}$. Note that for any decoration, we must have 
\begin{equation}\label{eq-gap}
k_{\mathfrak n_0} - \zeta_{\mathfrak n_0^1} k_{\mathfrak n_0^1} = k_{\mathfrak n_1} - k_{\mathfrak m_1} = \ldots = k_{\mathfrak n_q} - k_{\mathfrak m_q} = \zeta_{\mathfrak n_0^2}k_{\mathfrak n_0^2} + \zeta_{\mathfrak n_0^3}k_{\mathfrak n_0^3} : = h.
\end{equation}
We refer to $h$ as the \textit{gap} of the irregular chain and refer to the chain as a small gap (SG) irregular chain or large gap (LG) irregular chain if $h$ is small gap or large gap respectively, according to Definition \ref{def-gap}.

Suppose $\Q$ is an enhanced couple with irregular chain $(\mathfrak n_0, \ldots, \mathfrak n_q)$. We define $\Q^{\mathrm{sp}}$ by removing node $\mathfrak n_i, \mathfrak m_i, \mathfrak p_i$ for $i = 1, \ldots, q$ and giving $\mathfrak n_0$ children $\mathfrak n_0^1, \mathfrak n_0^{2}, \mathfrak n_0^{3}$, with $\mathfrak n_0^1$ retaining its position and $\mathfrak n_0^{2}, \mathfrak n_0^{3}$ keeping their relative positions. We refer to this as \textit{splicing} the couple at nodes $\mathfrak n_1, \ldots, \mathfrak n_q$ or at chain $(\mathfrak n_0, \ldots, \mathfrak n_q)$. We may correspondingly define $\M^{\mathrm{sp}}$. See Figure \ref{fig-irregular-chain}. We define a \textit{unit twist} of $\Q$ at node $\mathfrak n_i (i = 1, \ldots, q)$ by swapping $\mathfrak n_i$ and $\mathfrak m_i$. This results in a 1-1 correspondence of their decorations if we swap the decoration of $\mathfrak n_1$ and $\mathfrak m_1$ as well as change the decoration of $\mathfrak p_1$ so that it is still paired to $\mathfrak m_1$. If a unit twist can be performed at a node $\mathfrak n$, we call it \textit{twist-admissible}. We say that enhanced couples $\Q$ and $\Q'$ are \textit{congruent} ($\Q \sim \Q'$) if $Q'$ can be formed from $Q$ via unit twists at twist-admissible nodes (and vice versa). For an enhanced couple $\Q$ and set $\mathcal M \subset \mathcal N$, denote by $\textgoth{Q}_{\mathcal M}$:
\begin{equation*}
\mathcal {\textgoth{Q}_\mathcal M} = \{\mathcal Q' | \mathcal Q' \sim \mathcal Q \text{ via } \mathcal M' \in 2^{\mathcal M}\}.
\end{equation*}
Note that when $\mathcal M = \{\mathfrak n_1, \ldots \mathfrak n_q\}$ for a chain $(\mathfrak n_0, \ldots \mathfrak n_q)$, each twist of $\Q$ can be identified by the signs of $\zeta_{\mathfrak n_i}$. 
\end{definition}

\begin{remark}
Note that if $\Q$ and $\Q'$ are congruent, then $\Q^{\mathrm{lf}}$ and $(\Q')^{\mathrm{lf}}$ are also congruent.
\end{remark}

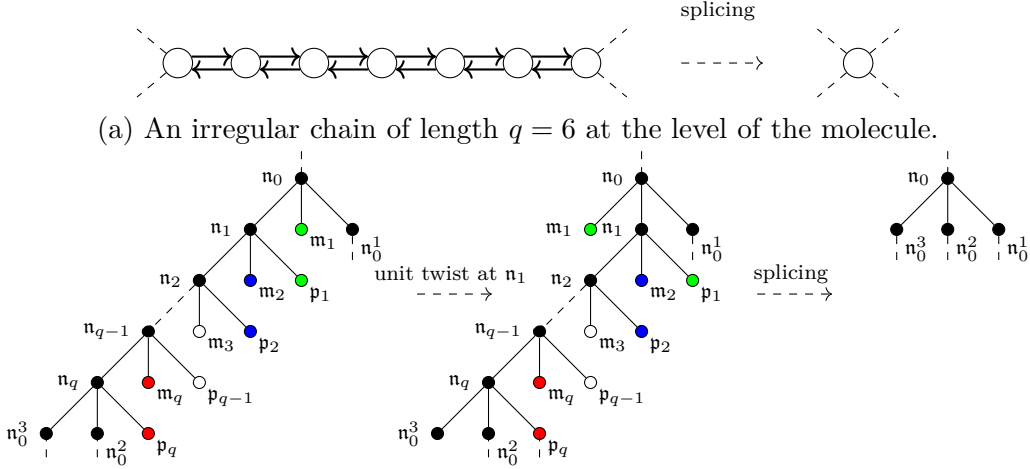
\begin{figure}
\begin{subfigure}[t]{0.9\linewidth}
\centering
\begin{tikzpicture}[scale = .9]
    \tikzstyle{every node} = [circle, draw = black]
    \node (1) at (-10,0) {};
    \node (2) at (-9,0) {};
    \node (3) at (-8,0) {};
    \node (4) at (-7,0) {};
    \node (5) at (-6,0) {};
    \node (6) at (-5,0) {};
    \node (7) at (-4,0) {};

    \draw[thick, ->] (1.25) -- (2.155);
    \draw[thick, ->] (2.205) -- (1.335);
    \draw[thick, ->] (2.25) -- (3.155);
    \draw[thick, ->] (3.205) -- (2.335);
    \draw[thick, ->] (3.25) -- (4.155);
    \draw[thick, ->] (4.205) -- (3.335);
    \draw[thick, ->] (4.25) -- (5.155);
    \draw[thick, ->] (5.205) -- (4.335);
    \draw[thick, ->] (5.25) -- (6.155);
    \draw[thick, ->] (6.205) -- (5.335);
    \draw[thick, ->] (6.25) -- (7.155);
    \draw[thick, ->] (7.205) -- (6.335);

    \draw[dashed] (1) -- (-10.6,.5);
    \draw[dashed] (1) -- (-10.6,-.5);

    \draw[dashed] (7) -- (-3.4,.5);
    \draw[dashed] (7) -- (-3.4,-.5);

    \draw[dashed, ->] (-2.6,0) -- node[above, pos = .45,draw = none]{\tiny{splicing}} (-1.4,0);

    \node (8) at (0,0) {};
    \draw[dashed] (8) -- (-.6,.5);
    \draw[dashed] (8) -- (-.6,-.5);
    \draw[dashed] (8) -- (.6,.5);
    \draw[dashed] (8) -- (.6,-.5);
\end{tikzpicture}
    \caption{An irregular chain of length $q = 6$ at the level of the molecule.}
    \label{fig-irregular-chain-mol}
\end{subfigure}
\begin{subfigure}[t]{0.9\linewidth}
\centering
\begin{tikzpicture}[scale = .45,
level distance=1.5cm,
  level 1/.style={sibling distance=1.5cm},
  level 2/.style={sibling distance=1.5cm}]
\tikzstyle{hollow node}=[circle,draw,inner sep=1.6]
\tikzstyle{solid node}=[circle,draw,inner sep=1.6,fill=black]
\tikzset{
red node/.style = {circle,draw=black,fill=red,inner sep=1.6},
blue node/.style= {circle,draw = black, fill= blue,inner sep=1.6}, green node/.style = {circle,draw=black,fill=green,inner sep=1.6}}, 
[edge from parent/.style={draw,dashed}]

\node[solid node, label = left: {\tiny $\mathfrak n_0$}] at (4,.4){}
    child{node[solid node, label = left: {\tiny $\mathfrak n_1$}]{}
        child{node[solid node, label = left: {\tiny $\mathfrak n_2$}]{}
            child[dashed]{node[solid node, label = left: {\tiny $\mathfrak n_{q-1}$}]{}
                child[solid]{node[solid node, label = left: {\tiny $\mathfrak n_{q}$}]{}
                    child{node[solid node, label = left: {\tiny $\mathfrak n_0^{3}$}]{}}
                    child{node[solid node, label={[label distance=-1.5mm] 300:\tiny $\mathfrak n_0^{2}$}]{}}
                    child{node[red node, label={[label distance=-1.5mm] 300:\tiny $\mathfrak p_q$}]{}}
                    }
                child[solid]{node[red node, label={[label distance=-1.5mm] 300:\tiny $\mathfrak m_{q}$}]{}}
                child[solid]{node[hollow node, label={[label distance=-1.5mm] 300:\tiny $\mathfrak p_{q-1}$}]{}}
                }
            child{node[hollow node, label={[label distance=-1.5mm] 300:\tiny $\mathfrak m_3$}]{}}
            child{node[blue node, label={[label distance=-1.5mm] 300:\tiny $\mathfrak p_2$}]{}}
            }
        child{node[blue node, label={[label distance=-1.5mm] 300:\tiny $\mathfrak m_2$}]{}}
        child{node[green node, label={[label distance=-1.5mm] 300:\tiny $\mathfrak p_1$}]{}}
    }
    child{node[green node, label={[label distance=-1.5mm] 300:\tiny $\mathfrak m_1$}]{}}
    child{node[solid node, label={[label distance=-1.5mm] 300:\tiny $\mathfrak n_0^1$}]{}}
;
\draw[dashed] (4,1.2) -- (4,.4);
\draw[dashed] (5.5,-2) -- (5.5,-1.2);
\draw[dashed] (-2,-7.8) -- (-2,-7);
\draw[dashed] (-3.5,-7.8) -- (-3.5,-7);

\draw[dashed, ->] (7.4,-3) -- node[above, pos = .45,draw = none]{\tiny{unit twist at $\mathfrak n_1$}}(9.6, -3);

\node[solid node, label = left: {\tiny $\mathfrak n_0$}] at (14,.4){}
    child{node[green node, label= left:{\tiny $\mathfrak m_1$}]{}}
    child{node[solid node, label = left: {\tiny $\mathfrak n_1$}]{}
        child{node[solid node, label = left: {\tiny $\mathfrak n_2$}]{}
            child[dashed]{node[solid node, label = left: {\tiny $\mathfrak n_{q-1}$}]{}
                child[solid]{node[solid node, label = left: {\tiny $\mathfrak n_{q}$}]{}
                    child{node[solid node, label = left: {\tiny $\mathfrak n_0^{3}$}]{}}
                    child{node[solid node, label={[label distance=-1.5mm] 300:\tiny $\mathfrak n_0^{2}$}]{}}
                    child{node[red node, label={[label distance=-1.5mm] 300:\tiny $\mathfrak p_q$}]{}}
                    }
                child[solid]{node[red node, label={[label distance=-1.5mm] 300:\tiny $\mathfrak m_{q}$}]{}}
                child[solid]{node[hollow node, label={[label distance=-1.5mm] 300:\tiny $\mathfrak p_{q-1}$}]{}}
                }
            child{node[hollow node, label={[label distance=-1.5mm] 300:\tiny $\mathfrak m_3$}]{}}
            child{node[blue node, label={[label distance=-1.5mm] 300:\tiny $\mathfrak p_2$}]{}}
            }    
        child{node[blue node, label= {[label distance = -1.5mm] 300:\tiny $\mathfrak m_2$}]{}}
        child{node[green node, label={[label distance=-1.5mm] 300:\tiny $\mathfrak p_1$}]{}}
    }
    child{node[solid node, label={[label distance=-1.5mm] 300:\tiny $\mathfrak n_0^1$}]{}}
;
\draw[dashed] (14,1.2) -- (14,.4);
\draw[dashed] (15.5,-2) -- (15.5,-1.2);
\draw[dashed] (11,-7.8) -- (11,-7);
\draw[dashed] (9.5,-7.8) -- (9.5,-7);

\draw[dashed, ->] (17.4,-3) -- node[above, pos = .45,draw = none]{\tiny{splicing}}(19.6, -3);

\node[solid node, label = left: {\tiny $\mathfrak n_0$}] at (23,.4){}
    child{node[solid node, label = {[label distance = -1.5mm] 300: \tiny $\mathfrak n_0^{3}$}]{}}
    child{node[solid node, label={[label distance=-1.5mm] 300:\tiny $\mathfrak n_0^2$}]{}}
    child{node[solid node, label={[label distance=-1.5mm] 300:\tiny $\mathfrak n_0^{1}$}]{}}
;
\draw[dashed] (23,1.2) -- (23,.4);
\draw[dashed] (24.5,-2) -- (24.5,-1.2);
\draw[dashed] (23,-2) -- (23,-1.2);
\draw[dashed] (21.5,-2) -- (21.5,-1.2);
\end{tikzpicture}
\caption{An irregular chain, along with a unit twist at one node of the chain as well as the resulting couple when both couples are spliced at the irregular chain. The white leaves are paired with leaves in the omitted part rather than each other, so that $\mathfrak m_i$ is paired with $\mathfrak p_{i}$ for $i = 1, \ldots, q$.} 
\end{subfigure}
\caption{Irregular chains at the level of the molecule and couple, before and after splicing.}
\label{fig-irregular-chain}
\end{figure}

\begin{proposition}\label{prop-CN}
Given a molecule $\M(\mathcal Q)$ and chain $(v_0, \ldots, v_{q})$, there can be at most one CN double bond. 
\end{proposition}
\begin{proof}
See \cite{ODW}. 
\end{proof}

\begin{lemma}\label{lem-splice}
Suppose that $\Q$ is an enhanced couple with an irregular chain $(\mathfrak n_0, \ldots, \mathfrak n_{q})$ of length $q$ and gap at most $h$, 
\begin{align}
\K_{\textgoth{Q}_\mathcal M} (t,s,k) 
&= \left( \frac{\beta T}{N}\right)^{n_{\mathrm{sp}}} \hspace{-.25cm}\zeta(\mathcal Q^{\mathrm{sp}}) \sum_{\mathscr E^{\mathrm{sp}}} \tilde \epsilon_{\mathscr E^{\mathrm{sp}}} \int_{\mathcal{E}^{\mathrm{sp}}} P_q(t_{\mathfrak n_0}, t_{\mathfrak n_0^2}, t_{\mathfrak n_0^3}, k[\mathcal N^{\mathrm{sp}}])  \notag\\
& \hspace{4cm}\times  \left( \prod_{\mathfrak n \in \mathcal N^{\mathrm{sp}}}e^{\zeta_{\mathfrak{n}}i\Gamma_{\mathfrak n}(t_{\mathfrak n})} \diff t_{\mathfrak{n}}\right)\diff \tau \prod^+_{\mathfrak l \in \mathcal L^{\mathrm{sp}}} n_{\mathrm{in}}(k_{\mathfrak l}),\label{eq-kqm}
\end{align}
where if $A(t)$, defined in Proposition \ref{prop-phase-renormalization}, satisfies $|\dot A(t)| \lesssim \beta T$ and $|\ddot A(t)| \lesssim T$, then $P_q$ satisfies
\begin{align}
\sup_{|k_{\mathfrak n_0} - k_{\mathfrak n^1_0}| \leq h} \left|\left|\widehat P_q(\tau_0, \tau_1, \tau_2,  k[\mathcal N^{\mathrm{sp}}])\right|\right|_{L^1_{\tau_0, \tau_1, \tau_2}} &\lesssim \log N\left(\beta Th\right)^q. \label{eq-pestimate}
\end{align} 
In the above, define $\mathcal{M} = \{\mathfrak{n}_1, \ldots, \mathfrak{n}_q\}$ and let $\mathcal Q^{\mathrm{sp}}$ denote the couple obtained from $\Q$ by splicing out the irregular chain and with branching nodes $\mathcal N^{\mathrm{sp}}$ and leaves $\mathcal L^{\mathrm{sp}}$. Similarly, let $\mathscr E^{\mathrm{sp}}$ denote a decoration of $\Q^{\mathrm{sp}}$ and $\mathcal E^{\mathrm{sp}}$ denote the domain of integration where we no longer integrate over $t_{\mathfrak n_j}$ ($j = 1, \ldots, q$) and instead impose $t_{\mathfrak n_0} > t_{\mathfrak n_0^2}, t_{\mathfrak n_0^3}$. Furthermore, $\tilde \epsilon_{\mathscr E^{\mathrm{sp}}}$ permits degeneracies at $\mathfrak{n}_0$. 
\end{lemma}

\begin{proof}
Without loss of generality, we assume that $\Q$ is the element of $\textgoth{Q}_{\mathcal M}$ where all $\zeta_{\mathfrak n_j}$ $(j = 0, \ldots, q)$ have the same sign. Note that $\zeta_{\mathfrak n_0^i} (i =1,2,3)$ may have any sign, so we set $\zeta_i = \zeta_{\mathfrak n_0^i}$. Following the notation in Definition \ref{def-chain}, we set $k_{\mathfrak n_i} = k_i$, $k_{\mathfrak p_i} = k_{\mathfrak m_i} = \ell_i$, and $t_i = t_{\mathfrak n_i} (0 \leq i \leq q)$. In addition, we let $\ell_0=k_{\mathfrak n_0^{1}}, k_{q+1} = k_{\mathfrak n_0^{2}},$ and $\ell_{q+1} = k_{\mathfrak n_0^{3}}$, and $t_{q+1} = \text{max}(t_{\mathfrak n_0^{2}}, t_{\mathfrak n_0^{3}})$. We keep all $k_{\mathfrak n}$ and $t_{\mathfrak n}$ for $\mathfrak n \notin \{\mathfrak n_i, \mathfrak p_i, \mathfrak m_i\} (1 \leq i \leq q)$ fixed so that by \eqref{eq-gap} we may write the gap $h$ in terms of $k_0, \ell_0, k_{q+1}, \ell_{q+1}$: $$h:=k_0 - \ell_0 = k_1 - \ell_1 = \ldots = k_q - \ell_q = k_{q+1} - \ell_{q+1}.$$
Throughout, denote $\Omega_j = \Omega_{\mathfrak n_j}$ and $\widetilde \Omega_j = \widetilde \Omega_{\mathfrak n_j}$. Then for $j =1, \ldots, q-1$,
    \begin{align*}
    \Omega_j &:= \omega_{k_{j+1}} -  \omega_{\ell_{j+1}} +\omega_{\ell_{j}} - \omega_{k_{j}}= 4\sin\left(\frac{\pi h}{2}\right)\left[\cos\left(\frac{\pi (k_{j+1}+\ell_{j+1})}{2}\right)-\cos\left(\frac{\pi (k_{j}+\ell_{j})}{2}\right)\right],
    \end{align*}
    Therefore, $|\Omega_j|\lesssim h$. Note that the same is true for $j = 1$ and $q$ despite the fact the signs may be different. Summing over the chain, we get:
    \begin{align*}
    \Omega = \sum_{j = 0}^q \Omega_{j} = 2(\zeta_3\sin(\pi k_{q+1}) + \zeta_2\sin(\pi \ell_{q+1}) + \zeta_1\sin(\pi \ell_{0}) - \sin(\pi k_{0})),
    \end{align*}
    and set $\widetilde \Omega = \sum \widetilde \Omega_j$. Also, denote 
    $
    \epsilon_j = \epsilon_{k_{j + 1}\ell_{j + 1}m_j}, \text{ and }
    \epsilon_j' = \epsilon_j \cdot \sqrt{\frac{\omega_{k_j}\omega_{\ell_j}}{\omega_{k_{j-1}}\omega_{\ell_{j-1}}}}
    $
    so that 
    \begin{align*}
    \prod_{j = 0}^q\epsilon_j = \epsilon_{k_{q + 1}\ell_{q + 1}\ell_0} \prod_{j = 1}^q \epsilon_{j'}.
    \end{align*}
    For each $\mathcal Q' \in \textgoth{Q}_{\mathcal M}$, the difference in $\mathcal K_{\mathcal Q'}$ compared to $\mathcal K_{\mathcal Q}$ can be identified as:
    \begin{align*}
    & \left( \frac{\beta T}{N}\right)^q\sum\limits_{\substack{k_i, \ell_i \in \Z_N \cap (0,1)  \\ k_{i} - \ell_{i}=h \\ i = 1, \ldots, q}}\; \int\limits_{t_{q+1} < t_q < \ldots < t_1 < t_0} \left(\prod_{j = 1}^q i\zeta_{\mathfrak n_j}\right) \left(\prod_{j = 0}^q\epsilon_{k_{j+1}\ell_{j+1}m_{j}}\right)  \nonumber \\
    & \hspace{4.5cm}\times \left( \prod_{j = 0}^q e^{\zeta_{\mathfrak n_j}i( \Omega_jT t_j+ \widetilde{\Omega}_{j}(t_{j}))}\right)\left( \prod_{j = 1}^q n_{\mathrm{in}}(m_j)\right) \diff t_1 \ldots \diff t_q,
    \end{align*}
    where $m_0=\ell_0$, and $m_j = \ell_j$ or $m_j = k_j$ for congruent couples with unit twist performed at $\mathfrak n_j$. A unit twist at node $\mathfrak{n}_j$ reverses the sign of $\zeta_{\mathfrak{n}_j}$ but leaves $\zeta_{\mathfrak{n}_j}\Omega_j$, $\zeta_{\mathfrak{n}_j}\widetilde\Omega_j$, and $\epsilon_j$ unchanged. Hence, when we sum over all couples $\mathcal{Q}' \in \textgoth{Q}_{\mathcal{M}}$, we are effectively summing over every possible choice of $\zeta_{\mathfrak{n}_j}$ for $j=1,\ldots,q$. This yields:
    \begin{align*}
    & \hspace{.4cm}\left( \frac{\beta T}{N}\right)^q\sum\limits_{\substack{\substack{k_i \in \Z_N \cap (0,1) \\ i = 1, \ldots, q}}} \hspace{.2cm}\int\limits_{t_{q+1} < t_q < \ldots < t_1 < t_0}\left(\prod_{j = 0}^q\epsilon_{j}\right) \left( \prod_{j = 0}^q e^{i( \Omega_jT t_j+ \widetilde{\Omega}_{j}(t_{j}))}\right)\\
    &\hspace{8.5cm}\times \left( \prod_{j = 1}^q n_{\mathrm{in}}(k_j - h)-n_{\mathrm{in}}(k_j)\right) \diff t_1 \ldots \diff t_q\\
    &= \epsilon_{k_{q+1}\ell_{q+1}\ell_{0}}\cdot e^{i( \Omega T t_0+ \widetilde{\Omega}(t_{0}))} \sum\limits_{\substack{\substack{k_i \in \Z_N \cap (0,1) \\ i = 1, \ldots, q}}} \hspace{.2cm}\int\limits_{t_{q+1} < t_q < \ldots < t_1 < t_0} \left( \frac{\beta T}{N}\right)^q \left(\prod_{j = 1}^q\epsilon_{j}'\right) \\
    &\hspace{3cm}\times  \left( \prod_{j = 1}^q e^{i( \Omega_jT(t_j - t_0)+ \widetilde{\Omega}_{j}(t_{j})-\widetilde{\Omega}_{j}(t_{0}))}\right)\left( \prod_{j = 1}^q n_{\mathrm{in}}(k_j - h)-n_{\mathrm{in}}(k_j)\right) \diff t_1 \ldots \diff t_q\\
    &= \epsilon_{k_{q+1}\ell_{q+1}\ell_{0}}\cdot e^{i( \Omega T t_0+ \widetilde{\Omega}(t_{0}))} \int\limits_{0 < s_1 < \ldots < s_q < t_0 - t_{q+1}}  \left( \frac{\beta T}{N}\right)^q  \sum\limits_{\substack{\substack{k_i \in \Z_N \cap (0,1) \\ i = 1, \ldots, q}}} \left(\prod_{j = 1}^q\epsilon_{j}'\right)\\
    &\hspace{3cm}\times \left( \prod_{j = 1}^q e^{i( \Omega_jT(s_j)+ \widetilde{\Omega}_{j}(s_{j}+t_0)-\widetilde{\Omega}_{j}(t_{0}))}\right)\left( \prod_{j = 1}^q n_{\mathrm{in}}(k_j - h)-n_{\mathrm{in}}(k_j)\right) \diff t_1 \ldots \diff t_q\\
    & := \epsilon_{k_{q+1}\ell_{q+1}\ell_{0}}\cdot e^{i( \Omega T t_0+ \widetilde{\Omega}(t_{0}))} P_q(t_0, t_{\mathfrak n_0^2}, t_{\mathfrak n_0^3}, k[\mathcal N^{\mathrm{sp}}]).
    \end{align*}
    Note that to show \eqref{eq-pestimate}, it suffices to get a bound uniform in choice of $k_j$ for
    \begin{align*}
    S_q(t_0, t_1, t_2, k[\mathcal N]) &:= \int\limits_{0 < s_1 < \ldots < s_{q} < t_1, t_2} \left( \prod_{j = 1}^q e^{i( \Omega_jT(s_j)+ \widetilde{\Omega}_{j}(s_{j}+t_0)-\widetilde{\Omega}_{j}(t_{0}))}\right)\diff s_1 \ldots \diff s_{q},
    \end{align*}
    where now $t_1 = t_0 - \mathfrak n_0^2, t_2 = t_0 - \mathfrak n_0^3$, as using the regularity of $n_{\mathrm{in}}$ we obtain the bound
    \begin{align*}
    \sup_{|k_0 - \ell_0| \leq h} |\widehat{P}_q(\tau_0, \tau_1, \tau_2, k[\mathcal N^{\mathrm{sp}}])| &\lesssim \sup_{|k_0 - \ell_0| \leq h}\left(\frac{\beta T}{N}\right)^q (Nh)^q \sup_{k_j \in \Z_N \cap (0,1)} |\widehat{S}_q(\tau_0, \tau_1, \tau_2, k[\mathcal N])| \\
    & \lesssim \sup_{k[\mathcal N^{\mathrm{sp}}]} (\beta T h)^q |\widehat{S}_q(\tau_0, \tau_1, \tau_2, k[\mathcal N])|.
    \end{align*}
    Note that
    \begin{align*}
    \widehat{S}_q(\tau_0, \tau_1, \tau_2, k[\mathcal N]) &= \int\limits_{0 < s_1 < \ldots < s_{q} < t_1, t_2; t_0 \in \R} \hspace{-.8cm}e^{-2\pi i (\tau_0t_0 + \tau_1 t_1 + \tau_2 t_2)} \chi(t_0) \chi(t_1) \chi(t_2) \\
    & \hspace{1cm} \times\left( \prod_{j = 1}^q e^{i( \Omega_jT(s_j)+ \widetilde{\Omega}_{j}(s_{j}+t_0)-\widetilde{\Omega}_{j}(t_{0}))} \diff s_j\right)  \diff t_0 \diff t_1 \diff t_2.
    \end{align*}

    We fix $K\gg 1$ sufficiently large and separately consider the cases of $\tau_0, \tau_1, \tau_2 \leq N^K$ and one of $\tau_0, \tau_1, \tau_2$. In the former case, we may apply Proposition \ref{prop-integral_est} to bound 
    \begin{align*}
    |\widehat{S}_q(\tau_0, \tau_1, \tau_2, k[\mathcal N])| \lesssim \langle \tau_1 + \lambda_1 \rangle^{-1} \langle \tau_2 + \lambda_2 \rangle^{-1} \langle \tau_0 + \lambda_0\rangle^{-1}, 
    \end{align*}
    for some $\lambda_0, \lambda_1, \lambda_2 \in \R$ which are linear combinations of $T\Omega_j$s, or $\tau_1, \tau_2$ in the case of $\lambda_0$. Note that there are finitely many such possibilities, so we omit the sum from the above bound. In the latter case, we first integrate in $s_1, \ldots, s_q, t_1, t_2$ and denote the result $\widetilde{S}_q(t_0)$. Then, we may integrate by parts in $t_0$, so that 
    \begin{align*}
    \widehat{S}_q(\tau_0, \tau_1, \tau_2, k[\mathcal N]) = \int_{t_0 \in \R}\frac{-1}{4\pi^2 \tau_0^2} e^{-2\pi i \tau_0 t_0}\frac{d^2}{dt^2}\left[\chi(t_0) \widetilde S_q(t_0)\right]\diff t_0. 
    \end{align*}

    Note that 
    \begin{align*}
    |\widetilde S_q|, |\widetilde S_{q}'(t_0)|, |\widetilde S_q''(t_0)| \leq N^{K/2} \langle \tau_{1}\rangle^{-1} \langle\tau_{2} \rangle^{-1} \langle \tau_1 + \tau_2 + \lambda \rangle^{-1},
    \end{align*}
    for some $\lambda \in \R$ and sufficiently large $K$, where we apply Proposition \ref{prop-integral_est}, integrating in the order $t_1, t_2, s_q, \ldots, s_1$. Therefore, after integrating in $t_0$,
    \begin{align*}
    |\widehat{S}_q(\tau_0, \tau_1, \tau_2, k[\mathcal N])| \lesssim N^{-K/2} \langle \tau_0 \rangle^{-2} \langle \tau_{1}\rangle^{-1} \langle\tau_{2} \rangle^{-1} \langle \tau_1 + \tau_2 + \lambda \rangle^{-1}.
    \end{align*}

    Therefore, we have
    \begin{align*}
    &\| \widehat{S_q}(\tau_0, \tau_1, \tau_2, k[\mathcal N]) \|_{L^1_{\tau_0, \tau_1, \tau_2}} \\
    \lesssim &\int_{\tau_0, \tau_1, \tau_2 \leq N^K} \langle \tau_0  + \lambda_0 \rangle^{-1} \langle \tau_{1} + \lambda_1\rangle^{-1} \langle \tau_2 + \lambda_2 \rangle^{-1} \diff \tau_0 \diff \tau_1\diff \tau_2 \\
    & \hspace{1cm} + \int_{\tau_0 \text{, } \tau_1 \text{ or } \tau_2 \geq N_K} N^{K/2}\langle \tau_0  + \lambda_0 \rangle^{-2} \langle \tau_1\rangle^{-1} \langle \tau_2\rangle^{-1} \langle \tau_1 + \tau_2 + \lambda \rangle^{-1} \diff \tau_0 \diff \tau_1 \diff \tau_2 \\
    \lesssim &\log N + N^{-K/2}.
    \end{align*}
    As $q$ is finite, we may choose $K$ appropriately large to recover \eqref{eq-pestimate}. Note that \eqref{eq-kqm} follows automatically.
\end{proof}

\subsection{Self Loops} \label{subsec-loops}
\begin{definition}
Let $\Q$ be an enhanced couple. We say that $\Q$ has a \textit{self-loop} at branching node $\mathfrak n$ if $\mathfrak n$ has two children of opposite signs which are paired as leaves. If $\Q$ has no self-loops, we refer to it as \textit{loop free}. We call branching nodes $(\mathfrak n_1, \ldots, \mathfrak n_q)$ a \textit{self-loop chain} if $\Q$ has a self-loop at $\mathfrak n_j$ $(j = 1, \ldots, q)$ and $\mathfrak n_{j + 1}$ is a child of $\mathfrak n_j$ $(j = 1, \ldots, q-1)$, where $q$ denotes the \textit{length} of the chain. We call $\mathfrak n_1$ the \textit{root} of the chain and refer to the third child of $\mathfrak n_{q}$ as $\mathfrak n_{q + 1}$ and the parent of $\mathfrak n_1$, if it exists, as $\mathfrak n_0$. A \textit{maximal self-loop chain} is a self-loop chain for which neither $\mathfrak n_0$ nor $\mathfrak n_{q +1}$ have a self-loop. In the corresponding molecule $\M = \M(\Q)$, we may denote $v_j = \mathfrak n_{j}$ and refer to $(v_1, \ldots, v_q)$ as a self-loop chain at the level of the molecule. See Figure \ref{fig-loop}.
\end{definition}

\begin{remark}
Due to our phase renormalization and corresponding definition of enhanced couples, note that self-loops can only occur at (1,2) nodes. For enhanced couple $\Q$ and self-loop chain $(\mathfrak n_1, \ldots, \mathfrak n_q)$, this implies that $\mathfrak n_j$ and $\mathfrak n_{j + 1}$ must have opposite sign. Also, note that the set of maximal self-loop chains of $\Q$ is well-defined with each maximal self-loop chain disjoint.
\end{remark}

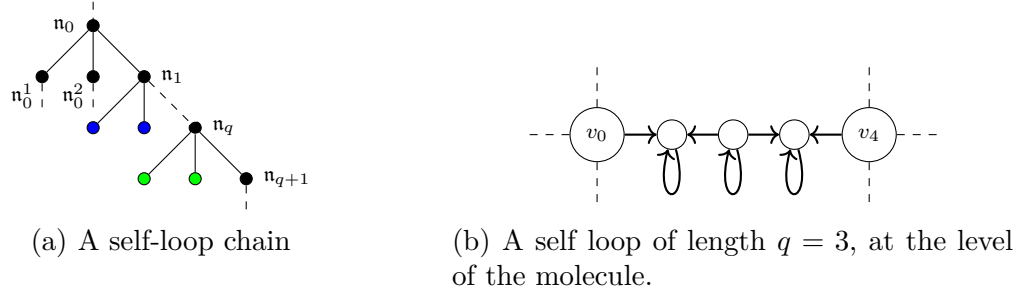
\begin{figure}
\begin{subfigure}[t]{0.45\linewidth}
\centering
\begin{tikzpicture}[scale = .45,
level distance=1.5cm,
  level 1/.style={sibling distance=1.5cm},
  level 2/.style={sibling distance=1.5cm}]
\tikzstyle{hollow node}=[circle,draw,inner sep=1.6]
\tikzstyle{solid node}=[circle,draw,inner sep=1.6,fill=black]
\tikzset{
red node/.style = {circle,draw=black,fill=red,inner sep=1.6},
blue node/.style= {circle,draw = black, fill= blue,inner sep=1.6}, green node/.style = {circle,draw=black,fill=green,inner sep=1.6}}, 
[edge from parent/.style={draw,dashed}]

\node[solid node, label = left: {\tiny $\mathfrak n_0$}] at (4,.4){}
    child{node[solid node, label={[label distance=-1.5mm] 240:\tiny $\mathfrak n_0^1$}]{}}
    child{node[solid node, label={[label distance=-1.5mm] 240:\tiny $\mathfrak n_0^2$}]{}}
    child{node[solid node, label = right: {\tiny $\mathfrak n_1$}]{}
        child{node[blue node]{}}
        child{node[blue node]{}}
        child[dashed]{node[solid node, label = right: {\tiny $\mathfrak n_{q}$}]{}
            child[solid]{node[green node]{}}
            child[solid]{node[green node]{}}
            child[solid]{node[solid node, label = right: {\tiny $\mathfrak n_{q+1}$}]{}}
            }
    }
;
\draw[dashed] (4,.4) -- (4, 1.2);
\draw[dashed] (2.5,-2) -- (2.5,-1.2);
\draw[dashed] (4,-2) -- (4,-1.2);
\draw[dashed] (8.5,-5) -- (8.5,-4.2);

\end{tikzpicture}
\caption{A self-loop chain} 
\label{fig-se}
\end{subfigure}
~
\begin{subfigure}[t]{0.45\linewidth}
\centering
\begin{tikzpicture}[scale = .9]
    \tikzstyle{every node} = [circle, draw = black]
    \node (1) at (-10,0) {\tiny ${v_0}$};
    \node (2) at (-8.9,0) {};
    \node (3) at (-8,0) {};
    \node (4) at (-7.1,0) {};
    \node (5) at (-6,0) {\tiny${v_4}$};

    \draw[thick, ->] (1) -- (2);
    \draw[thick, ->] (3) -- (2);
    \draw[thick, ->] (3) -- (4);
    \draw[thick, ->] (5) -- (4);

    \path[thick, ->] (2) edge [loop below, looseness = 20] (2);
    \path[thick, ->] (3) edge [loop below, looseness = 20] (3);
    \path[thick, ->] (4) edge [loop below, looseness = 20] (4);

    \draw[dashed] (1) -- (-10,1);
    \draw[dashed] (1) -- (-10,-1);
    \draw[dashed] (1) -- (-11,0);

    \draw[dashed] (5) -- (-6,1);
    \draw[dashed] (5) -- (-6,-1);
    \draw[dashed] (5) -- (-5,0);
\end{tikzpicture}
    \caption{A self loop of length $q = 3$, at the level of the molecule.}
    \label{fig-self-loop-mol}
\end{subfigure}
\caption{Self-loop chain at the level of the molecule and couple, before and after splicing.}
\label{fig-loop}
\end{figure}

\begin{lemma}\label{lem-loop}
For an enhanced couple $\mathcal Q$ and self-loop chain $(\mathfrak n_1, \ldots, \mathfrak n_{q})$ of length $q$, 
\begin{align}
\K_{\Q} (t,s,k) 
&= \left( \frac{\beta T}{N}\right)^{n} \hspace{-.25cm}\zeta(\mathcal Q) \sum_{\mathscr E} \tilde \epsilon_{\mathscr E} \int_{\mathcal{E}^{\mathrm{lf}}} S_q( t_{\mathfrak n_0}, t_{\mathfrak{n}_{q+1}}, k[\mathcal N]) \left( \prod_{\mathfrak n \in \mathcal N^{\mathrm{lf}}}e^{\zeta_{\mathfrak{n}}i\Gamma_{\mathfrak n}(t_{\mathfrak n})} \diff t_{\mathfrak{n}}\right)\diff \tau \prod^+_{\mathfrak l \in \mathcal L} n_{\mathrm{in}}(k_{\mathfrak l}),\label{eq-kqs}
\end{align}
where if $A(t)$, defined in Proposition \ref{prop-phase-renormalization}, satisfies $|\dot A(t)| \lesssim \beta T$ and $|\ddot A(t)| \lesssim T$, then $S_q$ satisfies
\begin{align}
\left|\left|\widehat{S}_q(\tau_0, \tau_{q + 1},  k[\mathcal N])\right|\right|_{L_{\tau_0, \tau_{q + 1}}^1} &\lesssim \log N \langle T \omega_{k_{\mathfrak n_1}} \rangle^{-\lceil q/2\rceil}\\
\left|\left|\widehat{S}_q(\tau_0, \tau_{q + 1},  k[\mathcal N])\right|\right|_{L_{\tau_0, \tau_{q + 1}}^\infty} &\lesssim \langle T \omega_{k_{\mathfrak n_1}} \rangle^{-\lceil q/2\rceil}.
\label{eq-sestimate}
\end{align} 
In the above, let $\mathcal N^{\mathrm{lf}}$ denote branching nodes of $\Q$ which are not $\mathfrak n_1, \ldots, \mathfrak n_q$ and $\mathcal E^{\mathrm{lf}}$ denote the domain of integration where we no longer integrate over $t_{\mathfrak n_j}$ $(j = 1, \ldots, q)$ and instead require $t_{\mathfrak n_0} > t_{\mathfrak n_{q + 1}}$. 
\end{lemma}

\begin{proof}
Denote $\zeta = \zeta_{\mathfrak n_1}$ and $\ell = k_{\mathfrak n_1}$ so that $\zeta_j := \zeta_{\mathfrak n_j} = (-1)^{j + 1}$ and $k_{\mathfrak n_j} = (-1)^{j + 1} \ell$ for $j = 1, \ldots, q+1$. Set $t_j = t_{\mathfrak n_j}$ for $j = 0, \ldots, q+1$, where $t_0= t$ (or $s$) if $\mathfrak n_1$ is a root and $t_{q + 1} = 0$ if $\mathfrak n_{q + 1}$ is a leaf. We fix the decoration of the couple as well as $t_{\mathfrak n}$ for $\mathfrak n \neq \mathfrak n_j$ $(j = 1, \ldots, q)$. Throughout, denote $\Omega_j = \Omega_{\mathfrak n_j}$, $\widetilde \Omega_j = \widetilde \Omega_{\mathfrak n_j}$, $\zeta_j = \zeta_{\mathfrak n_j}$ $(j = 1, \ldots, q)$. Note that 
\begin{align*}
\Omega_j &= (-1)^{j + 1}2\zeta \omega_\ell,
\end{align*}
where $\Omega_j, \widetilde \Omega_j$ are independent of $k_j$, for $k_j$ denoting the decorations of the paired (leaf) children of $\mathfrak n_j$. We isolate the following part of $\K_\Q$:
\begin{align*}
    S_q(t_0, t_{q + 1}, k[\mathcal N]) & :=  \int\limits_{t_{q+1} < t_q < \ldots < t_1 < t_0} \left( \prod_{j = 1}^q e^{(-1)^{{j + 1}}\zeta i( \Omega_jT t_j+ \widetilde{\Omega}_{j}(t_{j}))}\right) \diff t_1 \ldots \diff t_q, 
\end{align*}
so that 
\begin{align*}
\widehat{S_q}(\tau_0, \tau_{q + 1}, k[\mathcal N]) &= \int_{t_{q+1} < t_q < \ldots < t_1 < t_0} e^{-2\pi i (\tau_0 t_0 +\tau_{q + 1} t_{q + 1})} \chi(t_0) \chi(t_{q + 1}) \\
& \hspace{.5cm}\times \left( \prod_{j = 1}^q e^{(-1)^{{j + 1}}\zeta i( \Omega_jT t_j+ \widetilde{\Omega}_{j}(t_{j}))} \diff t_j\right) \diff t_0 \diff t_{q + 1}, 
\end{align*}
where we have inserted cutoff functions of appropriate decay to reflect that $0 \leq t_{q + 1}, t_0 \leq 1$.

Fix $K \gg 1$ sufficiently large. If $\tau_0, \tau_{q +1} \leq N^K$, using Proposition \ref{prop-integral_est}, we have that 
\begin{align*}
\left| \widehat{S_q}(\tau_0, \tau_{q + 1}, k[\mathcal N]) \right| & \lesssim \sum_{d_j \in \{0,1\}}\left(\prod_{j = 0}^{q + 1}\frac{1}{\langle Tq_j \rangle}\right),
\end{align*}
where $q_{\mathfrak n}$ is defined in Definition \ref{def:Trees}, so that 
\begin{align*}
q_0 = 2\pi\tau_0 + d_1 q_1, \hspace{.5cm} 
q_{q+ 1} = 2\pi\tau_{q + 1},
\end{align*}
and for $j = 1, \ldots, q$, $q_j$ is an integer multiple of $\omega_\ell$ (with an additional factor of $2\pi\tau_{q + 1}$), so that for any $\tau_0, \tau_{q + 1} \leq N^K$
\begin{align*}
\left| \widehat{S_q}(\tau_0, \tau_{q + 1}, k[\mathcal N]) \right| & \lesssim \langle \tau_0  + \lambda_0 \rangle^{-1} \langle \tau_{q + 1}\rangle^{-1} \langle T\omega_\ell \rangle^{-\lceil q/2 \rceil},
\end{align*}
for some $\lambda_0 \in \R$.

If one of $\tau_0, \tau_{q + 1} \geq N^K,$ we first integrate in $t_{q + 1}, \ldots, t_{1}$ with result $\widetilde S_q(t_0)$, then integrate by parts twice in $t_{0}$:
\begin{align*}
\widehat{S_q}(\tau_0, \tau_{q + 1}, k[\mathcal N]) = \int_{t_0 \in \R} e^{-2\pi i \tau_0 t_0} \chi(t_0) \widetilde S_q(t_0) \diff t_0 = \int_{t_0 \in \R} \frac{-1}{4\pi^2\tau_0^2} e^{-2\pi i\tau_0 t_0} \frac{d^2}{dt_0^2}\left[\chi(t_0) \widetilde S_q(t_0)\right] \diff t_0.
\end{align*}
Note that $|\widetilde S_q|, |\widetilde S_{q}'(t_0)|, |\widetilde S_q''(t_0)| \leq N^{K/2} \langle \tau_{q + 1}\rangle^{-1} \langle\tau_{q + 1} + T \omega_\ell\rangle^{-1}$ for $K$ large enough and all derivatives of $\chi$ are uniformly bounded. So, 
\begin{align*}
\left| \widehat{S_q}(\tau_0, \tau_{q + 1}, k[\mathcal N]) \right| & \lesssim N^{K/2}\langle \tau_0  + \lambda_0 \rangle^{-2} \langle \tau_{q + 1}\rangle^{-1} \langle \tau_{q + 1} + \lambda_{q + 1} \rangle,
\end{align*}
for some $\lambda_0, \lambda_{q + 1} \in \R$.

Therefore, since $\langle T\omega_\ell\rangle \lesssim T$, we have
\begin{align*}
&\| \widehat{S_q}(\tau_0, \tau_{q + 1}, k[\mathcal N]) \|_{L^1_{\tau_0, \tau_{q + 1}}} \\
\lesssim &\int_{\tau_0, \tau_{q + 1} \leq N^K} \langle \tau_0  + \lambda_0 \rangle^{-1} \langle \tau_{q + 1}\rangle^{-1} \langle T\omega_\ell \rangle^{-\lceil q/2 \rceil} \diff \tau_0 \diff \tau_{q + 1} \\
& \hspace{1cm} + \int_{\tau_0 \text{ or } \tau_{q + 1} \geq N_K} N^{K/2}\langle \tau_0  + \lambda_0 \rangle^{-2} \langle \tau_{q + 1}\rangle^{-1} \langle \tau_{q + 1} + \lambda_{q + 1} \rangle \diff \tau_0 \diff \tau_{q + 1} \\
\lesssim &\log N \langle T\omega_\ell\rangle ^{-\lceil q/2\rceil} + N^{-K/2}. 
\end{align*}
As $q$ is finite, we may choose $K$ appropriately large to recover \eqref{eq-sestimate}. Note that \eqref{eq-kqs} follows automatically.
\end{proof}

\subsection{Splicing Specifications}

We may define the following set whose existence and uniqueness is established in \cite{WKE, ODW}:
\begin{definition}\label{def-C}
Consider an enhanced couple $\mathcal Q$ and enhanced molecule $\M(\Q)$. Let $\mathscr C$ be defined as a unique collection of disjoint atomic groups, such that each atomic group in $\mathscr C$ is a chain, hyperchain, or pseudo-hyperchain and any chain $\C$ of $\M$ is a subset of precisely one atomic group in $\mathscr C$. 

Suppose also that we have chosen $\mathscr C_{SG} \subset \mathscr C$ of all SG negative chain-like objects. Consider the set $\mathscr D_{SG}$ defined as a unique collection of disjoint atomic groups, such that each atomic group in $\mathscr D_{SG}$ is a maximal wide ladder of $\mathscr C_{SG}$ and each chain $\C \in \mathscr D_{SG}$ is a subset of precisely one atomic group in $\mathscr D_{SG}$. 
\end{definition}

\begin{definition}[Splicing Locations] \label{def-M}
For an enhanced couple $\Q$ and corresponding enhanced molecule $\M(\Q)$, suppose we have a set $\mathscr D_{SG}$ of SG chain-like objects, coming from Definition \ref{def-C}. Then, define $\mathcal M_{SG}$ below. If $\mathcal L \in \mathscr D_{SG}$ consists of a single chain $\C = (v_0, \ldots, v_q)$:
\begin{enumerate}[label=(\roman*)]
\item If $\C$  is a chain, we include all admissible $\mathfrak n(v_i)$ into $\mathcal M_{SG}$.
\item If $\C$ is a hyperchain or pseudo-hyperchain with a CN double bond, we include all admissible $\mathfrak n(v_i)$ into $\mathcal M_{SG}$. If there is no CN double bond, we exclude one admissible $\mathfrak n(v_i)$. 
\end{enumerate}
Otherwise, let $\mathcal{L} = \{\mathcal{C}_1, \ldots, \mathcal{C}_m\}$, where each chain $\mathcal{C}_j$ has length $q^{(j)}$. We perform the following procedure:

\begin{enumerate}
  \item Start with $\mathcal{C}_1$ and determine which of its nodes should be added to $\mathcal{M}_{SG}$ by applying the previously described procedure, treating it as a hyperchain or pseudo-hyperchain if applicable.
  
  \item For each subsequent chain $\mathcal{C}_{j+1}$ with $j \leq m - 2$, check whether all $q^{(j)}$ nodes of $\mathcal{C}_j$ are contained in $\mathcal{M}_{SG}$. If so, treat $\mathcal{C}_{j+1}$ as in step (ii) above when determining which nodes to add to $\mathcal{M}_{SG}$. Otherwise, treat it as in step (i).

  \item For the final chain $\mathcal{C}_m$, if it is a hyperchain or pseudo-hyperchain, or if all $q^{(m-1)}$ nodes of $\mathcal{C}_{m-1}$ are in $\mathcal{M}_{SG}$, then process $\mathcal{C}_m$ as in step (ii). Otherwise, treat it as in step (i).
\end{enumerate}
\end{definition}

\begin{proposition}[Properties of Spliced Molecules] \label{prop-M}
For an enhanced couple $\Q$, and choice of SG negative chain-like objects $\mathscr C_{SG}$, consider the resulting couple $\mathcal Q^{\mathrm{sp}}$ obtained by splicing $\mathcal Q$ at the nodes in $\mathcal M_{SG}$, defined in Definition \ref{def-M}. Then,
    \begin{align} \label{eq-splice-exp}
    \begin{split}
    \K_{\textgoth{Q}_\mathcal M} (t,s,k) &= \left( \frac{\beta T}{N}\right)^{n_{\mathrm{sp}}} \zeta(\mathcal Q^{\mathrm{sp}}) \sum_{\mathscr E^{\mathrm{sp}}} \tilde \epsilon_{\mathscr E^{\mathrm{sp}}} \int_{\mathcal{E}^{\mathrm{sp}}} \prod_{\mathfrak n_0 \in \mathcal N^{\mathrm{sp}}_0} P_{q_{\mathfrak n_0}}(t_{\mathfrak n_0}, t_{\mathfrak n_0^2}, t_{\mathfrak n_0^3}, k[\mathcal N^{\mathrm{sp}}])  \\
    & \hspace{0.8cm}\times \prod_{\mathfrak n_1 \in \mathcal N^{\mathrm{l}}} S_{q_{\mathfrak n_1}}(t_{\mathfrak n_0}, t_{\mathfrak n_{q + 1}}, k[\mathcal N^{\mathrm{sp}}]) \left( \prod_{\mathfrak n \in \mathcal N^{\mathrm{lf}}}e^{\zeta_{\mathfrak{n}}i\Gamma_{\mathfrak n}(t_{\mathfrak n})} \diff t_{\mathfrak{n}}\right) \prod^+_{\mathfrak l \in \mathcal L^{\mathrm{sp}}} n_{\mathrm{in}}(k_{\mathfrak l}),
    \end{split}
    \end{align}
    where $\mathcal N^{\mathrm{sp}}$ denotes the branching nodes of $\Q^{\mathrm{sp}}$ with $\mathcal N_0^{\mathrm{sp}}$ denoting the nodes where an irregular chain was spliced out, $\mathcal N^{\mathrm{lf}}$ denoting all branching nodes which are loop-free, and $\mathcal N^{\mathrm{l}}$ denoting all branching nodes which are the root of a maximal self-loop chain of $\Q^{\mathrm{sp}}$. In the above, $P_{q_{\mathfrak n_0}}$ is given in Lemma \ref{lem-splice}, where $q_{\mathfrak n_0}$ denotes the length of the chain spliced out below $\mathfrak n_0$ and $S_{q_{\mathfrak n_1}}$ is given in Lemma \ref{lem-loop}, where $q_{\mathfrak n_1}$ denotes the length of the maximal self-loop chain which $\mathfrak n_1$ is a root of. Additionally, $\mathcal{E}^{\mathrm{sp}}$ denotes the domain of integration where we integrate only in $t_\mathfrak{n}$ for $\mathfrak n$ loop free in $\Q^{\mathrm{sp}}$, coming from Lemmas \ref{lem-splice} and \ref{lem-loop}. The molecule $\M^{\mathrm{sp}} = \M(\Q^{\mathrm{sp}})$ has the following properties: 
\begin{enumerate}[(a)]
\item The molecule $\mathbb M^{\mathrm{sp}}$ is connected and has either 2 atoms of degree 3 or one atom of degree 2, with the rest having degree 4.  
\item Each chain in $\mathscr C_{SG}$ is reduced to a single atom or a single double bond in $\mathbb M^{\mathrm{sp}}$. Similarly, each SG hyperchain is reduced to a triple bond and each SG pseudo-hyperchain is reduced to a single double bond pseudo-hyperchain. 
\item All SG double bonds in $\M^{\mathrm{sp}}$ are CN or are the only double bond in a pseudo-\hspace{0pt}hyperchain. As such, the molecule remains connected when any SG double bond is removed.
\item For any degree 4 atom $v$ in $\M^{\mathrm{sp}}$ with a self-loop, the remaining edges must have the same direction.
\end{enumerate}
\end{proposition}
\begin{proof}
Note that for an enhanced couple $\Q$, all maximal self-loops and irregular chains which we splice are disjoint, and by our choice of splicing, we do not create any new self-loops in $\Q^{\mathrm{sp}}$. In other words, $\mathcal N_0^{\mathrm{sp}}$ and $\mathcal N^\mathrm{l}$ are disjoint. Therefore, we may apply Lemmas \ref{lem-splice} and \ref{lem-loop} to obtain \eqref{eq-splice-exp}. Points (a) and (b) are shown in \cite[Proposition~5.15]{ODW}. For (c), note that \cite[Proposition~4.12]{WKE2023}. For (d), note that self-loops in $\M^{\mathrm{sp}}$ correspond to self-loops in $\Q^{\mathrm{sp}}$ which only appear at (1,2) nodes. 
\end{proof}

\section{Counting Algorithm}
\label{section-algorithm}
In this section, we prove the following proposition and use it to prove Proposition \ref{bound-couple}. 

\begin{proposition}\label{prop-rigidity}
Consider a labeled enhanced molecule $\M^{\mathrm{sp}} = \mathbb M(\Q^{\mathrm{sp}})$ of order $\chi(\M^{\mathrm{sp}}) = n\leq N^3$. Suppose we fix $k \in \Z_N\cap(0,1)$ and $\alpha_v \in \R$ for each atom $v$ of $\mathbb M^{\mathrm{sp}}$ which does not have a self-loop. Consider all $k$-decorations $(k_\ell)$ of $\M$ such that 
\begin{enumerate}[(i)]
\item The $k$-decoration $(k_\ell)$ is inherited from a $k$-decoration $\mathscr E$of $\Q^{\mathrm{sp}}$ that satisfies the SG and LG assumptions in Proposition \ref{prop-M} as well as non-degeneracy conditions $\tilde \epsilon_{\mathscr E}.$
\item The decoration is restricted by $(\alpha_v)$ in the sense that $|\Omega_v - \alpha_v| \leq T^{-1}$, where $T < N$. 
\end{enumerate}
Denote the number of such decorations by $\mathfrak C_{\M^{\mathrm{sp}}}$, where we take supremums in the parameters $\alpha_v$. Additionally, define the quantity $\mathfrak U$ to be: 
\begin{equation}
\mathfrak U = \left( \frac{\beta T}{N}\right)^{\chi} T^{-L/2}\cdot \mathfrak C_{\M^{\mathrm{sp}}}
\end{equation}
Then, we have the following bound on $\mathfrak U$:
\begin{equation}\label{eq-counting}
\mathfrak U \leq C^n (\log N)^2 \beta^2T N^{-\delta (n - 2)}.
\end{equation}
\end{proposition}

In order to prove this proposition, we will perform various operations to $\M^{\mathrm{sp}}$, specifically cutting steps followed by a counting algorithm, detailed below. Each operation will begin with $\M_{\mathrm{pre}}$ to yield $\M_{\mathrm{post}}$. The quantity $\mathfrak U$ corresponding to each, which are restricted by Proposition \ref{prop-rigidity} and any previously performed operations, we will denote $\mathfrak U_{\mathrm{pre}}$ and $\mathfrak U_{\mathrm{post}}$. At any step, we will have 
\begin{equation}
\mathfrak U_{\mathrm{pre}} \lesssim \mathfrak D \cdot \mathfrak U_{\mathrm{post}},
\end{equation}
where $\mathfrak D = \left(\frac{\beta T}{N}\right)^{-\Delta \chi} T^{\Delta L/2} \cdot \mathfrak C$ will denote a deviation estimate corresponding to the operation that was performed from $\M_{\mathrm{pre}}$ to yield $\M_{\mathrm{post}}$. Here, $\mathfrak C$ will denote a counting estimate for the operation. Throughout Section \ref{subsec-cutting} and \ref{subsec-algorithm}, we detail a modified version of what is presented in \cite{DIP25} to adjust for the fact that all atoms have degree at most 4 and bonds can have any direction. 

\subsection{Cutting Operation} \label{subsec-cutting}
\begin{definition}
Given a molecule $\M$ and atom $v$ of degree 3 or 4, suppose $v$ has two bonds $\ell_1$ and $\ell_2$. Define \textit{cutting} $\M$ at atom $v$ and relative to bonds $(\ell_1, \ell_2)$ via the following: 
\begin{enumerate}
    \item Replace atom $v$ with atoms $v_1$ and $v_2$. 
    \item The bonds at atom $v_1$ are $\ell_1$ and $\ell_2$ and at $v_2$ are any remaining bonds in the original atom $v$. 
\end{enumerate}
We call this cut an $\alpha$-cut, and the corresponding new atoms $v_1$ and $v_2$ $\alpha$-atoms, if the cutting operation does not generate a new connected component. In the case of an $\alpha$-cut, we require that $(\ell_1, \ell_2)$ have opposite direction (any remaining edges may have any direction). If the cutting operation does create a new component, we call it a $\beta$-cut and atoms $v_1$ and $v_2$ $\beta$-atoms. See Figure \ref{fig-cutting}.
\end{definition}

\begin{figure}[ht]
\begin{tikzpicture}[scale = .7, el/.style = {inner  sep=1pt, align=left},
every label/.append style = {font=\tiny},]
    \tikzstyle{every node} = [circle, draw = black]
    \node (v) at (-4,0) {{$v$}};
    \node (v_2) at (3.5,-.5) {\tiny{$v_1$}};
    \node (v_1) at (4.5,.5) {\tiny{$v_2$}};

    \draw[thick] (-4,1.5) -- node[el, right, pos = .3, draw = none]{\tiny$\ell_1$} (v) ;
    \draw[thick] (v) -- node[el, below, pos = .5, draw = none]{\tiny$\ell_2$} (-2.5,0);
    \draw[thick] (v) -- (-4,-1.5);
    \draw[dashed] (v) -- (-5.5,0);

    \draw[dashed, ->] (-1, 0) -- node[el, above, pos = .5,draw = none]{\tiny{cutting}} (1, 0);

    \draw[thick] (4.5,2) -- node[el, right, pos = .3, draw = none]{\tiny$\ell_1$} (v_1);
    \draw[thick] (v_1) -- node[el, below, pos = .5, draw = none]{\tiny$\ell_2$} (6,.5);
    \draw[dashed] (v_2) -- (2,-.5);
    \draw[thick] (v_2) -- (3.5,-2);
\end{tikzpicture}
\caption{Cutting operation}
\label{fig-cutting}
\end{figure}
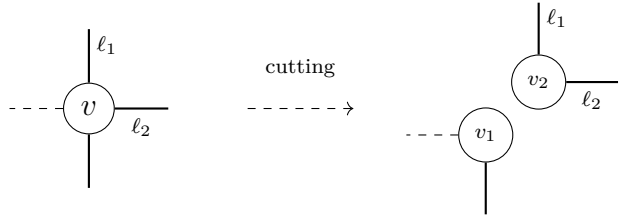

We apply this cutting to a molecule $\M^{\mathrm{sp}} = \M(\Q^{\mathrm{sp}})$ in the following steps: 

\begin{enumerate}[label= \underline{Step \arabic*:}]
    \item (\textit{SG Double Bonds}) For each SG double bond connecting atoms $v_1$ and $v_2$ , remove the double bond and call $v_1, v_2$ $\alpha$-atoms. Note if the double bond edges are $(\ell_1, \ell_2)$, this corresponds to cutting both $v_1, v_2$ relative to $(\ell_1, \ell_2)$, and then removing the two atoms and double bond which form the new component. By Proposition \ref{prop-M}, this operation creates no new components, so we call $v_1, v_2$ $\alpha$-atoms.
    \item (\textit{Self-loops}) For each self-loop at atom $v$, remove the self-loop and call $v$ an $\beta$-atom. Note that this corresponds to cutting $v$ relative to the self-loop and then removing the new component consisting of a single loop. 
    \item (\textit{SG, Not Tame}) For any SG atom $v$ which is not tame, cut $\M^{\mathrm{sp}}$ at $v$ relative to two bonds $(\ell_1, \ell_2)$ of opposite direction which are small gap. 
    \item (\textit{Remaining $\beta$-cuts}) For any remaining atom $v$ which is not tame (and therefore LG), perform a $\beta$-cut at $v$ if it is possible. 
\end{enumerate}

\begin{proposition} \label{prop-cutting-est}
For each of Steps 1 - 4 above, 
\begin{equation}
\mathfrak U_{\mathrm{pre}} \lesssim \mathfrak U_{\mathrm{post}} \cdot N^{\delta \Delta \chi}\left(\prod_v^{(\alpha)} \mathfrak D_v^\alpha \right)(\mathfrak D^\beta)^{\Delta V_\beta - 2 \Delta F}.
\end{equation}
Here, the product $\prod_v^{(\alpha)}$ is taken over any newly created $\alpha$ atoms whose corresponding gap $\lambda_v$ and $\Delta V_\beta$ denotes the number of new $\beta$-atoms. Denote
\begin{align}
\mathfrak D_v^\alpha &:= \begin{cases}\beta TN^\delta \lambda_v^{1/2} & v \text{ degree 2}, \\ \beta T^{1/2}N^{\delta} \log N & v \text{ degree 1}, \end{cases} \\
\mathfrak D^\beta &:= \beta T^{1/2} N^{\delta}.
\end{align}
\end{proposition}
\begin{proof}
This is essentially proved in \cite[Proposition~5.20]{DIP25}, however we reiterate it here to clarify that the dispersion relation plays no role in the estimates.  For Step 1, we remove double bond $(\ell_1, \ell_2)$. So, $\Delta E = -2, \Delta V = 0, \Delta F = 0, \Delta L = 0, \Delta V_\beta = 0$, corresponding to $\Delta \chi = -2$. If both $v_1$ and $v_2$ are degree 4, we may count $(k_{\ell_1}, k_{\ell_2})$ using $\mathfrak C \lesssim N(N\lambda)$ by fixing one and using the SG assumption. Therefore, 
\[\mathfrak D \lesssim \left(\frac{\beta T} {N}\right)^2 N^2 \lambda = (\beta T N^\delta \lambda^{1/2})^2N^{-2\delta}. \] 
Otherwise, one of $v_1,v_2$ is degree 3. Therefore, we may Proposition \ref{prop-3vc} to count $k_{\ell_1}, k_{\ell_2}$ (as well as the third edge of the degree 3 atom which becomes degree 1 and so will be fixed regardless by the cutting) with $\mathfrak C \lesssim N^2 T^{-1}\log T$, so that \[\mathfrak D \lesssim \left(\frac{\beta T}{N}\right)^2 N^2T^{-1}\log T \lesssim N^{-2\delta} (\beta T^{1/2} N^{\delta} \log N)^2,\] where we note that we may interpolate between the two estimates depending on if both or only one of $v_1, v_2$ is degree 3.

For Step 2, we remove a single bond $\ell$ corresponding to a self loop and have $\Delta E = -1, \Delta V = 0, \Delta F = 0, \Delta L = -1, \Delta V_\beta = 1$, so $\Delta \chi = -1$. Note that we may count $k_\ell$, we have $\mathfrak C \lesssim N$, so that \[\mathfrak D \lesssim \left(\frac{\beta T}{N}\right)T^{-1/2}N = N^{-\delta} (\beta T^{1/2} N^\delta)\]

For an $\alpha$-cut in Step 3, we have $\Delta E = 0, \Delta V = 1, \Delta F = 0, \Delta L = 0, \Delta V_\beta = 0$, so $\Delta \chi = -1$. It suffices to fix the values $k_{v_1}, \alpha_{v_1}$, where $v_1$ is one of the newly created atoms resulting from cutting at $v$. This will fix values for the other new atom. If $v$ is degree 4, by the SG assumption, there are $N\lambda$ choices of $k_{v_1}$ and $\lambda T$ choices for $\alpha_{v_1}$, using the fact that $|\omega_{k_{\ell_1}} - \omega_{k_{\ell_2}}| \lesssim |k_{\ell_1} - k_{\ell_2}|$. So, \[\mathfrak C \lesssim N\lambda^2 T \Rightarrow \mathfrak D \lesssim \beta T^2 \lambda^2 \lesssim N^{-\delta}(\beta T N^\delta \lambda^{1/2})^2 (\beta^{-1}N^{-\delta}\lambda ). \]
Note we have chosen $\lambda < \beta N^\delta$. If $v$ is degree 3 then we only need to fix the single edge and this automatically fixes $\alpha_{v_1}$. So, 
\[\mathfrak C \lesssim N\lambda \Rightarrow \mathfrak D \lesssim \beta T \lambda \lesssim \beta^2TN^\delta \lesssim N^{-\delta} (\beta T^{1/2}N^{\delta} \log N)^2. \]

For a $\beta$-cut either in Step 3 or Step 4, we have $\Delta E = 0, \Delta V = 1, \Delta F = 1, \Delta L = 0, \Delta V_{\beta} = 2$, so $\Delta \chi = 0$. Note that the values for $k_{v_1}, k_{v_2}, \alpha_{v_1}, \alpha_{v_2}$ will be fixed by adding the equations for all atoms in one of the components.
\end{proof}

\begin{proposition}[Properties of reduced molecules]\label{prop-cutting-properties}
Let the result of performing all steps of the cutting be a \textit{reduced molecule} $\M^{\mathrm{fin}}$. Call any atoms which are not $\alpha$ or $\beta$ atoms $\epsilon$-atoms. Then, the molecule $\M^{\mathrm{fin}}$ has the following properties: 
\begin{enumerate}
    \item Either $\M^{\mathrm{fin}}$ has one component, or each component has at least one $\beta$-atom.
    \item There is at most one component, called the \textit{odd component}, that contains two atoms of odd degree (either 1 or 3). Other components, called \textit{even components}, have all atoms degree 2 or 4. 
    \item Any $\epsilon$-atom must be tame or LG and there cannot be a $\beta$-cut which can be performed at the atom.
    \item If a component contains an $\epsilon$-atom, the $\alpha$ and $\beta$-atoms form chains. Each such chain has endpoints at two distinct $\epsilon$-atoms. If the component has no $\epsilon$-atoms, it is either a cycle which is not a SG double bond or self-loop or the odd component where both odd degree atoms have degree one. Furthermore, $\M^{\mathrm{fin}}$ contains no self-loops.
\end{enumerate}
\end{proposition}
\begin{proof}
See \cite[Proposition~9.5]{WKE2023} and \cite[Proposition~5.2]{DIP25}. Note that cutting removes all self-loops.
\end{proof}

\subsection{Algorithm} \label{subsec-algorithm}

\begin{proposition}[$\epsilon$-atom removal] \label{prop-atom-removal}
Suppose $\M'$ is a component of $\M^{\mathrm{fin}}$ as detailed in Proposition \ref{prop-cutting-properties} and $v$ is an $\epsilon$-atom. If we remove atom $v$ along with all of its bonds from $\M'$,

\begin{equation}
\mathfrak U_{\mathrm{pre}} \lesssim \mathfrak U_{\mathrm{post}} \cdot N^{\delta \Delta \chi}.
\end{equation}
Additionally, if $v$ is degree 3 and $\Delta F = 0$, 
\begin{equation}
\mathfrak U_{\mathrm{pre}} \lesssim \mathfrak U_{\mathrm{post}} \cdot \beta^2 T \log N,
\end{equation} 
or if $v$ has degree 4 with $\Delta F \leq 1$, 
\begin{equation}
\mathfrak U_{\mathrm{pre}} \lesssim \mathfrak U_{\mathrm{post}} \cdot N^{-\delta}(\beta T N^{\delta})^{-1}.
\end{equation}
\end{proposition}

\begin{proof}
Set $q = \Delta F$. For the $j$th component, set $p_j$ to be the number of bonds of $v$ which connected to this component. We consider the $q$-tuple $(p_1, \ldots, p_q)$. By convention $p_1 \geq p_2 \geq \ldots \geq p_q$. We only consider $\Delta \chi < 0$ as if $\Delta \chi = 0$, each bond of $v$ is connected to a unique component, whose decoration can be determined by adding the equations for all atoms in the corresponding component. So, $\text{deg}(v) > 1$ and the corresponding tuple is not $(1, \ldots, 1)$

If $\text{deg}(v) = 2$, then the only tuple we must consider is (2), for which $\Delta \chi = -1$. In this case, we must consider the direction of the bonds in addition to whether $v$ is LG or tame. The deviation $\mathfrak D_2$ is then bounded (using Proposition \ref{vector_counting-2}) by: 
\begin{itemize}
    \item If the bonds have the same direction, $\mathfrak D_2 \lesssim \left(\frac{\beta T}{N}\right) (NT^{-1/2}) \lesssim N^{-\delta}$. 
    \item If $v$ is tame, $\mathfrak D_2 \lesssim \left(\frac{\beta T}{N}\right) (N\lambda) \lesssim \beta^2 T N^{\delta}\lesssim N^{-\delta}$. 
    \item If $v$ is LG, $\mathfrak D_2 \lesssim \left(\frac{\beta T}{N}\right) (NT^{-1}\lambda^{-1}) \lesssim N^{-\delta}$. 
\end{itemize}

If $\text{deg}(v) = 3$, the possible tuples are 
\begin{itemize}
    \item (2,1): Then $\Delta \chi = 1$ and $\mathfrak D \lesssim \mathfrak D_2$.
    \item (3): Then, $\Delta \chi = -2$ and  the deviation $\mathfrak D_3 \lesssim \left(\frac{\beta T}{N} \right)^{2} N^2T^{-1}\log T \lesssim \beta^2 T N^{\delta/2}$. Here we use Proposition \ref{prop-3vc}. 
\end{itemize}
If $\text{deg}(v) = 4$, the possible tuples are 
\begin{itemize}
    \item (2, 1, 1): Then $\Delta \chi = -1$ and $\mathfrak D \lesssim \mathfrak D_2$.
    \item (2, 2): Then $\Delta \chi = -2$ and $\mathfrak D \lesssim(\mathfrak D_2)^2$. 
    \item (3, 1): Then $\Delta \chi = -2$ and $\mathfrak D \lesssim \mathfrak D_3$. 
    \item (4): Then $\Delta \chi = -3$ and the deviation $\mathfrak D_4 \lesssim \left( \frac{\beta T}{N}\right)^3 (N^3 T^{-1}\log T) \lesssim \beta^3 T^2 \log T$.
\end{itemize}
\end{proof}

Then, we employ the following algorithm on each component $\M'$ of $\M^{\mathrm{fin}}$ which is not a cycle. Here, the operations are listed in order of priority. So, one must always attempt Operation 1 before moving to Operation 2 and so forth: 
\begin{enumerate}[label= \underline{Operation \arabic*:}, align = left]
\item If any $\alpha$ or $\beta$-atoms have degree 1, remove them. 
\item Remove an $\epsilon$-atom of odd degree. 
\item Remove an $\epsilon$-atom for which $\text{deg}(v) = \text{deg}_0(v)$. 
\item Remove any other $\epsilon$-atom.
\end{enumerate} 

With this algorithm, we obtain the following counting estimate:

\begin{proposition}[Counting estimates for reduced molecules]\label{prop-counting}
Consider $\M'$ a component of $\M^{\mathrm{fin}}$. If $\M'$ is an odd component, 
\begin{equation}
\mathfrak U \lesssim \beta^2T(N^{-\delta})^{\chi - 2} (\log N)^2\left(\prod^{(\alpha)} \mathfrak D_v^\alpha\right)^{-1} \left(\mathfrak D^\beta \right)^{-V_\beta},   
\end{equation}

and if $\M'$ is an even component

\begin{equation}
\mathfrak U \lesssim N^{-\delta  \chi}\left(\prod^{(\alpha)} \mathfrak D_v^\alpha\right)^{-1} \left(\mathfrak D^\beta\right)^{2 - V_\beta}.   
\end{equation}
\end{proposition}

\begin{proof}
\textit{Case 1: $\M'$ is a LG double bond.} Then, $\chi = 1$ with $V_\alpha = 0, V_\beta = 2$ and $\M'$ is an even component. So, for $\lambda$ the gap of the double bond, \[\mathfrak U \lesssim \left( \frac{\beta T}{N}\right) NT^{-1} \lambda^{-1} \lesssim N^{-\delta},\]
using Proposition \ref{vector_counting-2}.

\textit{Case 2: $\M'$ is a larger cycle.} If $\M'$ is a cycle of at least 3 atoms, then $\chi = 1$ and $\M'$ is an even component, but $V_\alpha, V_\beta$ can vary. If $V_\beta \geq 3$, then \[\mathfrak U \lesssim \left( \frac{\beta T}{N}\right) N \lesssim N^{-\delta}(\beta T N^{\delta}) \lesssim N^{-\delta} (\beta T^{1/2} N^\delta)^{-1}.\] 

If $V_{\beta} = 2,$ then $V_{\alpha} \geq 1$, with some gap denoted $\lambda$. Then, 
\[\mathfrak U \lesssim \left( \frac{\beta T}{N}\right) \min(N, NT^{-1}\lambda^{-1}) \lesssim \beta T^{\frac{1}{2}} \lambda^{-1/2} \lesssim N^{-\delta} (\beta TN^{\delta} \lambda^{1/2})^{-1}N^{2\delta} \beta^2T^{3/2}, \] where $N^{-\delta} (\beta TN^{\delta} \lambda^{1/2})^{-1} \ll 1$.

If $V_\beta = 1$, then $V_\alpha \geq 2$, with two of the gaps denoted $\lambda_1, \lambda_2$, so \[\mathfrak U \lesssim \left( \frac{\beta T}{N}\right) NT^{-1}\min(\lambda_1^{-1},\lambda_2^{-1}) \lesssim \beta \lambda_1^{-1/2} \lambda_2^{-2/1} \lesssim N^{-\delta} (\beta TN^{\delta} \lambda_1^{1/2})^{-1}(\beta TN^{\delta} \lambda_2^{1/2})^{-1}N^{3\delta} \beta^3T^{2},\] where $N^{3\delta}\beta^3T^{2} \ll 1$.

\textit{Case 3: $\M'$ is a chain which is not a cycle.} If $\M'$ is a chain of $\alpha$ and $\beta$-atoms, then $\chi = 0$, $\M'$ contains at least two atoms, and $\mathfrak U \lesssim 1$. Note that regardless of if the atoms are $\alpha$ or $\beta$-atoms, we have $\beta T^{1/2}N^{\delta} \log N \cdot \max(\mathfrak D_v^\alpha, \mathfrak D^\beta)^{-1}\gtrsim 1$.

\textit{Case 4: $\M'$ contains an $\epsilon$-atom.} Note that by Proposition \ref{prop-atom-removal}, each removal of an $\epsilon$-atom has a deviation $\mathfrak D \lesssim N^{-\delta}$.

If $\M'$ is an even component, $V_\beta \geq 1$. In this case, we note that $\M'$ contains at least one atom of degree 4 which has no $\beta$-cut and such an atom will be removed first by priority. Therefore, by Proposition \ref{prop-atom-removal}, the deviation for removing this atom is $\mathfrak D \lesssim N^{-\delta} (\beta T N^{\delta})^{-1}$.

If $\M'$ is an odd component, $\M'$ contains at least one atom of degree 3 which has no $\beta$-cut and such an atom will be removed first by priority. Therefore, by Proposition \ref{prop-atom-removal} the deviation for removing this atom is $\mathfrak D \lesssim \beta^2 T\log N$. 
\end{proof}

Now Proposition \ref{prop-rigidity} follows immediately:

\begin{proof}[Proof of Proposition \ref{prop-rigidity}] We may simply apply Proposition \ref{prop-cutting-est}, followed by Proposition \ref{prop-counting}. If $\M^{\mathrm{fin}}$ has $F$ components, Proposition \ref{prop-cutting-est} contributes $(\mathfrak D^\beta)^{V_\beta - 2(F-1)}$. If $\M^{\mathrm{fin}}$ has an odd component, the number of even components is $F-1$, where $F$ is the number of components of $\M^{\mathrm{fin}}$. If there is not an odd component, the number of even components is $F$, while Proposition \ref{prop-counting} contributes $(\mathfrak D^\beta)^{2F - V_\beta}$. 
\end{proof}

\subsection{Proof of Proposition \ref{bound-couple}}

\begin{proof}[Proof of Proposition \ref{bound-couple}]
We prove (\ref{E1}) and (\ref{E2}) using bootstrap arguments. We establish the hypotheses as:
\begin{align}
    \left|\sum_{\Q}\K_\Q(t,t,k)\right|& \leq 2C\beta^2 T (N^{-\delta})^{n - \frac{5}{2}},\label{H1}\\
     \left|\sum_{\Q}\dot{\K}_\Q(t,t,k)\right|&\leq 2C \beta^2 T^2 (N^{-\delta})^{n - \frac{5}{2}} ,\label{H2}
\end{align}
for the same constant $C$ as in (\ref{E1}, \ref{E2}). Note that by the proof of Proposition \ref{prop-phase-renormalization}, $\dot{A}(t)$ is bounded and smooth (for almost every $\varrho$). Therefore, we have global existence and uniqueness of $A(t)$ for almost all $\varrho$. Note that 
\begin{align*}
\dot{A}(t) &= \left( \frac{\beta T}{N}\right) \sum_{\ell \in \Z_N \cap [0,1)} \frac{3}{2\kappa^2} \omega_\ell \sum_{\Q} \K_\Q (t,t, \ell) \\
&= C_0 \beta T t + \dot{A}(t) +\left( \frac{\beta T}{N}\right) \sum_{\ell \in \Z_N \cap [0,1)} \frac{3}{2\kappa^2} \omega_\ell \sum_{\Q} \K_\Q (t,t, \ell),
\end{align*}
for some constant $C_0$ and the second sum taken over enhanced couples $\Q = \{\T^+, \T^-, \mathscr P\}$ such that $1 \leq |\T^+|, |\T^-| \leq M$. Note that we may exclude the case $\Q$ having order one as there are no enhanced couples of order one. Therefore, by (\ref{H1}, \ref{H2}), we have: 
\begin{align}
    \dot{A}(t) & \lesssim \beta T + \beta^3 T^2 N^{2\delta} \lesssim \beta T\\
    \ddot{A}(t) & \lesssim \beta^3 T^3 N^{2\delta} \lesssim T \label{eq-Add}.
\end{align}
Now, we may apply Proposition \ref{prop-integral_est} to establish (\ref{E1}, \ref{E2}) with appropriate constant:

Since the couples we consider are of order at most $M^3$, the total number of couples is $O(C^{M^3}(M^3)!)$, independent of $N$. Corresponding to each enhanced couple $\Q$, there are $O(C^{M^3})$ choices for the collection of SG negative-chain like objects $\mathscr C_{SG}$ and $\mathcal M_{SG}$ defined in Definition \ref{def-M}. It suffices to fix an enhanced couple $\Q$ along with a choice $\mathscr C_{SG}$ and bound the corresponding expression for  $\K_{\textgoth{Q}_{\mathcal M}}$ in (\ref{eq-splice-exp}). Using Lemmas \ref{lem-splice} and \ref{lem-loop}, we may write $\K_{\textgoth{Q}_{\mathcal M}}$ in (\ref{eq-splice-exp}) as:

\begin{align*}
\begin{split}
    \eqref{eq-splice-exp} &= \left( \frac{\beta T}{N}\right)^{n_{\mathrm{sp}}} \zeta(\mathcal Q^{\mathrm{sp}}) \sum_{\mathscr E^{\mathrm{sp}}} \tilde \epsilon_{\mathscr E^{\mathrm{sp}}} \int_{\pmb{\tau}} \int_{\mathcal{E}^{\mathrm{sp}}}  \prod_{\mathfrak n_0 \in \mathcal N^{\mathrm{sp}}_0} \widehat{P_{q_{\mathfrak n_0}}}(\tau_{\mathfrak n_0}, \tau_{\mathfrak n_0^2}, \tau_{\mathfrak n_0^3}, k[\mathcal N^{\mathrm{sp}}])    \\
    & \hspace{0.8cm}\times \prod_{\mathfrak n_1 \in \mathcal N^{\mathrm{l}}} \widehat{S_{q_{\mathfrak n_1}}}(\tau_{\mathfrak n_0}, \tau_{\mathfrak n_{q + 1}}, k[\mathcal N^{\mathrm{sp}}])\left( \prod_{\mathfrak n \in \mathcal N^{\mathrm{lf}}}e^{\zeta_{\mathfrak{n}}i\Gamma_{\mathfrak n}(t_{\mathfrak n}) + 2\pi i \tau_{\mathfrak n} t_{\mathfrak n}} \diff t_{\mathfrak{n}}\right) \\
    & \hspace{0.8cm}\times
    \prod^+_{\mathfrak l \in \mathcal L^{\mathrm{sp}}} n_{\mathrm{in}}(k_{\mathfrak l})\E \left[\mathcal{B}_{\mathcal{T}^+}(\eta_{\mathcal{T}^+}(\varrho))\mathcal{B}^*_{\mathcal{T}^-}(\eta_{\mathcal{T}^-}(\varrho))\right],
\end{split}
\end{align*}
where $\pmb{\tau} = (\tau_\mathfrak n)$ for any $\tau_{\mathfrak n} \in \R$ appearing in an expression for $\widehat{P_{q_{\mathfrak n_0}}}$ or $\widehat{S_{q_{\mathfrak n_1}}}$. If $\tau_{\mathfrak n}$ does not appear in such an expressions, we set $\tau_{\mathfrak n} = 0$. Therefore, by Lemmas \ref{lem-splice} and \ref{lem-loop}, 
\begin{align*}
    |\eqref{eq-splice-exp}| &\lesssim \left( \frac{\beta T}{N}\right)^{n_{\mathrm{sp}}} \left(\beta^2 T N^{2\delta}\right)^{n_0} \sum_{\mathscr E^{\mathrm{sp}}} \tilde \epsilon_{\mathscr E^{\mathrm{sp}}}  \left(\prod_{\mathfrak n_1 \in \mathcal N^{\mathrm{l}}} \frac{\log N}{\langle T \omega_{k_{\mathfrak n_1}}\rangle^{-q_{\mathfrak n_1}/2}}\right) \mathcal X(k[\mathcal N^{\mathrm{sp}}]) \\
    & \hspace{1.5cm} \times\sup_{\pmb{\tau}} \int_{\mathcal{E}^{\mathrm{sp}}} \left( \prod_{\mathfrak n \in \mathcal N^{\mathrm{lf}}}e^{\zeta_{\mathfrak{n}}i\Omega_{\mathfrak n}(Tt_{\mathfrak n} + A(t_{\mathfrak n})) + 2\pi i \tau_{\mathfrak n} t_{\mathfrak n}} \diff t_{\mathfrak{n}}\right),
\end{align*}
where $n_0$ represents the number of nodes removed via splicing so that $n_0 + n_{\mathrm{sp}} = n$ and $\mathcal X$ is a uniformly bounded function which encodes the dependence on $k[\mathcal N^{\mathrm{sp}}]$ of all terms not in the integral expression. Then, we may fix factors $\lambda[\mathcal N^{\mathrm{lf}}]$ so that for each $\mathfrak n \in \mathcal N^{\mathrm{lf}}$, $|T\Omega_{\mathfrak n} - \lambda_{\mathfrak n}| \leq 1$: 

\begin{align*}
    |\eqref{eq-splice-exp}| &\lesssim \left(\beta^2 T N^{2\delta}\right)^{n_0} \left(\sup_{\lambda[\mathcal N^{\mathrm{lf}}]}\sum_{\tilde{\mathscr E}_{\lambda[\mathcal N^{\mathrm{lf}}]}^{\mathrm{sp}}}\tilde \epsilon_{\tilde{\mathscr E}_{\lambda[\mathcal N^{\mathrm{lf}}]}^{\mathrm{sp}}}\left( \frac{\beta T}{N}\right)^{n_{\mathrm{sp}}} \left(\prod_{\mathfrak n_1 \in \mathcal N^{\mathrm{l}}} \frac{\log N}{\langle T \omega_{k_{\mathfrak n_1}}\rangle^{-q_{\mathfrak n_1}/2}}\right) \mathcal X(k[\mathcal N^{\mathrm{sp}}])\right) \\
    &\hspace{3cm}\times\left(\sum_{{\lambda[\mathcal N^{\mathrm{lf}}]}}\sup_{\pmb \tau}\left|\int_{\mathcal{E}^{\mathrm{sp}}_{\pmb{\tau}}}\prod_{\mathfrak{n} \in \mathcal{N}^{\mathrm{lf}}}e^{\zeta_{\mathfrak{n}}i\lambda_{\mathfrak{n}}(t_{\mathfrak{n}}+\frac{1}{T}A(t_{\mathfrak{n}})) + 2\pi i \tau_\mathfrak n t_{\mathfrak n}} \diff t_{\mathfrak{n}}\right|\right),
\end{align*}
where $\lambda[\mathcal N^{\mathrm{sp}}] \in \Z^{|\mathcal N^{\mathrm{sp}}|}$ and ${\tilde{\mathscr E}_{\lambda[\mathcal N^{\mathrm{sp}}]}^{\mathrm{sp}}}$ denotes decorations of the couple satisfying $|T\Omega_{\mathfrak n} - \lambda_{\mathfrak n}| \leq 1$ for each $\mathfrak n \in \mathcal N^{\mathrm{sp}}$. Note that we only need to consider $|\lambda_{\mathfrak n}| \lesssim T$ as $|\Omega_{\mathfrak n}| \lesssim 1$ due to the dispersion relation. 

For the second term, we may apply Proposition \ref{prop-integral_est} with $f_\mathfrak n \equiv 1$ to get 
\begin{align*}
&\left(\sum_{{\lambda[\mathcal N^{\mathrm{lf}}]}}\sup_{\pmb \tau}\left|\int_{\mathcal{E}^{\mathrm{sp}}_{\pmb{\tau}}}\prod_{\mathfrak{n} \in \mathcal{N}^{\mathrm{lf}}}e^{\zeta_{\mathfrak{n}}i\lambda_{\mathfrak{n}}(t_{\mathfrak{n}}+\frac{1}{T}A(t_{\mathfrak{n}})) + 2\pi i \tau_\mathfrak n t_{\mathfrak n}}  \diff t_{\mathfrak{n}}\right|\right) \\
&\hspace{1.5cm}\lesssim \sum_{\lambda[\mathcal N^{\mathrm{lf}}]} \sum_{d_{\mathfrak n} \in \{0,1\}}\frac{1}{\langle \lambda_{\mathfrak n} + {\tilde q}_{\mathfrak n}\rangle} \lesssim C^{n_{\mathrm{sp}}} (\log N)^{n_{\mathrm{sp}}}
\end{align*}

For the first term, we use that at any branching node node $\mathfrak n$ with a self loop, $\epsilon_{\mathfrak n} \sim \omega_{k_{\mathfrak n}} \omega_{k_{\mathfrak n_1}}$, where $\mathfrak n_1$ is a child of $\mathfrak n$ which is paired to another child of $\mathfrak n$. Therefore, for any restricted decoration $\mathscr E$, 
\begin{align*}
\tilde \epsilon_{\tilde{\mathscr E}} \left(\prod_{\mathfrak n_1 \in \mathcal N^{\mathrm{l}}} \frac{\log N}{\langle T \omega_{k_{\mathfrak n_1}}\rangle^{-q_{\mathfrak n_1}/2}}\right) \lesssim (\log N)^{n_{\mathrm{sp}}} T^{-L/2},
\end{align*}
where $L$ denotes the total number of self-loops in $\Q^{\mathrm{sp}}$. Therefore, applying Proposition \ref{prop-rigidity},
\begin{align*}
|\eqref{eq-splice-exp}| & \leq C_1 (\log N)^{2n_{\mathrm{sp}} + 2} \beta^2 T N^{-\delta(n-2)} \lesssim \beta^2 T N^{-\delta(n - \frac{5}{2})}.
\end{align*}

To establish \eqref{E2}, note that for couple $\Q = \{\T^+, \T^-, \mathscr P\}$, $\dot{\K}_\Q(t,t,k)$ corresponds to removing the integration in $t$ at the root of $\T^+$ or $\T-$ and therefore cannot fix a factor $\lambda$ as above to get a gain. Therefore, we do not and in the proofs of Propositions \ref{prop-cutting-est} and \ref{prop-counting}, instead of using the estimate $\mathfrak C = N^2T^{-1}$ coming from Proposition \ref{prop-3vc} to count the corresponding root, we may instead use $\mathfrak C = N^2$. Therefore, we may conclude 
\begin{align*}
\left|\sum_{\Q}\dot{\K}_\Q(t,t,k)\right|&\leq C_2 \beta^2 T^2 (N^{-\delta})^{n - \frac{5}{2}}.
\end{align*}
Take $C = \max(C_1, C_2)$ and we complete the proof.
\end{proof}

\section{The Remainder}
\label{section-remainder}
In this section, we establish Proposition \ref{prop-remainder} and use it to bound the remainder $\mathcal R$. Recall that the remainder $\mathcal R$ as well as the linear operator $\mathscr L$ are defined in Proposition \ref{strucRem}.

\subsection{Bound on operator \texorpdfstring{$\Ell$}{L}} In order to introduce the kernels of \(\Ell\), we recall from \cite{WKE} the following extension of trees and couples:

\begin{definition}\label{def-flower}
A \emph{flower tree} is a tree \(\mathcal{T}\) in which one leaf \(\mathfrak{f}\) is distinguished and called the \emph{flower}. Note that a different choice of $\mathfrak f$ for the same tree $\T$ yields a different flower tree. The unique path connecting the root \(\mathfrak{r}\) and the flower \(\mathfrak{f}\) is referred to as the \emph{stem}. A \emph{flower couple} is an enhanced couple formed from two flower trees such that the two flowers are paired. We refer to the \emph{scale}, or order, of a flower couple as $m$.

The \emph{height} of a flower tree \(\mathcal{T}\) is defined as the number of branching nodes along its stem. Evidently, to construct a flower tree of height \(n\), one can start with a single node and successively attach two sub-trees \(n\) times. A flower tree is said to be \emph{admissible} if each of the attached sub-trees has scale at most \(M\).
\end{definition}

We adapt the following proposition from \cite[Proposition 11.2]{WKE} and \cite[Proposition 6.2]{ODW} to our $\beta$-FPUT model:

\begin{proposition} \label{prop-Ln}
Let $\mathscr L$ be defined as in (\ref{eq-linearop}). Note that $\mathscr L^{n}$ is an $\R$-linear operator for $n \geq 0$. Define its kernels $(\mathscr L^{n})_{k \ell}^\zeta(t,s)$ for $\zeta \in \{\pm\}$ by 
\begin{equation*}
(\mathscr L^{n} b)_k (t) = \sum_{\zeta \in \{\pm\}} \sum_{\ell} \int_{\R} (\mathscr L^{n})_{k \ell}^{\zeta} (t,s) b_{\ell}(s)^{\zeta}d s.
\end{equation*}
Then, for each $1 \leq n \leq L$ and $\zeta \in \{\pm\}$, we can decompose
\begin{equation}
(\mathscr L^{n})^\zeta_{k, \ell} = \sum_{n \leq m \leq M^3} (\mathscr L^{n})_{k, \ell}^{m, \zeta},
\end{equation}
such that for any $n \leq m \leq M^3$ and $k,\ell \in \Z_N\cap(0,1)$ and $t,s \in (0,1)$ with $t > s$, we have 
\begin{equation} \label{eq-Lnm}
\E |(\mathscr L^{n})_{k, \ell}^{m, \zeta}(t,s)|^2 \lesssim  N^{-\delta m}N^{30}.
\end{equation}
\end{proposition}

\begin{proof}
We follow the proof in \cite[Proposition 11.2]{WKE} and \cite[Proposition 6.2]{ODW}, for flower trees $\T$ of height $n$ and scale $m$ such that the $\zeta_\mathfrak r=+$ and $\zeta_\mathfrak f=\zeta$. Then for $t \geq s$, we similarly define as in \eqref{eq-JT}:
    \begin{equation} \label{eq-Jtilde}
    \left(\widetilde \J_\T\right)_{k,\ell}(t,s) = \left(\frac{\beta T}{N}\right)^m \zeta(\mathcal{T}) \sum_{\mathscr D} \epsilon_{\mathscr D} \boldsymbol{\delta}(t_{\mathfrak f^p} - s) \mathcal{A}_{\mathcal{T}}(t,\Omega[\mathcal{N}], \widetilde{\Omega}[\Nc])\prod_{\mathfrak f \neq \mathfrak l \in \mathcal L} \sqrt{n_{\mathrm{in}}(k_{\mathfrak{l}})}\mathcal B_\T (\eta_{\mathcal{T}}(\varrho)) \boldsymbol{1}_{k_{\mathfrak f} = \ell},
    \end{equation}
     where $\mathscr D$ is a $k$-decoration of $\T$, and $\mathfrak f^p$ is the parent of $\mathfrak f$. Then for flower couples $\mathcal Q=(\mathcal T^+,\mathcal T^-)$, where $\mathcal T^{\pm}$ has sign $\pm \zeta$, height $n$ and scale $m$. Then for $t \geq s$, we can similarly define as in \eqref{eq-KQ}:
     \begin{align}\label{eq-KQtilde}
         (\widetilde{\mathcal K}_{\mathcal Q})(t,s,k, \ell)  &= \left( \frac{\beta T}{N}\right)^{2m} \zeta(\mathcal Q) \sum_{\mathscr E} \epsilon_{\mathscr E}\int_\mathcal{E} \prod_{\mathfrak n \in \mathcal N}e^{\zeta_{\mathfrak{n}}i\Gamma_ \mathfrak n(t_\mathfrak n)} \diff t_{\mathfrak{n}} \prod_{\mathfrak f} \boldsymbol{\delta}(t_{\mathfrak f^p} - s) \prod^+_{\mathfrak f \neq \mathfrak l \in \mathcal L}n_{\mathrm{in}}(k_{\mathfrak l})\boldsymbol{1}_{k_{\mathfrak f} = \ell}
     \end{align}
    where $\mathscr E$ is a $k$-decoration of $\Q$. Then, we may write:
     \begin{align}
    (\mathscr L^{n})_{k, \ell}^{m, \zeta}(t,s) &= \sum_{\mathcal T} \left(\widetilde \J _\T\right)_{k,\ell}(t,s) \\
    \E \left| (\mathscr L^{n})_{k, \ell}^{m, \zeta}(t,s)\right|^2 &= \sum_{\mathcal Q} \widetilde{\mathcal K}_{\mathcal Q}(t,s,k,\ell)
    \end{align}
    where the sum in $\T$ is taken over flower trees and the sum in $\Q$ is taken over flower couples. Note that while $\mathscr L$ removes degeneracies between siblings not along the stem (ie. their descendant leaves are not fully paired), the same is not true along the stem. However, since the two flowers must be paired, the couples $Q$ we may consider in the sum above are enhanced. So, in order to obtain \eqref{eq-Lnm}, we may use the proof of Proposition \ref{bound-couple} in addtion to the following observations: 
    \begin{enumerate}
        \item In \eqref{eq-KQtilde}, the Dirac delta factor \(\boldsymbol{\delta}(t_{\mathfrak f^p} - s)\) imposes \(t_{\mathfrak f^p} = s\), so we skip any phase integral with respect to \(t_{\mathfrak f^p}\) for each of the two flowers \(\mathfrak f\). Although this alters the counting, the resulting discrepancy can be bounded by at most \(N^{10}\).
        \item If $\Q$ is an admissible flower couple congruent to another couple $\widetilde \Q$ in the sense of Definition \ref{def-irrchain}, then $\widetilde \Q$ is an admissible flower couple, provided its flower is chosen as the image of the flower of $\Q$. Therefore, we may apply Lemma \ref{lem-splice}. The only exception is if $\mathfrak f^p$ is part of the irregular chain for either flower in the couple. Since we may avoid splicing these two nodes, this leads to at most two additional double bonds in the molecule, which may cause a loss of at most $N^2$ in the algorithm.
    \end{enumerate}
\end{proof}

Before proving Proposition \ref{prop-remainder}, we recall the following hypercontractivity estimate: 
\begin{lemma}{(Gaussian Hypercontractivity)} \label{lem-hypercontractivity}
Let $\{\eta_k\}$ be i.i.d. Gaussians or random phase. Let $\zeta_j \in \{\pm\}$ and $X$ be a random variable of the form 
\begin{equation}
X = \sum_{k_1, \ldots, k_n} a_{k_1, \ldots, k_n} \prod_{j = 1}^n \eta_{k_j}^{\zeta_j}(\varrho),
\end{equation}
where $a_{k_1, \ldots, k_n}$ are constants. Then, for $q \geq 2$,
\begin{equation}
\E \left|X\right|^q \leq (q-1)^{\frac{nq}{2}}\cdot \left( \E |X|^2\right)^{q/2}
\end{equation}
\end{lemma}

\begin{proof}[Proof of Proposition \ref{prop-remainder}]
After applying Cauchy-Schwarz inequality:
\begin{align*}
||\mathscr L^{n}(a)||_Z^2 &= \sup_{0 \leq t \leq 1} N^{-1} \sum_{k \in \Z_N\cap(0,1)} |(\mathscr L^{n} a)_k(t)|^2 \\
&= \sup_{0 \leq t \leq 1} N^{-1} \sum_{k \in \Z_N\cap(0,1)} \left| \sum_{\zeta \in \{\pm\}} \sum_{\ell \in \Z_N\cap(0,1)} \int\limits_{0 \leq s \leq t} \sum_{n \leq m \leq M^3} (\mathscr L^{n})_{k, \ell}^{m, \zeta}a_{\ell}^\zeta (s) ds  \right|^2 \\
&\lesssim \sup_{0 \leq s \leq t \leq 1} \sup_{\zeta} \sup_m N^{-1}M^3 \sum_{k \in \Z_N\cap(0,1)} \left| \sum_{\ell \in \Z_N\cap(0,1)} (\mathscr L^{n})_{k, \ell}^{m, \zeta}(t,s) a_\ell^\zeta(s) \right|^2 \\
& \lesssim ||a||_Z^2 M^3\sup_{0 \leq s \leq t \leq 1} \sup_\zeta \sup_m \sum_{k \in \Z_N\cap(0,1)}\sum_{\ell \in \Z_N\cap(0,1)}  \left| (\mathscr L^{n})_{k, \ell}^{m, \zeta}(t,s) \right|^2 \\ 
& \lesssim ||a||_Z^2 N^2M^3\sup_{0 \leq s \leq t \leq 1} \sup_\zeta \sup_m  \sup_{k, \ell} |(\mathscr L^{n})_{k, \ell}^{m, \zeta}(t,s)|^2.
\end{align*} 
For fixed $m$ and $\zeta$, let \[L(t,s) = \sup_{k, \ell} \sup_{0 \leq s \leq t \leq 1} |(\mathscr L^{n})_{k, \ell}^{m, \zeta}(t,s)|.\] As taking $\partial_t$ derivative just corresponds to omitting the $t_\mathfrak{r}$ integration,  by Proposition \ref{prop-Ln} and Lemma \ref{lem-hypercontractivity}:
\begin{align}\label{eq-LR-bound}
\E ||L(t,s)||_{L_{t,s}^p((0,1)^2) L^p_{k,\ell}(\Z_N^2\cap (0,1)^2)}^p& \lesssim p^{mp} N^{-\delta mp/2} N^{30p},\\
\E ||\partial_{(t,s)}L(t,s)||^p_{L_{t,s}^p((0,1)^2) L^p_{k,\ell}(\Z_N^2\cap (0,1)^2)} & \lesssim p^{mp} N^{-\delta mp/2}N^{31p}.
\end{align}
By using the Gagliardo-Nirenberg inequality for $(t,s)\in(0,1)^2$, and bounding the $L_{k,\ell}^\infty$ norm by the $L_{k,\ell}^p$ norm for $k,\ell \in \Z_N\cap(0,1)$ with an extra loss $N^{2/p}$, we conclude that
\begin{align}
    \E ||L(t,s)||_{L_{t,s}^\infty((0,1)^2) L^\infty_{k,\ell}(\Z_N^2\cap (0,1)^2)}^p& \lesssim p^{mp} N^{-\delta mp/2} N^{31p}.
\end{align}
Thus with probability $\geq 1-N^{-p/2}$, we have
\begin{align}
    L(t,s) \lesssim p^{m}N^{-\delta m/2} N^{35},
\end{align}
which implies 
\begin{equation} \label{remains}
\sup_{k, \ell} \sup_{0 \leq s \leq t \leq 1} |(\mathscr L^{n})_{k, \ell}^{m, \zeta}(t,s)| \lesssim N^{-\delta m/2} N^{35}.
\end{equation}
with probability $\geq 1-N^{-A}$ by taking $p \geq 2A$.
\end{proof}

\subsection{Bound on the remainder} Finally, we bound the remainder $\mathcal R$. Recall that assuming that the operator $1 - \mathscr L$ is invertible in a suitable space, the remainder (\ref{eq-remainder}) is equivalent to:
\begin{align}  
    \mathcal{R} = (1 - \Ell)^{-1}(\mathscr R + \mathscr Q(\mathcal{R},\mathcal{R}) + \mathscr C(\mathcal{R},\mathcal{R},\mathcal{R})), \label{eq-op}
\end{align}
with all relevant quantities defined in Proposition \ref{strucRem}. 

\begin{proposition} \label{prop-rem}
With probability $\geq 1 - N^{-A}$, the mapping defined by the right hand side of (\ref{eq-op}) is a contraction mapping from the set $\{\mathcal{R}: ||\mathcal{R}||_Z \leq N^{-500}\}$ to itself. 
\end{proposition}
\begin{proof}
By Proposition~\ref{prop-remainder}, there is an exceptional set of probability at most \(N^{-A}\) which we must exclude. Outside of this set, on another set of probability at least \(1 - N^{-A}\) we have
\begin{align}
\left\|\sum_{n(\T) = n}(\J_\T)(t)\right\|_{Z}\;&\lesssim\;N^{-\delta n/2}N^{50}. \label{eq-Jest-restate} \\
\|\mathscr{R}(t)\|_{Z}\;&\lesssim\;N^{-2M\delta}N^{50} \label{eq-Rest-restate}
\end{align}
We obtain the above estimates by applying Proposition~\ref{bound-couple} since both $\mathscr R$ and $\sum_{n(\T) = n} \J_\T$ are sums over enhanced trees of order at most $M^3$ which are preserved under congruence. Additionally, we may remove the expectation as in the proof of Proposition \ref{prop-remainder}.

Now, assume we are outside of all our exceptional sets with \(\|\mathcal{R}\|_Z \leq N^{-500}\). Recall that $M = 10^6/\epsilon,$ where $\delta < \epsilon$. Therefore, 
\[
  \|\mathscr{R}\|_Z \;+\;
  \|\mathscr{Q}(\mathcal{R},\mathcal{R})\|_Z \;+\;
  \|\mathscr{C}(\mathcal{R},\mathcal{R},\mathcal{R})\|_Z 
  \;\lesssim\; N^{-600}. 
\]
Note that the estimates for \(\|\mathscr{Q}(\mathcal{R},\mathcal{R})\|_Z\) and \(\|\mathscr{C}(\mathcal{R},\mathcal{R},\mathcal{R})\|_Z\) follow from
\begin{align*}
    \|\mathcal{IC^+}(u,v,w)\|_Z
    \;\lesssim\; N^{3}\,\|u\|_{Z}\,\|v\|_{Z}\,\|w\|_{Z},
\end{align*}
where we may bound at least two factors by $N^{-500}$ and any remaining factors using \eqref{eq-Jest-restate}.

Finally, note that \((1 - \mathscr{L})^{-1}\) maps \(Z\) to itself, since
\[
  (1 - \mathscr{L})^{-1} 
  \;=\; (1 - \mathscr L^{M})^{-1}\,\bigl(1 + \mathscr{L} + \cdots + \mathscr{L}^{M - 1}\bigr),
\]
and \(\|\mathscr L^{M}\|_{Z \to Z} < 1\) for sufficiently large $M$ so that we may invert \((1 - \mathscr L^{M})\).Therefore, via Neumann series, 
\[
  \|(1 - \mathscr{L})^{-1}\|_{Z \to Z} \;\lesssim\; N^{50}.
\]
\end{proof}

\section{Proof of the Main Theorem}
\label{section-mainthm}
We restate and prove a variation of results stated in \cite{AET, FPUTWP}, which allows us to make sense of the collision kernel:

\begin{lemma}[Making sense of the collision kernel] \label{lem-collision-kernel-convergence}
Let $\Psi \in L^1$ with $\int \Psi = 1$ and $|\Psi(x)| \lesssim \langle x \rangle^{-2}$ and $\phi \in C^\infty$ periodic. Then, 
\begin{align}
t&\int_{\substack{x,y,z \in \mathbb T \\ x + y -z = k (\text{mod 1})}} |T_{+,+,-}|^2 \phi_k\phi_{x} \phi_{y} \phi_{z} \left[\frac{1}{\phi_k} - \frac{1}{\phi_{x}} - \frac{1}{\phi_{y}} + \frac{1}{\phi_{z}} \right] \Psi(t\Omega(x,y,z,k) \diff x \diff y \diff z \label{eq-approx-collision}\\
&= \int_{\substack{x,y,z \in \mathbb T \\ x + y - z = k (\text{mod 1})}} \delta(\Omega(x,y,z,k))|T_{+,+,-}|^2 \phi_k\phi_{x} \phi_{y} \phi_{z} \left[\frac{1}{\phi_k} - \frac{1}{\phi_{x}} - \frac{1}{\phi_{y}} + \frac{1}{\phi_{z}} \right] \diff x \diff y \diff z + O(t^{-\frac{1}{4}}). \notag
\end{align}
\end{lemma}

\begin{proof}
As is established in \cite{AET}, one must separately consider two sets: 
\begin{align*}
U_+ &= \{(x,y,z) \in \mathbb T^3 \colon x + y - z = k\}, \\
U_-^\pm &= \{(x,y,z) \in \mathbb T^3 \colon x +y - z = k\pm 1\}.
\end{align*}
All solutions of $\Omega(x,y,z,k) = 0$ in $U_+$ satisfy $\{x,y\} = \{z,k\}$ (so the integrand of the collision kernel vanishes) while those in $U_-$ satisfy $y = h(x,z)$ or $y = 1 + h(x,z)$ for a function $h$ given explicitly in \cite{AET}. Note that for $(x,y,z)$ in each of the sets $U_+, U_-^+, U_-^-$, $z$ is uniquely determined by $x$ and $y$ and we may write $\Omega(x,y)$ without ambiguity. Correspondingly, we interpret \eqref{eq-approx-collision} as an integral in $x$ and $y$. On each of the sets $U_+, U_-^+, U_-^-$ we make sense of the $\delta$-function in $\Omega$ by making the change of variables $y \mapsto s = \Omega(x,y)$.

Beginning with $(x,y,z) \in U_+$, we set $z = x + y - k$, we may compute (via application of several trigonometric identities which we omit) 
\begin{align*}
\Omega(x,y) &= 4\sin\left(\frac{\pi(x + y)}{2}\right)\sin\left(\frac{\pi(k - x)}{2}\right) \sin\left(\frac{\pi(y - k)}{2}\right), \\
\partial_y \Omega(x,y) &= 2\pi \sin\left(\frac{\pi(y + z)}{2}\right) \sin \left(\frac{\pi(k - x)}{2}\right).
\end{align*}
We may assume due to symmetry that $|\sin(\pi(x - k))| \leq |\sin(\pi(y - k))|,$ so that 
\begin{equation}
\left|\Omega(x,y)\right| \gtrsim \sin\left(\frac{\pi(x+y)}{2}\right) \sin^2\left(\frac{\pi(x - k)}{2}\right).
\end{equation}
Since $\phi \in C^\infty$ is periodic, we also have that
\begin{align}
\left|\phi_k\phi_{x} \phi_{y} \phi_{z} \left[\frac{1}{\phi_k} - \frac{1}{\phi_{x}} - \frac{1}{\phi_{y}} + \frac{1}{\phi_{z}} \right]\right| \lesssim \min(|k - x|, |k - x \pm 1|),
\end{align}
Therefore, we may conclude (again via trigonometric identities which we omit) that 
\begin{align}
\left|\frac{T_{+,+,-}^2}{\partial_y \Omega(x,y)}\right| &\lesssim \left|\frac{\sin(\pi k) \sin(\pi x) \sin(\pi y) \sin(\pi z)}{\sin\left(\frac{\pi(y + z)}{2}\right) \sin^{1/4}\left(\frac{\pi(x + y)}{2}\right) \sin^{1/2}\left(\frac{\pi(k - x)}{2}\right)}\right||\Omega(x,y)|^{1/4} \\
&\lesssim \frac{|s|^{1/4}}{\min(|k - x|, |k - x \pm 1|)^{1/2}}.
\end{align}
Since $t \int|s|^{1/4}\Psi(ts)\diff s = O(t^{-1/4})$, the contribution of $U_-$ to the collision kernel is lower order.

Next, we turn our attention to $U_-^-$ (by symmetry the case of $U_-^+$ is identical). In this case, we observe from \cite{AET} that $k - x \in [0,1]$ and $y \leq k - x$. Therefore, we may compute
\begin{align}
\Omega(x,y) &= \sin(\pi x) + \sin(\pi y) + \sin(\pi(x + y - k)) - \sin(\pi k), \\
\partial_y\Omega(x,y) &= 2\pi \cos\left(\frac{\pi(2y + x - k)}{2}\right) \cos \left(\frac{\pi(k - x)}{2}\right).
\end{align}
One can verify the following inequalities when $0 \leq y \leq k - x \leq 1$:
\begin{align*}
\frac{\sin(\pi k) \sin(\pi x)}{\cos \left(\frac{\pi(k - x)}{2}\right)} \leq \cos\left(\frac{\pi(k - x)}{2}\right), \text{ and }
\frac{\sin(\pi y)}{\cos\left(\frac{\pi(2y + x - k)}{2}\right)} \leq 2 \sin \left(\frac{\pi(k - x)}{2}\right),
\end{align*}
so that 
\begin{align}
\left|\frac{T_{+,+,-}^2}{\partial_y \Omega(x,y)}\right| &\lesssim 1.
\end{align}
Therefore, we may use the following: 
\begin{equation} \label{eq-larget}
\left| t\int \Psi(ts) f(s) \diff s - f(0)\right| \lesssim C(f) t^{-\frac{1}{2}},
\end{equation}
where we use $f$ to denote the relevant integrand, involving the $\phi$-terms as well as $\frac{|T_{+,+,-}|^2}{\partial_y\Omega(x,y)}$ and $C(f)$ depends on the $L^\infty$ norms of $f$ and $f'$.
\end{proof}

\begin{proof}[Proof of Theorem \ref{main}]
Note that for finite $N,$ the FPUT system is a system of ODEs with conserved mass. Therefore, almost surely, it has a global solution for which $|b_k(t)| \leq N$. Let $E$ denote the complement of the union of all exceptional sets arising from Propositions \ref{prop-remainder} and \ref{prop-rem} so that $\Pp(E) \geq 1 - N^{-A}$.

First we focus on $\E \left[|a_k(t)|^2 \1_{E}\right]$. Recall that $\E\left[|a_k(t)|^2\1_{E}\right] = \E\left[|b_k(s)|^2\1_{E}\right],$ where $s = \frac{t}{T}$, where we are taking the expectation over $E$, where a smooth solution exists. Therefore, we decompose
\begin{align}
\E\left[|b_k(s)|^2\1_{E}\right] &= \sum_{0 \leq n_1, n_2 \leq M} \E\left[ \left(\J_{n_1}\right)_k(s)\overline{\left(\J_{n_2}\right)_k(s)}\1_{E}\right] \label{eq-bkexp}\\
& + 2 \mathrm{Re} \sum_{0 \leq n \leq M} \E\left[ \left(\J_{n}\right)_k(s) \overline{\mathcal R_k(s)} \1_{E} \right] + \E\left[|\mathcal R_k(s)|^2 \1_{E}\right]. \notag
\end{align}
Note that by the proof of Proposition \ref{prop-rem}, the last two terms are bounded by $N^{-100}$. Focusing on the first term, note that we can remove $\1_E$ since 
\begin{equation}
\left| \E \left( (\mathcal J_{n_1})_k(s)  \overline{(\mathcal J_{n_2})_k(s)} \1_{E^c}\right) \right|\leq \left(\E |(\mathcal J_{n_1})_k(s)|^4\right)^{1/4} \left(\E |(\mathcal J_{n_2})_k(s)|^4\right)^{1/4} \left( \Pp(E^c)^{1/2}\right) \lesssim N^{-A/2 + 10},
\end{equation}
where we may then use Lemma \ref{lem-hypercontractivity} and Proposition \ref{bound-couple} to get the desired bound. 

So, it suffices to consider
\begin{equation}
\sum_{0 \leq n_1, n_2 \leq M}\E \left[ (\mathcal J_{n_1})_k(s)  \overline{(\mathcal J_{n_2})_k(s)}\right] = \sum_\mathcal {Q}\mathcal{K_Q}(s,s,k),  
\end{equation}
for enhanced couples $\Q = \{\T_1, \T_2, \mathscr P\}$, where $n_1, n_2 \leq M$ (for $n_i := n(\T_i)$). When $n:= n_1 + n_2 \geq 3,$ we may apply Proposition \ref{bound-couple} to see that all such terms are $o_{\ell_k^\infty} \left(\frac{t}{T_{\mathrm{kin}}}\right)$. When $n = 0, $ the only contribution to $\E|b_k(s)|^2$ is $n_{\mathrm{in}}(k)$ and when $n = 1$, note that there are no enhanced couples so the contribution is 0. So, we are left with $n = 2$. Define 
\begin{equation} \label{eq-sum-kernel}
\mathcal S_k := \sum_{n(\T_1)=1}\E \left|(\J_{\T_1})_k(s)\right|^2 + 2\mathrm{Re} \left(\sum_{n(\T_0)=0,n(\T_2)=2}\E \left(  (\J_{\T_2})_k(s)\overline{(\J_{\T_0})_k(s)}\right)\right),
\end{equation}
which is the remaining term since there is only one tree with one branching node. 

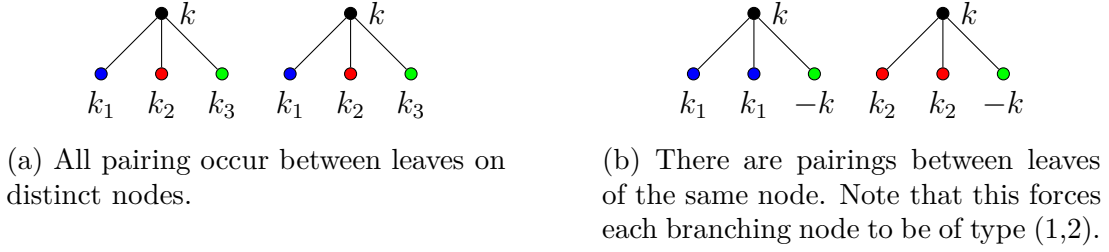
\begin{figure}
\centering
\begin{subfigure}[t]{.4\linewidth}
\centering
\begin{tikzpicture}[level distance=.8cm,
  level 1/.style={sibling distance=.8cm},
  level 2/.style={sibling distance=.8cm}]
\tikzstyle{hollow node}=[circle,draw,inner sep=1.6]
\tikzstyle{solid node}=[circle,draw,inner sep=1.6,fill=black]
\tikzset{
red node/.style = {circle,draw=black,fill=red,inner sep=1.6},
blue node/.style= {circle,draw = black, fill= blue,inner sep=1.6}, 
purple node/.style= {circle,draw = black, fill= purple,inner sep=1.6}, 
orange node/.style= {circle,draw = black, fill= orange,inner sep=1.6},
yellow node/.style= {circle,draw = black, fill= yellow,inner sep=1.6},
green node/.style = {circle,draw=black,fill=green,inner sep=1.6}}
\node[solid node, label = right:{$k$}] at (-5.5,.5){}
    child{node[blue node, label=below:{$k_1$}]{}}
    child{node[red node, label=below: $k_2$]{}}
    child{node[green node, label=below: $k_3$]{}}
;
\node[solid node, label = right: {$k$}] at (-3,.5){}
    child{node[blue node, label=below: $k_1$]{}}
    child{node[red node, label=below: $k_2$]{}}
    child{node[green node, label=below: $k_3$]{}}
;
\end{tikzpicture}
\caption{All pairing occur between leaves on distinct nodes.}
\label{fig-11couples}
\end{subfigure}
\hspace{1cm} %
\begin{subfigure}[t]{.4\linewidth}
\centering
\begin{tikzpicture}[level distance=.8cm,
  level 1/.style={sibling distance=.8cm},
  level 2/.style={sibling distance=.8cm}]
\tikzstyle{hollow node}=[circle,draw,inner sep=1.6]
\tikzstyle{solid node}=[circle,draw,inner sep=1.6,fill=black]
\tikzset{
red node/.style = {circle,draw=black,fill=red,inner sep=1.6},
blue node/.style= {circle,draw = black, fill= blue,inner sep=1.6}, 
purple node/.style= {circle,draw = black, fill= purple,inner sep=1.6}, 
orange node/.style= {circle,draw = black, fill= orange,inner sep=1.6},
yellow node/.style= {circle,draw = black, fill= yellow,inner sep=1.6},
green node/.style = {circle,draw=black,fill=green,inner sep=1.6}}
\node[solid node, label = right:{$k$}] at (-5.5,.5){}
    child{node[blue node, label=below:{$k_1$}]{}}
    child{node[blue node, label=below: $k_1$]{}}
    child{node[green node, label=below: $-k$]{}}
;
\node[solid node, label = right: {$k$}] at (-3,.5){}
    child{node[red node, label=below: $k_2$]{}}
    child{node[red node, label=below: $k_2$]{}}
    child{node[green node, label=below: $-k$]{}}
;
\end{tikzpicture}
\caption{There are pairings between leaves of the same node. Note that this forces each branching node to be of type (1,2).}
\label{fig-11couples-deg}
\end{subfigure}
\caption{Couples of order 2 where both trees have order 1.}
\label{fig-iterates-compute}
\end{figure}

We first consider the first term of \eqref{eq-sum-kernel}. For this we note that for any $\Omega \neq 0,$ we may integrate by parts and use $A(0) = 0$ to obtain
\begin{align}
\int_0^se^{i\Omega(Ts_1+A(s_1))}\diff s_1 &=\frac{e^{i\Omega(Ts+A(s))}-1}{i\Omega T}-\int_0^s\frac{\dot{A}(s_1)}{T}e^{i\Omega(Ts_1+A(s_1))}\diff s_1. \label{eq-int1}
\end{align}

For the first term of \eqref{eq-int1}, we note that 
\begin{equation}
\left|\frac{e^{i\Omega(Ts + A(s))} -1}{i\Omega T}\right|^2 = 2 \mathrm{Re}\left[\frac{e^{i\Omega(Ts + A(s))} -1}{(i\Omega T)^2} \right]=\left|\frac{\sin\left(\Omega(Ts + A(s))/2\right)}{\Omega T/2}\right|^2.
\end{equation}

For the second term of \eqref{eq-int1}, note that we may apply Proposition \ref{prop-integral_est}, using that $|\dot{A}(s)| \lesssim \beta T, |\ddot A(s)| \lesssim \beta^3T^3N^{2\delta}$ due to Proposition \ref{bound-couple}, to conclude
\begin{align}
\left| \int_0^s\frac{\dot{A}(s_1)}{T}e^{i\Omega(Ts_1+A(s_1))}\diff s_1\right| \lesssim \frac{N^{-\gamma/4}}{\langle \Omega T \rangle}. \label{eq-Sk-lot}
\end{align}
Therefore, for any signs $\zeta_1, \zeta_2, \zeta_3 \in \{\pm\}$, the corresponding contribution to $\mathcal S_k$ is bounded by 
\begin{align*}
    \left( \frac{\beta T}{N}\right)^2 \sum_{\Omega(\vec k) \neq 0} |{T}_{\zeta_1, \zeta_2, \zeta_3}|^2 \frac{N^{-\gamma/4}}{\langle \Omega(\vec k) T \rangle^2} \lesssim \left( \frac{\beta T}{N}\right)^2 N^{-\gamma/10} N^2T^{-1}\lesssim \frac{Ts}{T_{\mathrm{kin}}}N^{-\delta},
\end{align*}
where we can account for small $s$ by rescaling $T$.

Considering Figure \ref{fig-iterates-compute}, we obtain 
\begin{align}
&\sum_{n(\T_1)=1}\E \left|(\J_{\T_1})_k(s)\right|^2 
= \left( \frac{\beta T}{N}\right)^2 \sum_{\substack{\zeta_1, \zeta_2, \zeta_3 \in \{\pm\}}}C_{\zeta_1, \zeta_2, \zeta_3}\Bigg[ \sum_{\Omega(\vec k) = 0}|{T}_{\zeta_1, \zeta_2, \zeta_3}|^2\phi_{k_1} \phi_{k_2} \phi_{k_3} \notag\\
&\hspace{0.75cm} +\sum_{\Omega(\vec k) \neq 0} |{T}_{\zeta_1, \zeta_2, \zeta_3}|^2\phi_{k_1} \phi_{k_2} \phi_{k_3}\left|\frac{\sin(\Omega(\vec{k})(Ts+A(s))/2)}{\Omega(\vec{k})T/2}\right|^2 \Bigg]+ \K_\Q^{\mathrm{deg}}(s,s,k) + O\left(\frac{Ts}{T_{\mathrm{kin}}}N^{-\delta}\right),\label{eq-first-iterate1}
\end{align}
where $\phi_{k_i} := n_{\mathrm{in}}\left(k_i\right)$ and $\vec k =(k,k_1, k_2, k_3)$. In the above, $C_{\zeta_1, \zeta_2, \zeta_3}$ denotes the multiplicity of couples of a given type and $\K_\Q^\mathrm{deg}$ represents the contribution of all couples corresponding to Figure \ref{fig-11couples-deg}. For the first term with $\sum_{\Omega(\vec k) = 0}$, we may bound this term by $O\left(\frac{Ts}{T_{\mathrm{kin}}}N^{-\delta}\right)$ using a tighter version of Proposition \ref{prop-3vc} where we take $\lambda = 0$ and replace $T^{-1}$ with $(TN^{2\delta})^{-1}$ (note $TN^{2\delta} \ll 1$) to get a bound of $N^{2 - \delta}T^{-1}$ (or $N$ if two leaves which are not paired have the same decoration). Also, to bound $\K_\Q^{\mathrm{deg}}$, we may use the proof of Lemma \ref{lem-loop}:
\begin{equation} \label{eq-deg-bound}
\left| \K_{\Q^{\mathrm{deg}}}(s,s,k)\right| \lesssim \left(\frac{\beta T}{N}\right)^2 N^2 \frac{\omega_k^2}{\langle T\omega_k \rangle^2} \lesssim \beta^2 \ll \frac{Ts}{T_{\mathrm{kin}}}N^{-\delta}. 
\end{equation}

\begin{figure}
\centering
\begin{subfigure}[t]{.35\linewidth}
\centering
\begin{tikzpicture}[level distance=.8cm,
  level 1/.style={sibling distance=.8cm},
  level 2/.style={sibling distance=.8cm}]
\tikzstyle{hollow node}=[circle,draw,inner sep=1.6]
\tikzstyle{solid node}=[circle,draw,inner sep=1.6,fill=black]
\tikzset{
red node/.style = {circle,draw=black,fill=red,inner sep=1.6},
blue node/.style= {circle,draw = black, fill= blue,inner sep=1.6}, 
purple node/.style= {circle,draw = black, fill= purple,inner sep=1.6}, 
orange node/.style= {circle,draw = black, fill= orange,inner sep=1.6},
yellow node/.style= {circle,draw = black, fill= yellow,inner sep=1.6},
green node/.style = {circle,draw=black,fill=green,inner sep=1.6}}
\node[solid node, label = right: $k$ ] at (-5,0){}
    child{node[solid node, label = left: $k_1$ ]{}
        child{node[red node, label=below: $k_2$]{}}
        child{node[green node, label=below: $k_3$]{}}
        child{node[blue node, label=below: $k$]{}}
    }
    child{node[red node, label=below: $k_2$]{}}
    child{node[green node, label=below: $k_3$]{}}
;
\node[blue node, label=right: $k$] at (-3.5,0){};
\end{tikzpicture}
\caption{All pairing occur between leaves on distinct nodes. Note that the top branching node may have any type which will force a type for the remaining one.}
\label{fig-20couples}
\end{subfigure}
\hspace{.25cm}
\begin{subfigure}[t]{.55\linewidth}
\centering
\begin{tikzpicture}[level distance=.8cm,
  level 1/.style={sibling distance=.8cm},
  level 2/.style={sibling distance=.8cm}]
\tikzstyle{hollow node}=[circle,draw,inner sep=1.6]
\tikzstyle{solid node}=[circle,draw,inner sep=1.6,fill=black]
\tikzset{
red node/.style = {circle,draw=black,fill=red,inner sep=1.6},
blue node/.style= {circle,draw = black, fill= blue,inner sep=1.6}, 
purple node/.style= {circle,draw = black, fill= purple,inner sep=1.6}, 
orange node/.style= {circle,draw = black, fill= orange,inner sep=1.6},
yellow node/.style= {circle,draw = black, fill= yellow,inner sep=1.6},
green node/.style = {circle,draw=black,fill=green,inner sep=1.6}}
\node[solid node, label = right: $k$ ] at (-5,0){}
    child{node[solid node, label = left: $k_1$ ]{}
        child{node[green node, label=below: $k_3$]{}}
        child{node[green node, label=below: $k_3$]{}}
        child{node[blue node, label=below: $k$]{}}
    }
    child{node[red node, label=below: $k_2$]{}}
    child{node[red node, label=below: $k_2$]{}}
;
\node[blue node, label=right: $k$] at (-3.5,0){};

\node[solid node, label = right: $k$ ] at (-1,0){}
    child{node[solid node, label = left: $k_1$ ]{}
        child{node[green node, label=below: $k_3$]{}}
        child{node[green node, label=below: $k_3$]{}}
        child{node[blue node, label=below: $k_2$]{}}
    }
    child{node[blue node, label=below: $k_2$]{}}
    child{node[red node, label=below: $k$]{}}
;
\node[red node, label=right: $k$] at (.5,0){};

\end{tikzpicture}
\caption{There are pairings between leaves of the same node. Note that this forces any branching node with two leaves paired to be (1,2).}
\label{fig-20couples-deg}
\end{subfigure}
\caption{Couples of order 2 where one tree is order 2}
\label{fig-iterates-2}
\end{figure}

Next, we consider the second term of \eqref{eq-sum-kernel}. For this we note that for any $\Omega \neq 0,$ we may compute similarly to \eqref{eq-int1} to obtain
\begin{align}
\int_0^s\int_0^{s_1} e^{i\Omega(Ts_1 + A(s_1))}&e^{-i\Omega(Ts_2 + A(s_2))} \diff s_2 \diff s_1 \notag\\
&= -\frac{s}{i\Omega T} + \frac{e^{i \Omega(Ts + A(s)) }-1}{(i \Omega T)^2} -\int_0^s \frac{\dot A(s_1)}{i\Omega T^2} e^{i \Omega(Ts_1 + A(s_1))} \diff s_1 \label{eq-int2}\\
& \hspace{.5cm} + \int_0^s\int_0^{s_1} \frac{\dot A(s_2)}{T}e^{i\Omega(Ts_1 + A(s_1))}e^{-i\Omega(Ts_2 + A(s_2)} \diff s_2 \diff s_1. \notag
\end{align}
As in the case of two order 1 trees paired, we would like to ignore the contributions of the last two terms. For the third term, we may use \eqref{eq-Sk-lot} while for the fourth term, we can apply Proposition \ref{prop-integral_est} to obtain the same bound. Therefore, the corresponding contribution to $\mathcal S_k$ for any $\zeta_1, \zeta_2, \zeta_3 \in \{\pm\}$ dictating a type for the top branching node, we have the bound 
\begin{align*}
    \left( \frac{\beta T}{N}\right)^2 \sum_{\Omega(\vec k) \neq 0} |{T}_{\zeta_1, \zeta_2, \zeta_3}|^2 \frac{N^{-\gamma/4}}{\langle \Omega(\vec k) T \rangle} \lesssim \left( \frac{\beta T}{N}\right)^2 N^{-\gamma/10} N^2T^{-1}\lesssim \frac{Ts}{T_{\mathrm{kin}}}N^{-\delta},
\end{align*}
where we use the fact that $|\Omega T| \lesssim T$ due to the dispersion relation to lose at most $\log T$ when we fix the factor $\Omega T$ for counting purposes.

Considering Figure \ref{fig-iterates-2}, we may sum over couples where the top node has children with sign $\zeta_1, \zeta_2, \zeta_3$ to obtain
\begin{align}
2\mathrm{Re} &\left(\sum_{n(\T_0)=0,n(\T_2)=2}\E \left(  (\J_{\T_2})_k(s)\overline{(\J_{\T_0})_k(s)}\right)\right) \notag\\
& = \left( \frac{\beta T}{N}\right)^2 \sum_{\substack{\zeta_1, \zeta_2, \zeta_3 \in \{\pm\}}}2C_{\zeta_1, \zeta_2, \zeta_3}\Bigg[ \sum_{\Omega(\vec k) = 0}|{T}_{\zeta_1, \zeta_2, \zeta_3}|^2\left(-\zeta_1 \phi_k \phi_{k_2} \phi_{k_3} -\zeta_2 \phi_k \phi_{k_1} \phi_{k_2} -\zeta_3 \phi_k \phi_{k_1} \phi_{k_2}\right)\notag \\
&\hspace{.5cm} +\sum_{\Omega(\vec k) \neq 0} |{T}_{\zeta_1, \zeta_2, \zeta_3}|^2 \mathrm{Re} \left[-\frac{s}{i \Omega(\vec k) T} +\frac{e^{i\Omega(\vec{k})(Ts+A(s))}-1}{(i\Omega(\vec{k})T)^2}\right] \label{eq-first-iterate2} \\
&\hspace{2cm} \times\left(-\zeta_1 \phi_k \phi_{k_2} \phi_{k_3} -\zeta_2 \phi_k \phi_{k_1} \phi_{k_2} -\zeta_3 \phi_k \phi_{k_1} \phi_{k_2}\right) \Bigg] + \K_{\Q_1}^{\mathrm{deg}}(s,s,k) + \K_{\Q_2}^{\mathrm{deg}}(s,s,k), \notag
\end{align}
where $\K_{\Q_1}^{\mathrm{deg}}$ represents all contributions from couples corresponding to the first couple in Figure \ref{fig-20couples-deg} and $\K_{\Q_2}^{\mathrm{deg}}$ similarly for the second couple. Note that checking multiplicity of couples, one can verify that for any signs $\zeta_1, \zeta_2, \zeta_3$, the coefficient is $2C_{\zeta_1, \zeta_2, \zeta_3},$ where $C_{\zeta_1, \zeta_2, \zeta_3}$ is defined in \eqref{eq-first-iterate1}.

The first term, as in \eqref{eq-first-iterate1}, can be bounded by $O\left(\frac{Ts}{T_{\mathrm{kin}}}N^{-\delta}\right)$. For the second term, we note that $\mathrm{Re}(-s/i\Omega(\vec k)T) = 0$. To bound the terms $\K_{\Q_i}^{\mathrm{deg}}(s,s,k)$, we are not able to use Lemma \ref{lem-loop} as we could for $\K_\Q^\mathrm{deg}$, however we can directly compute them using \eqref{eq-int2}. Similarly to the couples corresponding to Figure \ref{fig-20couples} with $\Omega \neq 0,$ we may discount the contribution coming from the first and last two terms of \eqref{eq-int2}. Therefore, we may recover the estimate \eqref{eq-deg-bound}, noting that for $\Q_2^\mathrm{deg}$, we instead have 
\begin{align}
\K_{\Q_2}^{\mathrm{deg}}(s,s,k) \lesssim \left(\frac{\beta T}{N}\right)^2 \sum_{k_2, k_3 \in \Z_N \cap [0,1)]} \frac{\omega_{k_2}^2}{\langle T\omega_{k_2} \rangle^2} \lesssim \beta^2 \ll \frac{Ts}{T_{\mathrm{kin}}}N^{-\delta}. \label{eq-deg-bound2}
\end{align}

So, summing \eqref{eq-first-iterate1} and \eqref{eq-first-iterate2}, we get 

\begin{align*}\label{eq-K2}
\mathcal S_k &= \frac{\beta^2 T^2 s^2}{N^2} \sum_{\substack{\zeta_1, \zeta_2, \zeta_3 \in \{\pm\}}} \Bigg\{\sum_{ \Omega(\vec k) \neq 0}^{\times} C_{\zeta_1, \zeta_2, \zeta_3}|{T}_{\zeta_1, \zeta_2, \zeta_3}|^2\phi_k \phi_{k_1} \phi_{k_2} \phi_{k_3} \left[\frac{1}{\phi_k} - \frac{\zeta_1}{\phi_{k_1}} - \frac{\zeta_2}{\phi_{k_2}} - \frac{\zeta_3}{\phi_{k_3}} \right]  \notag\\
& \hspace{4cm}\times\left|\frac{\sin(\Omega(\vec{k})(Ts+A(s))/2)}{\Omega(\vec{k})Ts/2}\right|^2 \Bigg\} +O\left( \frac{Ts}{T_{\mathrm{kin}}}N^{-\delta}\right). 
\end{align*} 

We note that for terms corresponding to the non-resonant portion of the nonlinearity, we may use the estimates in Proposition 1 of \cite{VCFPUT}, namely that for $\{\zeta_1, \zeta_2, \zeta_3\} \neq\{+,+,-\}$, 
\begin{equation} \label{eq-normalform-est}
\sum_{\substack{k_1, k_2, k_3 \in \Z_N \cap[0,1) \\ \zeta_1k_1 + \zeta_2 k_2 + \zeta_3 k_3 = k (\text{mod } 1)}} \frac{\left|T_{\zeta_1, \zeta_2, \zeta_3}\right|^2}{(T\Omega(\vec k))^2} \lesssim N^2T^{-2} \log N, 
\end{equation}
so that the contribution to the sum $\mathcal S_k$ is $O\left(\frac{Ts}{T_{\mathrm{kin}}}N^{-\delta}\right)$. Therefore, by Corollary \ref{cor-iterates} and \ref{lem-collision-kernel-convergence}, we have that 
\begin{align*}
\mathcal S_k &= \beta^2 t \cdot \frac{t+A(t/T)}{t} \cdot (t+A(t/T))\int_{\substack{x,y,z \in \mathbb T \\ x + y - z = k (\text{mod 1})}} 9|{T}_{+,+,-}|^2 \\
&\hspace{0.5cm}\times \phi_k\phi_{x} \phi_{y} \phi_{z} \left[\frac{1}{\phi_k} - \frac{1}{\phi_{x}} - \frac{1}{\phi_{y}} + \frac{1}{\phi_{z}} \right]\left| \frac{\sin(\Omega(\vec{k})(t+A(t/T))/2)}{\Omega(\vec{k})(t+A(t/T))/2}\right|^2\diff x \diff y \diff z +O\left( \frac{t}{T_{\mathrm{kin}}}N^{-\delta}\right) \\
&=\frac{t}{T_{\mathrm{kin}}} \K(n_\mathrm{in})(k)+o_{\ell_k^{\infty}}\left( \frac{t}{T_{\mathrm{kin}}}\right),
\end{align*}
which concludes the derivation of \eqref{WKE}. 

Now, we turn our attention to $\E \left[a_k(t) a_{1-k}(t) \1_{E}\right]$. Setting $s = \frac{t}{T}$, we may write 
\begin{align*}
\E \left[a_k(t) a_{1-k}(t) \1_{E_1} \right] = e^{-is\left(T(\omega_k + \omega_{-k}) + \widetilde \omega_k(s) + \widetilde \omega_{-k}(s)\right)} \E [b_k(s) b_{1-k}(s) \1_{E}]. 
\end{align*}
Note that since the phase renormalization $\widetilde \omega_k$ is deterministic, we may ignore this initial oscillatory factor which is order 1 and write $\E [b_k(s) b_{1-k}(s) \1_{E}]$ similarly to \eqref{eq-bkexp}. By considering a generalization of enhanced couples where both roots have sign $+$ and one is decorated with $k$ and the other $1-k,$ we may similarly reduce ourselves to the case of couples order $n = 0, 1, 2$. Note that if $n = 0,$ we may use that $\E \eta_k \eta_{1-k} = 0$. For $n = 1,$ one can verify that any corresponding couple $\K_\Q$ is purely imaginary, so that $2 \mathrm{Re} \K_\Q = 0$, see \cite{2019}. So, we are left to consider couples of order 2:
\begin{equation} \label{eq-sum-kernel-vanish}
\mathcal S_k := \sum_{n(\T_1)=1, n(\T_2) = 1}\E \left[(\J_{\T_1})_k(s) (\J_{\T_2})_{-k}(s)\right] + 2\mathrm{Re} \left(\sum_{n(\T_0)=0,n(\T_2)=2}\E \left(  (\J_{\T_2})_k(s)\overline{(\J_{\T_0})_k(s)}\right)\right),
\end{equation}

For the first term of \eqref{eq-sum-kernel-vanish}, note that we only need to consider couples as in Figure \ref{fig-11couples} with both trees having roots with sign + and the second tree instead decorated with $-k$ rather than $k$. Note that this forces at least one of the trees to contain a non-resonant branching node. Using \eqref{eq-int1}, if $\T_1$ corresponds to $\Omega_1$ and $T_2$ corresponds to $\Omega_2$, we must consider terms of the form
\begin{align*}
\omega_k \omega_{k_1} \omega_{k_2} \omega_{k_3}\frac{e^{i\Omega_1(Ts + A(s))} - 1}{i \Omega_1 T}\cdot \frac{e^{i\Omega_2(Ts + A(s))} - 1}{i \Omega_2 T}.
\end{align*}
Using \eqref{eq-normalform-est}, since there must be at least one non-resonant branching node,
\begin{align}
\left(\frac{\beta T}{N}\right)^2&\sum_{k_1, k_2, k_3 \in \Z_N \cap [0,1)} \left|\frac{\omega_k \omega_{k_1} \omega_{k_2} \omega_{k_3}}{T^2\Omega_1 \Omega_2} \right| \notag \\
&\lesssim \left(\frac{\beta T}{N}\right)^2\left(\sum_{k_1, k_2, k_3 \in \Z_N \cap [0,1)} \frac{\omega_k \omega_{k_1} \omega_{k_2} \omega_{k_3}}{(T \Omega_1)^2} \right)^{1/2} \left(\sum_{k_1, k_2, k_3 \in \Z_N \cap [0,1)} \frac{\omega_k \omega_{k_1} \omega_{k_2} \omega_{k_3}}{(T \Omega_2)^2} \right)^{1/2} \label{eq-normalform} \\
&\lesssim \left(\frac{\beta T}{N}\right)^2 (NT^{-1/2} \log T) (N \log NT^{-1}) \lesssim \beta^2 T^{1/2} (\log N)^2. \notag
\end{align}
Therefore, we may conclude that 
\begin{align*}
\sum_{n(\T_1)=1, n(\T_2) = 1}\E \left[(\J_{\T_1})_k(s) (\J_{\T_2})_{-k}(s)\right] \lesssim \frac{t}{T_{\mathrm{kin}}} N^{-\delta}.
\end{align*}

For the second term of \eqref{eq-sum-kernel-vanish}, we must consider couples as in Figure \ref{fig-20couples} and \ref{fig-20couples-deg}, where in both cases the trivial tree has sign + and is decorated with $-k$ instead. Those in Figure \ref{fig-20couples-deg} we may deal with similarly to \eqref{eq-deg-bound} and \eqref{eq-deg-bound2}. For couples as in Figure \ref{fig-20couples}, we use a variant of \eqref{eq-int2}, where now we have$\Omega_1 + \Omega_2 = 2\omega_k$ rather than 0, so we have
\begin{align}
\int_0^s\int_0^{s_1} e^{i\Omega_1(Ts_1 + A(s_1))}&e^{-i\Omega_2(Ts_2 + A(s_2))} \diff s_2 \diff s_1 \notag\\
&= -\frac{e^{i(\Omega_1 + \Omega_2)(Ts + A(s))} - 1}{\Omega_1(\Omega_1 + \Omega_2) T^2} + \frac{e^{i \Omega_1(Ts + A(s)) }-1}{\Omega_1 \Omega_2 T^2} \notag \\
&\hspace{.5cm}-\int_0^s \frac{\dot A(s_1)}{i\Omega_2 T^2} e^{i (\Omega_1 + \Omega_2)(Ts_1 + A(s_1))} \diff s_1 -\int_0^s \frac{\dot A(s_1)}{i\Omega_2 T^2} e^{i \Omega_1(Ts_1 + A(s_1))} \diff s_1 \label{eq-int3}\\
& \hspace{.5cm} + \int_0^s\int_0^{s_1} \frac{\dot A(s_2)}{T}e^{i\Omega(Ts_1 + A(s_1))}e^{-i\Omega(Ts_2 + A(s_2)} \diff s_2 \diff s_1. \notag
\end{align}
The last three terms of \eqref{eq-int3} we may similarly disregard as we did in \eqref{eq-int2}. For the second term, we may deal with it similarly to \eqref{eq-normalform}. For the first term, note that the corresponding contribution may be bounded by
\begin{align*}
\left(\frac{\beta T}{N}\right)^2 \sum_{k_1, k_2, k_3 \in \Z_N \cap [0,1)} \frac{\omega_{k_1} \omega_{k_2} \omega_{k_3}}{T} \frac{1}{\langle \Omega_1 T \rangle} \lesssim \beta^2 \log T \lesssim \frac{t}{T_{\mathrm{kin}}}N^{-\delta}. 
\end{align*}
\end{proof}

\section*{References}
\printbibliography[heading = none]

\end{document}